\documentclass[11pt]{article}
\usepackage{amsmath,amssymb,amscd}
\usepackage{amsfonts}
\usepackage{microtype}
\usepackage{soul}
\usepackage{tikz-cd}
\usepackage{pb-diagram,pb-xy}
\usepackage[pagewise]{lineno}
\usepackage{enumitem}
\usepackage{cancel}

\usepackage{hyperref}
\hypersetup{
    colorlinks=true,
    linkcolor=blue,
    filecolor=magenta,
    urlcolor=cyan,
}

\pdfoutput=1

\def\N{\mathbb{N}}

\def\R{\mathbb{R}}

\def\C{C^{\infty}(M, \R)}
\def\CN{C^{\infty}(N, \R)}
\newcommand{\Lie}{\boldsymbol{\pounds}}            

\newtheorem{definition}{Definition}[section]
\newtheorem{lemma}[definition]{Lemma}
\newtheorem{proposition}[definition]{Proposition}
\newtheorem{theorem}[definition]{Theorem}
\newtheorem{remark}[definition]{Remark}
\newtheorem{corol}[definition]{Corollary}
\newtheorem{remarks}[definition]{Remarks}
\newtheorem{examples}[definition]{Examples}
\newtheorem{example}[definition]{Example}

\newtheorem{some remarks}[definition]{Some remarks}
\newtheorem{some remarks a}[definition]{Some remarks on the rank of $\alpha$}
\newtheorem{example L-L}[definition]{The Lie algebroid of a Loday algebroid}
\newenvironment{proof}{\noindent{\bf Proof.}}{\hfill $\blacklozenge$}
\newenvironment{proof-Th}{\noindent{\bf Proof of the Theorem}}{\hfill $\blacklozenge$}

\DeclareMathOperator{\tr}{tr}

\def\lcf{\lbrack\! \lbrack}
\def\rcf{\rbrack\! \rbrack}

\def\lbr{\lbrace\!\!\!\lbrace}
\def\rbr{\rbrace\!\!\!\rbrace}

\begin{document}

\title{Loday algebroids cohomology, nonlinear connections and characteristic classes}

\vspace{5mm}

\author{Raquel Caseiro and Fani Petalidou}

\date{}
\maketitle

\vskip 20 mm

\begin{abstract}
\noindent
The purpose of this paper is to contribute to the further study of \emph{Loday algebroids} by introducing, first,
the concept of \emph{action Loday algebroid} and clarifying the notion of a \emph{(co)morphism of Loday algebroids}.
We then focus on the study of \emph{Loday algebroids nonlinear connections} and their related \emph{theory of characteristic classes}.
These classes arise from  multidifferential operators (on the module of smooth sections of the Loday algebroid) in the Loday algebroid cohomology.
In particular, the Chern-Simons differential operators for nonlinear connections on Loday algebroids are considered and a
generalized Chern-Simons formula for such connections is established.
Using representations of Loday algebroids, we define their \emph{secondary characteristic classes}.
\end{abstract}

\vspace{5mm} \noindent {\emph{Keywords: }} Loday algebroid; Action Loday algebroid; Loday algebroid (co)morphisms; Loday algebroid cohomology; Nonlinear connections of Loday algebroids; Characteristic classes of Loday algebroids.

\vspace{3mm} \noindent MSC (2020): 17A32, 53C05, 53D17, 57R20, 58J10

\microtypesetup{protrusion=false}
\tableofcontents
\microtypesetup{protrusion=true}

\section{Introduction}
The concept of \emph{Loday algebroid} discussed here was originally proposed by Grabowski,
Khudaverdyan and Poncin in \cite{grab-k-pon} in order to give a more general, appropriate and geometric definition to
the different concepts of \emph{Leibniz algebroids} that have appeared in the literature since the late '90s.
As explained in \cite{grab-k-pon}, since Leibniz brackets -- a non-commutative version of Lie brackets
on a vector space defined by Loday \cite{lod, lod-pir} (currently they called \emph{Loday brackets}, see Definition \ref{def-Loday algebra} below) --
arise in Geometry and Physics as brackets on
spaces of smooth sections of vector bundles (e.g. Courant algebroids \cite{li-bland-meinr}),
there were several attempts to formalize the concept of \emph{Leibniz algebroid}.
As far as we know, the first definition of this notion was given by Ib\'a\~{n}ez, de Le\'on,
Marrero and Padr\'on \cite{ILMP1999} in relation with Nambu-Poisson structures.
According their formulation, it is a structure on a vector bundle $E\to M$ (of constant rank)
that consists of a Leibniz bracket $\lcf \cdot, \cdot \rcf$ on the space $\Gamma(E)$ of smooth sections of $E$ and a vector bundle map $\rho: E \to TM$
whose induced map $\rho : \Gamma(E) \to \Gamma(TM)$ is a homomorphism of Leibniz algebras satisfying, for all $e_1, e_2 \in \Gamma(E)$ and $f\in \C$, the (left) Leibniz rule
\[ \lcf e_1,fe_2\rcf = f\lcf e_1,e_2\rcf + \rho(e_1)(f)e_2.\]
Thereafter, a ``bloom" of publications appeared on this subject and some of them in connection
with Physics and the study of dynamical systems with nonholonomic constraints (see \cite{grab-k-pon} and references therein).
A weak point of this definition is that the bracket does not have the locality property, i.e.
no relationship is defined between the brackets $\lcf fe_1,e_2\rcf$ and $\lcf e_1, e_2\rcf$. Hence,
as mentioned in \cite{jurco-Vysoky}, $\lcf e_1,e_2\rcf$ can depend on the values of the section
$e_1$ at every point of the manifold $M$. In this case, we cannot localize the bracket and study its
(local) properties by writing it in local frame components, as we can with Lie and Courant brackets on $\Gamma(E)$, which are
differential operators with respect to each argument. We note that the term \emph{Loday algebroid} also appears in \cite{stienon-xu} due to Sti\'enon and Xu.
This term concerns \emph{Leibniz algebroids} endowed with a fiberwise pseudometric on $E$; the resulting notion
is slightly more general than that of Courant algebroids, but does not reduce to a Loday algebra when $M$ is a single point. The purpose of \cite{grab-k-pon} was to suggest a new concept of Leibniz/Loday algebroid, imposing minimal requirements
on the elements that define this structure, so that the new notion
\begin{itemize}
\item[-]
reduces to a Loday algebra when $M$ shrinks to a point.
\item[-]
includes a locality condition on both arguments, i.e. in addition to the left anchor the structure has a right anchor
satisfying a condition of Leibniz rule type.
\item[-]
contains Lie and Courant algebroids as particular cases.
\end{itemize}
Furthermore, the authors of \cite{grab-k-pon}
associate with each Loday algebroid a coboundary operator $\partial_E$ on the graded space of multidifferential operators on $\Gamma(E)$,
thus defining a cochain complex -- and hence a Loday algebroid cohomology -- which allows Loday algebroids interpreted as homological vector fields on supercommutative manifolds, as in Lie \cite{vaintrob} and Courant \cite{roy} cases. The naive cohomology of a Loday algebroid, according to \cite{stienon-xu}, is determined by
a differential on the graded space $\Gamma(\bigwedge^{\bullet} \ker \rho)$ and is completely different from that studied in \cite{grab-k-pon}.

\vspace{1mm}
\noindent
The existence of a Loday algebroid cohomology theory and of Cartan calculus leads us, as an application of this theory,
to study \emph{nonlinear connections of a Loday algebroid $E$ on a vector bundle $B$}. We use the term \emph{nonlinear} in the sense that
the connection has the locality property with respect of both arguments. Just as the theory of linear (classical) connections undoubtedly plays a
fundamental role in Differential Geometry with important applications in Mathematical Physics (Yang-Mills theory, General Relativity, etc) \cite{mang-sard},
so too the theory of nonlinear connections has an increasing interesting in problems of modern Differential Geometry related to nonholonomic systems and higher order differential equations.
We introduce the notion of \emph{nonlinear $E$-connection on a vector bundle $B$} in such a way that the adjoint representation of $E$ makes sense
and we prove that a covariant derivation on the space of $B$-valued multidifferential operators
on the $\C$-module of smooth sections of $E$ can be associated with each \emph{nonlinear $E$-connection on a vector bundle $B$}. This then naturally
allows us to establish the theory of nonlinear $E$-connections together with the theory of their curvature and of their characteristic
classes, which live in the Loday algebroid cohomology groups, in terms of differential operators, and to extend some classical results of connection theory to this setting.
Our construction follows, in general, the basic approach of Fernandes and Crainic \cite{rui, cr-fer-sec-cl} for Lie algebroids (linear) connections and
of Cueca and Mehta \cite{CM2021} for Courant algebroids linear connections.
However, since the left Leibniz rule and the skew-symmetry of the Loday bracket fail -- a fact that produces tensorial problems for the curvature -- we are obliged
to diverse, at certain points, from their techniques.

\vspace{1mm}
\noindent
The paper is structured as follows. In Section \ref{section - loday alg} we first recall the basic notions related to
Loday algebroids and Loday algebroids cohomology, and present some characteristic examples.
The new contributions in this section are the notion of \emph{action Loday algebroid} (subsection \ref{subsection - action Loday alg}),
where we address the problem of defining a Loday algebroid structure on a pullback bundle of a Loday algebroid, and that of
\emph{Loday algebroid (co)morphism} (subsection \ref{subsection - Loday alg comorphisms}). Section \ref{section - nonlinear connections}
is devoted to the theory of \emph{nonlinear connections} in the context of Loday algebroids.
We begin by fixing the notion and introducing the
\emph{connection matrix} and the \emph{curvature matrix} associated with it, which are matrices of differential
operators on the module of smooth sections of the Loday algebroid, needed below for the construction of characteristic classes.
We then discuss how connection and curvature matrices transform under a change of frame of local sections of $B$. The properties of
the induced connections on the dual bundle and the bundle of endomorphisms of $B$
are given afterwards, and the Bianchi identity is derived.
The behavior of Loday algebroids nonlinear connections under Loday algebroids (co)morphisms is also studied, and several
concrete examples of such connections are presented at the end. The main results of Section \ref{section char classes}
concern the \emph{characteristic classes of Loday algebroids}. First, we describe a construction
of the \emph{modular class} of a Loday algebroid, following the classical path \cite{evens-lu-weinstein, stienon-xu},
prove that it is an element of the first Loday algebroid cohomology of $E$, and present some examples.
The \emph{modular class of a base preserving Loday algebroid (co)morphism}
is also introduced, and is shown to be the characteristic class of a line bundle.
The interpretation of a nonlinear connection and its curvature by matrices allows us to extend the usual
Chern-Weil theory in this setting and to define the secondary characteristic classes of nonlinear $E$-connections,
which are useful for study global properties of Loday algebroids.
These classes are, in general, associated with representations, belong to the Loday algebroid cohomology groups,
and yield analogous results about the geometric properties of the manifold. In classical construction \cite{milnor-stasheff, cr-fer-sec-cl}, based on the fact that the initial
Lie algebroid structure of $E$\footnote{In \cite{milnor-stasheff}, $E=TM$.} can be
pulled back to $M\times \Delta^k$, where $\Delta^k$ is the standard $k$-simplex, a (linear) $\pi_1^\ast E$-connection on $\pi_1^\ast B$ is defined --
$\pi_1$ being the projection of $M\times \Delta^k$ on $M$ -- parametrized by $\Delta^k$, and the Chern-Simons theory is then developed.
This approach is not suitable in our case, because differential operators we work with cannot, in general, pulled back over a projection.
We overcame this obstacle by adapting the approach of \cite{balcerzak} for $\R$-linear connections on Lie algebroids.
Then, we establish a generalized Chern-Simons formula for the Chern-Simons transgression operators (Theorem \ref{theorem-chern-simons}) and
introduce the secondary characteristic classes for nonlinear connections of Loday algebroids,
whose basic properties are given in Theorem \ref{theorem - classes metric connect}.
In Section \ref{section - appendices} we conclude the paper with two appendices on: (i) \emph{Differential operators,
jet bundles and jet modules} (Appendix \ref{app-Jet bundles}) and (ii) \emph{Multidifferential operators } (Appendix \ref{app-multidiff}).

\vspace{2mm}
\noindent
\textbf{Acknowledgements:} This research was  partially supported  by the Centre for Mathematics of the University of Coimbra (CMUC, https://doi.org/10.54499/UID/00324/2025) under the Portuguese Foundation for Science and Technology (FCT) (Grants UID/00324/2025 and UID/PRR/00324/2025) and the Division of Geometry of the Department of Mathematics of the Aristotle University of Thessaloniki (AUTH).
The authors would like to thank Jan Vysok\'y for his helpful explanations.

\section{Loday algebroids}\label{section - loday alg}
In this section, we begin by reviewing the notion of Loday algebroids and some of their basic properties, as well as their cohomology, following the treatment in \cite{grab-k-pon}. We then introduce the notion of \emph{actions of Loday algebroids} and study the (co)morphisms between Loday algebroids.

\subsection{Definition and basic examples}
\begin{definition}\label{def-Loday algebra}
A \emph{Loday algebra} over a commutative ring $\mathbb{K}$ is a $\mathbb{K}$-module $\mathcal{E}$ equipped
with a bilinear operation $\lcf \cdot, \cdot \rcf: \mathcal{E} \times \mathcal{E} \to \mathcal{E}$ satisfying the following version of the Jacobi identity: For any $e_1, e_2, e_3 \in \mathcal{E}$,
\begin{equation}\label{Jacobi-identity}
\lcf e_1, \lcf e_2, e_3 \rcf \rcf = \lcf \lcf e_1, e_2 \rcf, e_3 \rcf + \lcf e_2, \lcf e_1, e_3 \rcf \rcf.
\end{equation}
I.e., for each $e\in \mathcal{E}$, $\lcf e, \cdot \rcf$ is a derivation of $\lcf \cdot, \cdot \rcf$.\footnote{This type of algebra is introduced in \cite{lod-pir} under the name \emph{left Leibniz algebra}.
The term \emph{Loday algebra} has prevailed in the last years to avoid collision with other concepts of \emph{Leibniz algebras} in the literature.}
\end{definition}

\begin{definition}\cite{grab-k-pon}
A \emph{Loday algebroid structure} on a vector bundle $E\to M$ (of constant rank) is defined by a Loday algebra bracket $\lcf \cdot, \cdot \rcf$ on the $\C$-module $\Gamma(E)$ of smooth sections of $E$, which is a bidifferential
operator of total order less or equal to $1$ such that, for any $e\in \Gamma(E)$, the adjoint operator $\mathrm{ad}_e : \Gamma(E) \to \Gamma(E)$, $e' \mapsto \mathrm{ad}_e(e'):= \lcf e, e' \rcf$, is a derivative endomorphism of $\Gamma(E)$.
\end{definition}

\noindent
Equivalently, we have:

\begin{theorem}\label{theorem - def - Loday}\cite{grab-k-pon}
A \emph{Loday algebra bracket} $\lcf \cdot, \cdot \rcf$ on the space $\Gamma(E)$ of smooth sections of a vector bundle $E\to M$ (of constant rank) defines a \emph{Loday algebroid structure} if and only if there are vector bundle morphisms
\begin{equation*}
\rho : E \to TM \quad \quad \mathrm{and} \quad \quad \alpha : T^\ast M \otimes E \to \mathrm{End}(E),
\end{equation*}
covering the identity on $M$, such that, for all $e_1, e_2 \in \Gamma(E)$ and all $f\in \C$,
\begin{equation*}\label{left loday-br}
\lcf e_1, fe_2 \rcf = f\lcf e_1, e_2 \rcf + \rho(e_1)(f)e_2
\end{equation*}
and
\begin{equation*}\label{right loday-br}
\lcf fe_1, e_2 \rcf = f\lcf e_1, e_2 \rcf - \rho(e_2)(f)e_1 + \alpha(df \otimes e_1)(e_2).
\end{equation*}
If this is the case, the anchors are uniquely determined and the left anchor $\rho$ induces a
homomorphism between the Loday algebras $(\Gamma(E), \lcf \cdot, \cdot \rcf)$ and
$(\Gamma(TM), [\cdot, \cdot ])$\footnote{The space of smooth vector fields on $M$ with composition law the usual Lie bracket is a special kind of Loday algebra.}:
\begin{equation}\label{rho-homom}
\rho(\lcf e_1, e_2 \rcf) = [\rho(e_1), \rho(e_2)].
\end{equation}
By identifying $T M$ with $T M\otimes \langle \mathrm{Id}_E \rangle$, we have that the generalized right anchor map can be
written as $\rho \otimes \mathrm{Id}_E - \alpha^T$, where $\alpha^T : E \to TM \otimes \mathrm{End}(E)$ is
the `transpose" of $\alpha$, i.e. $\alpha^T(e_2)(df\otimes e_1) = \alpha(df \otimes e_1)(e_2)$.
\end{theorem}

\noindent
The equality \eqref{rho-homom} means also that $\rho$ is a representation of the Loday algebra $(\Gamma(E), \lcf \cdot, \cdot\rcf)$ by derivations on $\C$.

\vspace{1mm}
\noindent
By duality, the map $\alpha$ induces several associated vector bundle maps, all denoted by $\alpha^T$, as well as the corresponding induced maps on the spaces of smooth sections:
\begin{enumerate}
\item
\[\alpha^T : E\otimes E\otimes E^\ast \to TM\]
such that, for any $e_1, e_2\in \Gamma(E)$, $\mu\in \Gamma(E^\ast)$ and $f\in \C$,
\begin{equation*}
\alpha^T(e_1\otimes e_2 \otimes \mu)(f) : = \langle  \mu, \alpha^T(e_2)(df \otimes e_1)\rangle = \langle  \mu, \alpha(df \otimes e_1)(e_2)\rangle.
\end{equation*}
\item
\[\alpha^T : T^\ast M \otimes E \otimes E \to E \;\;\quad or \quad \;\; \alpha^T : T^\ast M \to E^\ast \otimes E^\ast \otimes E,\]
such that
\begin{equation}\label{alpha - alpha t}
\langle \mu, \alpha^T(df\otimes e_1 \otimes e_2)\rangle = \langle  \mu, \alpha(df \otimes e_1)(e_2)\rangle = \langle  \mu, \alpha^T(df)(e_1,e_2)\rangle.
\end{equation}
\end{enumerate}
Also, $\alpha$ defines a derivation
\[\mathrm{d}: \C \to \Gamma(E^\ast) \otimes \Gamma(E^\ast) \otimes \Gamma(E)  \cong \mathrm{Hom}_{\C}(\Gamma(E)\otimes \Gamma(E), \Gamma(E))\]
such that $\mathrm{d} = \alpha^T \circ d$, where $d$ is the usual derivation of smooth functions on $M$.

\begin{remark}(Locality of the bracket)\label{remark-locality}
\emph{Using an argumentation similar to that developed in \cite{marle-coimbra} for the bracket of a Lie algebroid, one can prove that the bracket of two smooth sections $e_1,e_2$ of a Loday algebroid $E$ depends only on the $1$-jets of $e_1$ and $e_2$. That means that, the bracket $\lcf \cdot, \cdot \rcf$ is a first-order bidifferential operator with respect to each argument. Its symbol with respect to the first argument is the generalized right anchor map $\rho \otimes \mathrm{Id}_E - \alpha^T$, which is a derivation in $f$ and is $\C$-linear in $e_1$ and $e_2$. The symbol with respect to the second argument is the left anchor map, which is also a derivation in $f$, is $\C$-linear in $e_1$, and acts as the identity on $e_2$. Moreover, $\lcf \cdot, \cdot \rcf$ is of total order at most $1$ (see Appendix \ref{app-multidiff}).}
\end{remark}

\vspace{1mm}
\noindent
Some classical examples of Loday algebroids are presented below. Most of them are discussed in detail in \cite{grab-k-pon}.
\begin{enumerate}
\item
\emph{Loday algebras:} A finite dimensional Loday algebra $(\mathcal{E}, \lcf \cdot, \cdot \rcf)$ is a Loday algebroid over a point.
\item
\emph{Bundles of Loday algebras:} Let $E$ be a Loday algebroid with trivial
left and right anchors maps, i.e. $\rho = 0$ and $\alpha = 0$. Then, at each point $x\in M$,
the bracket $[\cdot, \cdot]_x : E_x \times E_x \to E_x$ given, for any $e_1, e_2 \in E_x$,
by $[e_1,e_2]: = \lcf \tilde{e}_1, \tilde{e}_2\rcf (x)$, where $\tilde{e}_1, \tilde{e}_2 \in \Gamma(E)$ such that
$\tilde{e}_1(x) = e_1$ and $\tilde{e}_2(x)=e_2$, defines a Loday algebra structure on $E_x$.
Easily, using the left and right Leibniz rules, we can verify that the value of $[e_1,e_2]$ does not depend on the choice of sections. Thus, we can see $E$ as a
vector bundle whose each fiber $E_x$, $x\in M$, has a Loday algebra structure dependent on the point $x$. The converse is also true.
\item
\emph{Lie algebroids:} A Lie algebroid is a Loday algebroid whose the bracket $\lcf \cdot, \cdot \rcf$ is skew-symmetric and $\alpha = 0$.
\item
\emph{Courant algebroids:} The vector bundle $E$ is furthermore equipped with a compatible fiberwise nondegenerate symmetric bilinear form $\langle \cdot, \cdot \rangle$ and $\alpha : T^\ast M \otimes E \to \mathrm{End}(E)$ is the map $df\otimes e_1 \mapsto \langle e_1, \cdot \rangle d_Ef$, where $d_E = \rho^T \circ d : \C \to \Gamma(E^\ast ) \cong \Gamma(E)$.
\item
\emph{Loday algebroids associated to a Nambu-Poisson structure:} A Nambu-Poisson structure of order $n$ on a smooth $m$-dimensional manifold $M$,
$3\leqslant n \leqslant m$, is a decomposable $n$-vector
field $\Lambda$ \cite{vai-nambu} which defines a vector bundle map $\Lambda^ \sharp : \Gamma(\bigwedge^{n-1} T^\ast M)\to \Gamma(TM)$ given by
\[\Lambda^\sharp (\omega) = i_{\omega}\Lambda, \quad \quad \omega\in \Gamma(\bigwedge^{n-1} T^\ast M).\]
The Loday algebroid attached to $(M, \Lambda)$ is the tetrad $(\bigwedge^{n-1} T^\ast M, \lcf \cdot, \cdot \rcf, \rho, \alpha)$, where, for any
$\omega, \eta \in \Gamma(\bigwedge^{n-1} T^\ast M)$,
\[\lcf \omega, \eta \rcf = \mathcal{L}_{\Lambda^\sharp (\omega)}\eta + (-1)^n \Lambda(d \omega)\eta,\]
\[\rho(\omega) = \Lambda^{\sharp}(\omega) \quad \mathrm{and} \quad \alpha(df \otimes \omega)(\eta) = \rho(\eta)(f)\omega - \rho(\omega)(f)\eta + df \wedge i_{\rho(\omega)}\eta. \]
\item
\emph{Grassmann-Dorfman algebroid:} Consider $E= TM \oplus \bigwedge^p T^\ast M$ with (left) anchor $\rho$ the projection on the tangent bundle and bracket $\lcf \cdot, \cdot \rcf$
defined as follows. For all pairs $(X, \zeta), (Y,\eta)\in \Gamma(E)$ and $f\in \C$,
\[\lcf (X, \zeta), (Y,\eta)\rcf = ([X,Y], \mathcal{L}_X\eta - i_Yd\zeta).\]
This is a Loday algebroid with
\[\alpha(df \otimes (X, \zeta))(Y, \eta) = (0, df \wedge (i_X\eta + i_Y\zeta)),\]
that is neither a Lie nor a Courant algebroid \cite{bi-sheng, grab-k-pon}.
\end{enumerate}

\vspace{1mm}
\noindent
The following lemma summarizes several properties of the structure $(\lcf \cdot, \cdot \rcf, \rho, \alpha)$ on $E$ that will be useful in what follows.
\begin{lemma}\label{lemma-Raquel}
Let $(E, \lcf \cdot, \cdot \rcf, \rho, \alpha)$ be a Loday algebroid over $M$ and $\lcf \cdot, \cdot \rcf_{+}$ the symmetrization of $\lcf \cdot, \cdot \rcf$:
\[\lcf e_1, e_2 \rcf_{+}: = \lcf e_1, e_2 \rcf + \lcf e_2, e_1 \rcf.\]
Then, for any $e_1, e_2, e_3 \in \Gamma(E)$ and $f,g,h\in \C$, the following compatibility conditions between $\lcf \cdot, \cdot \rcf$
and the maps $\rho$ and $\alpha$ are satisfied:
\begin{enumerate}
\item
\begin{eqnarray}\label{cond-lemma-rq-1}
\lefteqn{\alpha(df\otimes e_1)(\lcf e_2, e_3 \rcf) - \lcf \alpha(df\otimes e_1)(e_2), e_3 \rcf -\lcf e_2,  \alpha(df\otimes e_1)(e_3)\rcf = }\nonumber \\
& & - \,\rho(e_3)(f) \lcf e_1, e_2 \rcf_{+} + \alpha (df \otimes  \lcf e_1, e_2 \rcf )(e_3) - \alpha (d\rho(e_2)(f) \otimes e_1 )(e_3);
\end{eqnarray}
\item
\begin{eqnarray}\label{cond-lemma-rq-2}
\lefteqn{\lcf e_1, \alpha(dg\otimes e_2)(e_3) \rcf = \alpha (dg \otimes  \lcf e_1, e_2 \rcf )(e_3)}\nonumber \\
& & + \, \alpha (d\rho(e_1)(g) \otimes e_2 )(e_3) + \alpha (dg \otimes  e_2)(\lcf e_1, e_3 \rcf );
\end{eqnarray}
\item
\begin{eqnarray}\label{cond-lemma-rq-3}
\lefteqn{[\alpha(df \otimes e_1), \alpha(dg \otimes e_2)]_{com}(e_3) = \alpha(dg \otimes \alpha(df \otimes e_1)(e_2))(e_3)}\nonumber \\
& &  + \,\rho(e_1)(g)\alpha(df \otimes e_2)(e_3) - \rho(e_2)(f)\alpha(dg \otimes e_1)(e_3) - \rho(e_3)(f)\alpha(dg \otimes e_2)(e_1),
\end{eqnarray}
where $[\cdot, \cdot]_{com}$ denotes the commutator bracket on $\mathrm{End}(E)$;
\item
\begin{equation}\label{rho-alpha = 0}
\rho(\alpha(df\otimes e_1)(e_2)) = 0;
\end{equation}
\item
\begin{equation}\label{alfa - +}
\lcf fe_1, e_2 \rcf_{+} = f\lcf e_1, e_2 \rcf_+ + \alpha(df \otimes e_1)(e_2).
\end{equation}
\end{enumerate}
\end{lemma}
\begin{proof}
The identities \eqref{cond-lemma-rq-1}--\eqref{cond-lemma-rq-3} are obtained, respectively, by applying the Jacobi identity \eqref{Jacobi-identity} to the triples $(fe_1,e_2,e_3)$, $(e_1,ge_2,e_3)$, and $(fe_1,ge_2,e_3)$. In the proof of \eqref{cond-lemma-rq-3}, we also make use of the identities \eqref{cond-lemma-rq-1} and \eqref{cond-lemma-rq-2}. The identity \eqref{rho-alpha = 0} follows directly from \eqref{rho-homom} and shows that $\alpha$ takes values in the subbundle $\ker \rho$ of $E$. It is straightforward to verify that \eqref{alfa - +} holds. Thus, $\alpha$ may be viewed as the obstruction to the $\C$-linearity of the bracket $\lcf \cdot, \cdot \rcf_{+}$ on $\Gamma(E)$.

\end{proof}

\begin{remark}
\emph{The compatibility relation \eqref{cond-lemma-rq-1} makes it clear that, if $\alpha=0$, then either the bracket $\lcf\cdot,\cdot\rcf$ is skew-symmetric, in which case $E$ is necessarily a Lie algebroid, or the anchor $\rho$ also vanishes, in which case $E$ is a bundle of Loday algebras.}
\end{remark}

\vspace{1mm}
\noindent
Before presenting the next example, we recall the following notion.

\begin{definition}\cite{grab-k-pon}\label{pseudoalgebra-Loday}
Let $\mathcal{E}$ be a faithful $\mathcal{A}$-module endowed with a $\mathbb{K}$-bilinear bracket $[\cdot,\cdot] : \mathcal{E}\times \mathcal{E}\to \mathcal{E}$
($\mathcal{A}$ being an associative commutative algebra over the field $\mathbb{K}$ of characteristic 0).
We say that $(\mathcal{E}, [\cdot,\cdot])$ \emph{is a pseudoalgebra}, if $[\cdot,\cdot]$ is a bidifferential operator of total order $\leq 1$ and,
for any $e\in \mathcal{E}$, the adjoint map $\mathrm{ad}_e : = [e, \cdot]$
is a derivative endomorphism of $\mathcal{E}$. If, furthermore, the bracket $[\cdot,\cdot]$
satisfies the Jacobi identity, we say that $(\mathcal{E}, [\cdot,\cdot])$ is a \emph{Loday pseudoalgebra}.
\end{definition}

\begin{example L-L}\cite{grab-k-pon}\label{Loday to Lie}
\emph{A nice result relating Loday algebroids to Lie algebroids is the following. Let $(E, \lcf \cdot, \cdot \rcf, \rho, \alpha)$
be a Loday algebroid over a smooth manifold $M$, $\mathcal{E}=\Gamma(E)$, and
\[\mathcal{E}_0 = \mathrm{span}\{\lcf e,e\rcf \, / \, e\in \mathcal{E}\}\]
the two-sided ideal of $\mathcal{E}$ generated by the symmetrization of $\lcf \cdot ,\cdot\rcf$. We have $\lcf \mathcal{E}_0, \mathcal{E} \rcf = 0$ and
$\lcf \mathcal{E}, \mathcal{E}_0 \rcf \subset \mathcal{E}_0$. Setting $\bar{\mathcal{E}_0} : = \C \cdot \mathcal{E}_0$,
we obtain a $\C$-submodule of $\mathcal{E}$ and may consider the quotient $\C$-module $\bar{\mathcal{E}}= \mathcal{E}/ \bar{\mathcal{E}_0}$.
It is well known \cite{grab-k-pon} that the Loday algebra structure on $\mathcal{E}$ induces a Lie
pseudoalgebra structure on $\bar{\mathcal{E}}$, with bracket $[\cdot, \cdot] : \bar{\mathcal{E}}\times \bar{\mathcal{E}} \to \bar{\mathcal{E}}$,
$[\bar{e}_1, \bar{e}_2]: =\overline{\lcf e_1, e_2 \rcf}$, where $\bar{e}_{\bullet}$ denotes the equivalence class of $e_{\bullet} \in \mathcal{E}$,
and anchor map $\bar{\rho} : \bar{\mathcal{E}} \to \mathfrak{X}(M)$, $\bar{\rho}(\bar{e}):= \rho(e)$.
Moreover, $\rho(\bar{\mathcal{E}_0})=0$ and, for any $e_1,e_2\in \mathcal{E}$ and $f\in \C$, we have $\alpha(df\otimes e_1)(e_2)\in \bar{\mathcal{E}_0}$.
Assume that $\bar{\mathcal{E}_0}$ is the $\C$-module of smooth sections of a vector subbundle $E_0$ of $E$. Then $\bar{\mathcal{E}}$ is the $\C$-module
of smooth sections of the quotient vector bundle $\bar{E}= E/E_0$, which therefore carries a natural Lie algebroid structure.}

\noindent
\emph{Also, there is a short exact sequence of morphisms of Loday pseudoalgebras over $\C$, }
\begin{equation}\label{exact seq loday}
0 \longrightarrow \bar{\mathcal{E}_0} \longrightarrow \mathcal{E} \longrightarrow \bar{\mathcal{E}} \longrightarrow 0,
\end{equation}
\emph{where $\bar{\mathcal{E}_0}$ is a Loday pseudoalgebra with the trivial left anchor and $\bar{\mathcal{E}}$ is a Lie pseudoalgebra. Moreover, since $\bar{\mathcal{E}}$ is free,
as space of smooth sections of a vector bundle of constant rank, the exact sequence \eqref{exact seq loday} splits \cite[Proposition 21, II-1.11]{bourbaki}}.

\vspace{2mm}
\noindent
\emph{If $E$ is a Courant algebroid with anchor $\rho$ of constant rank, $\mathcal{E}_0 = \mathrm{span}\{d_E \langle e,e \rangle \, / \, e\in \mathcal{E}\}$ and
the corresponding Lie algebroid $\bar{E}$ is the quotient bundle $E/\mathrm{Im}\rho^T$ \cite{severa}.}
\end{example L-L}

\begin{remark}\label{rem-loday-popescu}
\emph{At this point, recalling the notion of a \emph{generalized algebroid} introduced by Popescu in \cite{popescu}, we observe that every Loday algebroid can be regarded as a $\bar{\mathcal{E}}_0$-algebroid. A \emph{generalized algebroid} is a pair $(E,\mathcal{S})$, where $E$ is a vector bundle over a smooth manifold $M$, endowed with an anchor map $\rho\colon E\to TM$ and a bracket $[\cdot,\cdot]$ on $\Gamma(E)$, and $\mathcal{S}$ is a $\C$-submodule of $\Gamma(E)$ such that:}
\begin{itemize}
\item[-]
\emph{the Jacobiator of $[\cdot,\cdot]$ takes values in $\mathcal{S}$;}
\item[-]
\emph{$[e_1,fe_2]-f[e_1,e_2]-\rho(e_1)(f)e_2 \in \mathcal{S}$ and $[fe_1,e_2]-f[e_1,e_2]+\rho(e_2)(f)e_1 \in \mathcal{S}$, for any $e_1,e_2\in \Gamma(E)$ and $f\in \C$;}
\item[-]
\emph{$\rho([e_1,e_2]) = [\rho(e_1),\rho(e_2)]$, for any $e_1,e_2\in \Gamma(E)$;}
\item[-]
\emph{$[e_1,e_2]\in \mathcal{S}$, whenever $e_1$ or $e_2$ are in $\mathcal{S}$.}
\end{itemize}
\emph{Taking into account the relations of Lemma \ref{lemma-Raquel} and the fact that $\alpha(df\otimes e_1)(e_2)\in \bar{\mathcal{E}_0}$, we can easily verify that, in the case of Loday algebroids, the above conditions are satisfied for $\mathcal{S}= \bar{\mathcal{E}_0}$.}
\end{remark}

\subsection{Actions of Loday algebroids}\label{subsection - action Loday alg}
By the general theory of vector bundles, it is well known that, for vector bundles over different base manifolds and a vector bundle map $(\Phi,\phi)$
\[\begin{tikzcd} F \arrow{r}{\Phi}\arrow[swap]{d} \arrow{d}[swap]{\pi_{_F}} & E \arrow{d}{\pi_{_E}}\\ N \arrow{r}[swap]{\phi}& M, \end{tikzcd}\]
there is, in general, no induced map $\Gamma(F)\to\Gamma(E)$ between the spaces of smooth sections. The reason is that $\Gamma(E)$ and $\Gamma(F)$ are bimodules over different algebras. If we take $F$ to be the pullback vector bundle $\phi^\ast E$ of $E$ over $\phi$, together with the vector bundle map $\mathrm{pr}\colon\phi^\ast E\to E$ (which is the projection onto the second factor of $\phi^\ast E\subset N\times E$), it is not possible, in general, to pull back an algebroid structure on $\Gamma(E)$ to an algebroid structure of the same type on $\Gamma(\phi^\ast E)$. In \cite{hig-mck}, Higgins and Mackenzie consider the case where $E$ is a Lie algebroid and define a Lie algebroid structure on $\phi^\ast E$ by using the additional structure of an action of $E$ on $\phi$. The same method is used by Li-Bland and Meinrenken \cite{li-bland-meinr} in the case of Courant algebroids. Assuming that $E$ is a Loday algebroid, the purpose of this paragraph is to construct a Loday algebroid structure on $\phi^\ast E$. Our construction uses an \emph{action of the Loday algebroid} $E$ on $\phi$. This notion is defined below by adapting the notion of an \emph{action of a Lie algebroid}.
Before proceeding, we remark that, in the Lie algebroid case, since the left and right anchor maps coincide up to sign, an action is defined by a map $\varrho\colon\Gamma(E)\to\Gamma(TN)$. In the Loday algebroid case, since the left and right anchor maps are different, in addition to $\varrho\colon\Gamma(E)\to\Gamma(TN)$, we need an extra map with values in $\Gamma(TN)$, which will be used to define the right anchor map on $\Gamma(\phi^\ast E)$.

\begin{definition}\label{def-action-loday alg - NEW}
Suppose that $(E, \lcf \cdot, \cdot \rcf, \rho, \alpha)$ is a Loday algebroid on $M$, $N$
a smooth manifold, and $\phi : N \to M$ a smooth map. An \emph{(infinitesimal) action of $E$ on $\phi$ (or on $N$)}
is a pair $(\varrho, \mathfrak{a}^T)$ of $\R$-linear maps, $\varrho : \Gamma(E) \to \Gamma(TN)$ and $\mathfrak{a}^T :
\Gamma(E)\otimes \Gamma(E)\otimes \Gamma(E^\ast) \to \Gamma(TN)$\footnote{The tensor product of $\C$-modules is considered over the algebra $\C$.} such that:
\begin{enumerate}
\item
$\varrho : \Gamma(E) \to \Gamma(TN)$ is a Loday algebra
homomorphism, i.e.
\begin{equation}\label{def-act lod alg - 1 - NEW}
\varrho(\lcf e_1, e_2 \rcf) = [\varrho(e_1), \varrho(e_2)], \quad \mbox{for  all } \, e_1, e_2 \in \Gamma(E),
\end{equation}
which is
\begin{enumerate}
\item
$\C$-linear:
\begin{equation}\label{def-act lod alg - 3 - NEW}
\varrho(ge)=\phi^\ast g \varrho(e) = (g\circ \phi)\varrho(e), \quad \mbox{for  any } \, g\in \C \; and \; e\in \Gamma(E),
\end{equation}
\item
and $\phi$-related, compatible, with the left anchor map $\rho$ of $E$:
\begin{equation}\label{def-act lod alg - 2 - NEW}
(d\phi \circ \varrho)(e) = \rho(e)\circ \phi, \quad \mbox{for  all } \, e\in \Gamma(E).
\end{equation}
\end{enumerate}
\item
$\mathfrak{a}^T : \Gamma(E)\otimes \Gamma(E)\otimes \Gamma(E^\ast) \to \Gamma(TN)$ is
\begin{enumerate}
\item
$\C$-linear: i.e., for any $e_1, e_2 \in \Gamma(E)$, $\mu \in \Gamma(E^\ast)$ and $g\in \C$,
\begin{equation}\label{def-act lod alg - 5 -NEW}
\mathfrak{a}^T(g e_1\otimes e_2 \otimes \mu) = \phi^\ast g \mathfrak{a}^T (e_1\otimes e_2 \otimes \mu)=(g\circ \phi) \mathfrak{a}^T (e_1\otimes e_2 \otimes \mu);
\end{equation}
\item
$\phi$-related, compatible, with $\alpha^T$: for any $e_1, e_2 \in \Gamma(E)$ and $\mu \in \Gamma(E^\ast)$,
\begin{equation}\label{def-act lod alg - 6 -NEW}
(d\phi \circ \mathfrak{a}^T)(e_1\otimes e_2 \otimes \mu) = \alpha^T(e_1\otimes e_2 \otimes \mu)\circ \phi;
\end{equation}
\item
compatible with the bracket $\lcf \cdot, \cdot \rcf$ on $\Gamma(E)$ in the sense that, for any $e_1, e_2, e_3 \in \Gamma(E)$, $\mu \in \Gamma(E^\ast)$ and $f, g\in C^\infty (N, \R)$, the following conditions are satisfied:
\begin{eqnarray}\label{cond-lemma-rq-1 - action}
\lefteqn{\mathfrak{a}(df\otimes e_1)(\lcf e_2, e_3 \rcf) - \lcf \mathfrak{a}(df\otimes e_1)(e_2), e_3 \rcf -\lcf e_2,  \mathfrak{a}(df\otimes e_1)(e_3)\rcf = }\nonumber \\
& & - \varrho(e_3)(f) \lcf e_1, e_2 \rcf_{+} + \mathfrak{a} (df \otimes  \lcf e_1, e_2 \rcf )(e_3) - \mathfrak{a} (d\varrho(e_2)(f) \otimes e_1 )(e_3),
\end{eqnarray}
\begin{eqnarray}\label{cond-lemma-rq-2 - action}
\lefteqn{\lcf e_1, \mathfrak{a}(dg\otimes e_2)(e_3) \rcf = \mathfrak{a}(dg \otimes  \lcf e_1, e_2 \rcf )(e_3)}\nonumber \\
& & + \, \mathfrak{a} (d\varrho(e_1)(g) \otimes e_2 )(e_3) + \mathfrak{a} (dg \otimes  e_2)(\lcf e_1, e_3 \rcf,
\end{eqnarray}
\begin{eqnarray}\label{cond-lemma-rq-3 - action}
\lefteqn{[\mathfrak{a}(df \otimes e_1), \mathfrak{a}(dg \otimes e_2)]_{com}(e_3) = \mathfrak{a}(dg \otimes \mathfrak{a}(df \otimes e_1)(e_2))(e_3)}\nonumber \\
& &  +\, \varrho(e_1)(g)\mathfrak{a}(df \otimes e_2)(e_3) -  \varrho(e_2)(f)\mathfrak{a}(dg \otimes e_1)(e_3) \nonumber \\
& & -\, \varrho(e_3)(f)\mathfrak{a}(dg \otimes e_2)(e_1),
\end{eqnarray}
where $\mathfrak{a} : \Gamma(T^\ast N) \otimes \Gamma(E) \to \mathrm{End}(\Gamma(E))$ denotes the ``transpose" map of $\mathfrak{a}^T$.
\item
Moreover, we require the associated map $\mathfrak{a}^T : \Gamma(T^\ast N)\otimes \Gamma(E)\otimes \Gamma(E) \to \Gamma(E)$ to take values in $\ker \varrho$:
\begin{equation}\label{def-act lod alg - 4 -NEW}
\varrho (\mathfrak{a}^T (df\otimes e_1 \otimes e_2)) = 0.
\end{equation}
\end{enumerate}
\end{enumerate}
\end{definition}

\vspace{2mm}
\noindent
Furthermore, we remark that the transpose $\varrho^T : \Gamma(T^\ast N) \to \Gamma(E^\ast)$ of $\varrho$ defines a natural ``differential" $\underline{d} : C^\infty(N, \R) \to \Gamma(E^\ast)$ given, for any $h\in C^\infty(N, \R)$, by
\begin{equation}\label{varrho t - d}
\underline{d}h = (\varrho^T \circ d)h,
\end{equation}
whence we get
\begin{equation*}
\langle \underline{d}h, e\rangle = \varrho(e)(h), \quad \mathrm{for}\; \mathrm{all}\; e\in \Gamma(E).
\end{equation*}
Also, for any $e_1,e_2 \in \Gamma(E)$, we have
\begin{equation*}
\mathfrak{a}^T(e_1\otimes e_2 \otimes \underline{d}h) = \langle \underline{d}h, \mathfrak{a} (\cdot \otimes e_1)(e_2) \rangle = \langle \varrho^T dh, \mathfrak{a} (\cdot \otimes e_1)(e_2) \rangle = \langle dh, \varrho(\mathfrak{a} (\cdot \otimes e_1)(e_2)) \rangle \stackrel{\eqref{def-act lod alg - 4 -NEW}}{=} 0.
\end{equation*}
The last result is also true for the sections of $\Gamma(E^\ast)$ of type $\varrho^T(\eta)$, for all $\eta\in \Gamma(T^\ast N)$.

\begin{theorem}\label{theorem-action lod alg - new}
Let $(E, \lcf \cdot, \cdot \rcf, \rho, \alpha)$ be a Loday algebroid over $M$, $N$ a smooth manifold, $\phi : N \to M$ a smooth map, $\phi^\ast E$ the pullback vector bundle of $E$ over $N$, $\phi^\ast : \Gamma(E) \to \Gamma(\phi^\ast E)$ the map which assigns to each smooth section of $E$
its pullback section, and $(\varrho, \mathfrak{a}^T)$ an action of $E$ on $\phi$. Then the pullback vector bundle $\phi^\ast E$ carries a Loday algebroid structure with the following data:
\begin{itemize}
\item[-]
\emph{Left anchor map}: The map $\varrho_{_N} : \Gamma(\phi^\ast E)\to \Gamma(TN)$ given, for any $f\phi^\ast e = f\otimes e \in \Gamma(\phi^\ast E)$, where $f\in C^\infty(N, \R)$ and $e\in \Gamma(E)$, by
\begin{equation}\label{l-anchor - act lod alg - new}
\varrho_{_N}(f\phi^\ast e ):= f\varrho(e).
\end{equation}
\item[-]
\emph{Map which controls the right anchor map}: The map $\mathfrak{a}_{_N}^T : \Gamma(\phi^\ast E)\otimes \Gamma(\phi^\ast E)\otimes \Gamma(\phi^\ast E^\ast)\to \Gamma(TN)$ defined, for any $f_1\phi^\ast e_1, f_2\phi^\ast e_2 \in \Gamma(\phi^\ast E)$ and $g\phi^\ast \mu \in \Gamma(\phi^\ast E^\ast) = \Gamma((\phi^\ast E)^\ast)$, where $f_1,f_2,g\in C^\infty(N, \R)$, by
\begin{equation}\label{r-anchor - act lod alg - new}
\mathfrak{a}_{_N}^T (f_1\phi^\ast e_1\otimes f_2\phi^\ast e_2\otimes g\phi^\ast \mu): = f_1f_2g\,\mathfrak{a}^T(e_1\otimes e_2 \otimes \mu).
\end{equation}
\item[-]
\emph{$\R$-Bilinear bracket on $\Gamma(\phi^\ast E)$}: The bracket $\lbr \cdot, \cdot \rbr$ defined, for any $e_1, e_2 \in \Gamma(E)$ and $f_1, f_2 \in C^\infty(N, \R)$, by
\begin{eqnarray}\label{bracket action - new}
\lbr f_1\phi^\ast e_1, f_2\phi^\ast e_2 \rbr & = & f_1f_2 \phi^\ast \lcf e_1, e_2\rcf + f_1\varrho_{_N}(\phi^\ast e_1)(f_2)\phi^\ast e_2 - f_2\varrho_{_N}(\phi^\ast e_2)(f_1)\phi^\ast e_1 \nonumber \\
& & + \, f_2\mathfrak{a}_{_N}(df_1\otimes \phi^\ast e_1)(\phi^\ast e_2),
\end{eqnarray}
where $\mathfrak{a}_{_N} : \Gamma(T^\ast N) \otimes \Gamma(\phi^\ast E)\to \mathrm{End}(\Gamma(\phi^\ast E))$ is the map related with $\mathfrak{a}_{_N}^T$ as follows:
\begin{equation*}
\langle g\phi^\ast \mu, \, \mathfrak{a}_{_N}(df_1\otimes \phi^\ast e_1)(f_2\phi^\ast e_2)\rangle = \mathfrak{a}_{_N}^T(\phi^\ast e_1 \otimes f_2\phi^\ast e_2\otimes g\phi^\ast \mu)(f_1).
\end{equation*}
\end{itemize}
\end{theorem}
\begin{proof}
First, we check that $\varrho_{N}$ and $\mathfrak{a}_{_N}$ are well defined. Condition \eqref{def-act lod alg - 3 - NEW} of Definition \ref{def-action-loday alg - NEW} ensures that $\varrho_{_N}$ defines a vector bundle map, also denoted by $\varrho_{_N}$, $\varrho_{_N} : \phi^\ast E \to TN$.
Indeed, for any $e\in \Gamma(E)$, $h\in \C$ and $f, f_1, f_2, g \in C^\infty(N, \R)$,
\begin{equation*}
\varrho_{_{N}}(f\phi^\ast (he) ) \stackrel{\eqref{l-anchor - act lod alg - new}}{=} f\varrho(he)\stackrel{\eqref{def-act lod alg - 3 - NEW}}{=} f\phi^\ast h \varrho(e)
\end{equation*}
and
\begin{equation*}
\varrho_{_{N}}(f\phi^\ast (he) ) = \varrho_{_N}(f\phi^\ast h \phi^\ast e)\stackrel{\eqref{l-anchor - act lod alg - new}}{=} f\phi^\ast h \varrho(e).
\end{equation*}
So,
\[\varrho_{_{N}}(f\phi^\ast h \phi^\ast e) = f\phi^\ast h\varrho_{_N}(\phi^\ast e).\]
Also, for use in the following, we note that
\begin{equation*}\label{phi and rho}
\varrho_{_N}(\phi^\ast e)(\phi^\ast h) \stackrel{\eqref{l-anchor - act lod alg - new}}{=} \langle \phi^\ast dh, \varrho(e) \rangle = \langle dh, d\phi(\varrho(e))\rangle \circ \phi \stackrel{\eqref{def-act lod alg - 2 - NEW}}{=} \langle dh, \rho(e)  \rangle \circ \phi = \phi^\ast (\rho(e)(h)).
\end{equation*}
With the above notation, the map $\mathfrak{a}_{_N}^T$ satisfies
\begin{eqnarray*}
\mathfrak{a}_{_N}^T (f_1\phi^\ast (h e_1)\otimes f_2\phi^\ast e_2\otimes g\phi^\ast \mu) &  \stackrel{\eqref{r-anchor - act lod alg - new}}{=} & f_1f_2g\,\mathfrak{a}^T(he_1\otimes e_2 \otimes \mu) \nonumber \\
& \stackrel{\eqref{def-act lod alg - 5 -NEW}}{=} & f_1f_2g\phi^{\ast}h \mathfrak{a}^T(e_1\otimes e_2 \otimes \mu) \nonumber \\
&\stackrel{\eqref{r-anchor - act lod alg - new}}{=} & \mathfrak{a}_{_N}^T (f_1\phi^\ast h \phi^\ast e_1\otimes f_2\phi^\ast e_2\otimes g\phi^\ast \mu).
\end{eqnarray*}
So, $\mathfrak{a}_{_N}^T $ is well defined and by construction is $C^\infty(N, \R)$-linear. Hence, it induces a vector bundle map, also denoted by $\mathfrak{a}_{_N}^T$, $\mathfrak{a}_{_N}^T : \phi^\ast E\otimes \phi^\ast E \otimes \phi^\ast E^\ast \to TN $
and the corresponding vector bundle map $\mathfrak{a}_{_N} : T^\ast N \otimes \phi^\ast E \to \mathrm{End}(\phi^\ast E)$. Also, the following diagram is commutative:
\begin{equation}\label{diag - alpha_N - phi}
\begin{tikzcd}[column sep=80]
\Gamma(T^\ast N)\otimes \Gamma(E) \otimes \Gamma(E) \arrow{r}{\mathrm{Id}\otimes \phi^\ast \otimes \phi^\ast} \arrow{d}[swap]{\mathfrak{a}^T}& \Gamma(T^\ast N)\otimes \Gamma(\phi^\ast E) \otimes \arrow{d}{\mathfrak{a}_{_N}^T} \Gamma(\phi^\ast E) \\
\Gamma(E) \arrow{r}[swap]{\phi^\ast} & \Gamma(\phi^\ast E).
\end{tikzcd}
\end{equation}
Therefore, for any $f\in C^\infty(N, \R)$ and $e_1, e_2 \in \Gamma(E)$,
\begin{eqnarray}\label{rho_N-alpha_N}
\varrho_{_N}(\mathfrak{a}_{_N}(df \otimes \phi^\ast e_1)(\phi^\ast e_2)) & \stackrel{\eqref{alpha - alpha t}}{=} &  \varrho_{_N}(\mathfrak{a}_{_N}^T(df \otimes \phi^\ast e_1 \otimes \phi^\ast e_2)) \nonumber \\
& \stackrel{\eqref{diag - alpha_N - phi}}{=} & \varrho_{_N}(\phi^\ast (\mathfrak{a}^T(df\otimes e_1 \otimes e_2)))\nonumber \\
& \stackrel{\eqref{l-anchor - act lod alg - new}}{=}& \varrho (\mathfrak{a}^T(df\otimes e_1 \otimes e_2)) \nonumber \\
&\stackrel{\eqref{def-act lod alg - 4 -NEW}}{=}& 0.
\end{eqnarray}

\vspace{1mm}
\noindent
The bracket on pullback sections is defined by
\begin{equation}\label{phi*-br}
\lbr \phi^\ast e_1, \phi^\ast e_2 \rbr = \phi^\ast \lcf e_1, e_2 \rcf
\end{equation}
and then extended to any sections by requiring the left and the right Leibniz identity to hold. So, we obtain the expression \eqref{bracket action - new}, which is well defined. By a straightforward calculation, we prove that, for two local frames $(U,e_1,\ldots,e_r)$ and $(U',e'_1,\ldots,e'_r)$ of smooth sections of $E$, with $U\cap U'\neq\emptyset$, the brackets \eqref{bracket action - new} of two smooth sections of $\phi^\ast E$, defined with respect to $(\phi^{-1}(U),\phi^\ast e_1,\ldots,\phi^\ast e_r)$ and $(\phi^{-1}(U'),\phi^\ast e'_1,\ldots,\phi^\ast e'_r)$, coincide on the overlap $\phi^{-1}(U) \cap \phi^{-1}(U')=\phi^{-1}(U\cap U')$.
The relation \eqref{rho_N-alpha_N} guarantees that $\varrho_{_N}\colon \Gamma(\phi^\ast E)\to\Gamma(TN)$ is a Loday algebra homomorphism (it is enough to apply $\varrho_{_N}$ to \eqref{bracket action - new}). After a long but straightforward calculation, we verify that the Jacobi identity remains valid. In the proof, besides the definition and the properties of the triple $(\lbr \cdot,\cdot \rbr,\varrho_{_N},\mathfrak{a}_{_N})$, we use the compatibility conditions of the action $(\varrho,\mathfrak{a}^T)$ of $E$ on $\phi$ given in Definition \ref{def-action-loday alg - NEW}. Therefore, by Theorem \ref{theorem - def - Loday}, the triple $(\lbr \cdot,\cdot \rbr,\varrho_{_N},\mathfrak{a}_{_N})$ defines a Loday algebroid structure on $\phi^\ast E$.
\end{proof}

\vspace{1mm}
\noindent
The quadruplet $(\phi^\ast E, \lbr \cdot, \cdot \rbr, \varrho_{_N}, \mathfrak{a}_{_N})$ is called \emph{action Loday algebroid} associated to the infinitesimal action $(\varrho, \mathfrak{a}^T)$ of $E$ on $\phi$.

\begin{examples}\label{examples LA actions}
\end{examples}
\begin{enumerate}
\item\label{example: action Loday algebras}
\emph{Action Loday algebroid of a Loday algebra action.}
Let $(\mathcal{E}, \lcf \cdot, \cdot \rcf)$ be a Loday algebra, viewed as a Loday algebroid
over an one-point set $\{p\}$. An action of $\mathcal{E}$ over a smooth manifold $M$ is a pair of maps $(\varrho, \mathfrak{a}^T)$ such that
$\varrho : (\mathcal{E}, \lcf \cdot, \cdot \rcf) \to (\Gamma(TM), [\cdot, \cdot])$ is a homomorphism of Loday algebras and
$\mathfrak{a}^T:\mathcal{E}\otimes \mathcal{E}\otimes \mathcal{E}^*\to \Gamma(TM)$ verifies the compatibility conditions \eqref{cond-lemma-rq-1 - action} - \eqref{def-act lod alg - 4 -NEW}. The remaining conditions \eqref{def-act lod alg - 5 -NEW} and \eqref{def-act lod alg - 6 -NEW} are trivially satisfied.

\vspace{1mm}
\noindent
Let $(\varrho, \mathfrak{a}^T)$ be an action of $\mathcal{E}$ on $M$.
Then, according to Theorem \ref{theorem-action lod alg - new}, the trivial vector bundle $M \times \mathcal{E}$ over $M$, considered as the pullback vector bundle of $\mathcal{E} \to \{p\}$ by the projection map $M\to \{p\}$, acquires a Loday algebroid structure $(\lbr \cdot, \cdot \rbr, \varrho_{_M}, \mathfrak{a}_{_M})$
defined as follows. The maps $\varrho_{_M} : M \times \mathcal{E} \to TM$ and $\mathfrak{a}_{_M}^T : M \times (\mathcal{E}\otimes \mathcal{E} \otimes \mathcal{E}^\ast) \to TM$
are defined by the action $(\varrho, \mathfrak{a}^T)$: for any $m \in M$, $e, e_1, e_2 \in \mathcal{E}$ and $\mu \in \mathcal{E}^\ast$,
\[\varrho_{_M}(m, e): = \varrho(e)(m), \quad \quad \mathfrak{a}_{_M}^T(m, e_1\otimes e_2 \otimes \mu) = \mathfrak{a}^T(e_1\otimes e_2 \otimes \mu)(m).\]
While the bracket is given, for any $x,y\in C^\infty(M, \mathcal{E})$, by the formula
\begin{equation}\label{bracket action Loday algebra}
\lbr x, y \rbr : = \lcf x, y \rcf_{_\mathcal{E}} + \mathcal{L}_{\varrho_{_M}(x)}y - \mathcal{L}_{\varrho_{_M}(y)}x + \mathfrak{a}_{_M}(dx)(y),
\end{equation}
where $\lcf x, y \rcf_{_\mathcal{E}}$ is the pointwise bracket of maps into $\mathcal{E}$, $\mathcal{L}_{\varrho_{_M}(x)}y$
(resp. $\mathcal{L}_{\varrho_{_M}(y)}x$) denotes the Lie derivative with respect to $\varrho_{_M}(x)$ (resp. $\varrho_{_M}(y)$)
of the vector valued function $y$ (resp. $x$), and $dx$ denotes the differential of $x$ which is a section of $T^\ast M \otimes (M\times \mathcal{E})$.
As in the Lie algebroids \cite{mck} and Courant algebroids \cite{li-bland-meinr} cases, if $x, y$ are constant maps, then $\lbr x, y \rbr$
is constant also, equal to the pointwise bracket in $\mathcal{E}$ of the values of $x$ and $y$. Hence, \eqref{bracket action Loday algebra} is an extension of the Loday bracket
on constant sections and a bidifferential operator of total order less or equal to 1 such that, for any $x\in C^\infty (M, \mathcal{E})$,
the adjoint map $\mathrm{ad}_x : C^\infty (M, \mathcal{E}) \to C^\infty (M, \mathcal{E})$, $y\mapsto \mathrm{ad}_x (y)= \lbr x, y \rbr$ is a
derivative endomorphism of $C^\infty (M, \mathcal{E})$.

\item\label{example: Loday action on projection}
\emph{Loday algebroid action on a projection.} Let $(E, \lcf \cdot, \cdot \rcf, \rho, \alpha)$ be a Loday algebroid over a smooth manifold $M_1$, $M_2$ a smooth manifold, and $\pi_i : M_1\times M_2 \to M_i$ the projection of $M_1\times M_2$ on the $i$-factor, $i=1,2$. We set $N=M_1\times M_2$ and we note that $TN \cong \pi_1^\ast TM_1 \oplus \pi_2^\ast TM_2$ and $T^\ast N \cong \pi_1^\ast T^\ast M_1 \oplus \pi_2^\ast T^\ast M_2$. We claim that the pair of maps $(\varrho , \mathfrak{a}^T)$, $\varrho : \Gamma(E) \to \Gamma(TN)$ and $\mathfrak{a}^T : \Gamma(E)\otimes \Gamma(E)\otimes \Gamma(E^\ast) \to \Gamma(TN)$, defined, for any $e, e_1, e_2\in \Gamma(E)$ and $\mu\in \Gamma(E^\ast)$, by
\begin{equation}\label{action pi_1}
\varrho(e) =\pi_1^\ast \rho(e) \quad \mathrm{and} \quad \mathfrak{a}^T(e_1 \otimes e_2 \otimes \mu) = \pi_1^\ast \alpha^T (e_1 \otimes e_2 \otimes \mu),
\end{equation}
is a Loday algebroid action of $E$ on $\pi_1$.
Before we proceed, we note the following. Any section $\eta$ of $T^\ast N$ is of the form
$\sum (f_i \otimes \eta^i_1) \oplus \sum (g_j \otimes \eta^j_2) = \sum (f_i \pi_1^\ast\eta^i_1) \oplus \sum (g_j \pi_2^\ast \eta^j_2)$, where $\eta^i_1\in \Gamma(T^\ast M_1)$, $\eta^i_2\in \Gamma(T^\ast M_2)$, and $f_i,g_j\in \CN$. Then, for all $e_1, e_2 \in \Gamma(E)$ and $\mu \in \Gamma(E^\ast)$, we have
\begin{eqnarray*}
\langle \mu, \mathfrak{a}(\eta)(e_1\otimes e_2)\rangle & = & \langle \eta, \mathfrak{a}^T(e_1 \otimes e_2 \otimes \mu) \rangle \nonumber \\
& \stackrel{\eqref{action pi_1}}{=}& \langle \sum (f_i \pi_1^\ast\eta^i_1) \oplus \sum (g_j \pi_2^\ast \eta^j_2), \pi_1^\ast \alpha^T (e_1 \otimes e_2 \otimes \mu)\rangle \nonumber \\
& = & \sum f_i \pi_1^\ast \langle \eta^i_1, \alpha^T (e_1 \otimes e_2 \otimes \mu) \rangle = \sum f_i \pi_1^\ast \langle \mu, \alpha^T (\eta^i_1)(e_1 \otimes e_2)) \rangle
\end{eqnarray*}
Hence,
\begin{equation}\label{alpha - pi_1}
\mathfrak{a}(\eta)(e_1\otimes e_2) = \sum f_i \pi_1^\ast\alpha^T (\eta_1^i)(e_1\otimes e_2) = \sum f_i \pi_1^\ast \alpha^T (\eta_1^i\otimes e_1\otimes e_2).
\end{equation}
Since
\begin{equation}\label{cond-on bracket of pullbacks}
[\pi_1^*X, \pi_1^* Y]_{TN}=\pi_1^* [X,Y]_{TM_1}, \quad \quad X,Y \in \Gamma(TM_1),
\end{equation}
for any $e_1,e_2 \in \Gamma(E)$, we have
\begin{eqnarray*}
\varrho(\lcf e_1, e_2 \rcf) & = & \pi_1^\ast \rho(\lcf e_1, e_2 \rcf)=\pi_1^\ast [\rho(e_1), \rho(e_2)]
\stackrel{\eqref{cond-on bracket of pullbacks}}{=}[\pi_1^\ast \rho(e_1), \pi_1^\ast \rho(e_2)] = [\varrho(e_1),\varrho(e_2)].
\end{eqnarray*}
The other two conditions \eqref{def-act lod alg - 3 - NEW}--\eqref{def-act lod alg - 2 - NEW} on $\varrho$ are easily verified. Regarding $\mathfrak{a}^T$, the relations \eqref{def-act lod alg - 5 -NEW}--\eqref{def-act lod alg - 4 -NEW} are also readily verified. Their proof relies on the results established above (in particular, the relation \eqref{alpha - pi_1}) and on the properties of the maps $(\rho,\alpha)$ of $E$ (see Lemma \ref{lemma-Raquel}).

\item\label{example: Courant alg action}
\emph{Courant algebroid action.} Let $(E, \lcf \cdot, \cdot \rcf, \rho, \langle \cdot, \cdot \rangle_{_E})$ be a Courant algebroid over $M$, $N$ a smooth manifold, and $\phi: N\to M$ a smooth map. Let $\varrho : \Gamma(E) \to \Gamma(TN)$ be a map satisfying the requirements \eqref{def-act lod alg - 1 - NEW}-\eqref{def-act lod alg - 2 - NEW} and, for any $e\in \Gamma(E)$ and $\eta\in \Gamma(T^*N)$, the claims:
\[\varrho \circ {\mathcal{G}_{_E}^\flat}^{-1}\circ \varrho^T = 0,\]
\[\lcf e, ({\mathcal{G}_{_E}^\flat}^{-1}\circ \varrho^T) \eta \rcf = ({\mathcal{G}_{_E}^\flat}^{-1}\circ \varrho^T) \mathcal{L}_{\varrho(e)}\eta \;\, \mathrm{and} \;\, \lcf ({\mathcal{G}_{_E}^\flat}^{-1}\circ \varrho^T) \eta, e\rcf = - ({\mathcal{G}_{_E}^\flat}^{-1}\circ \varrho^T)(i_{\varrho(e)}d\eta),\]
where $\varrho^T : \Gamma(T^\ast N) \to \Gamma(E^\ast)$ is the transpose of $\varrho$ and ${\mathcal{G}_{_E}^\flat} : \Gamma(E)\to \Gamma(E^\ast)$ is the map defined by the fiberwise nondegenerate symmetric bilinear form $\langle \cdot, \cdot \rangle_{_E}$ on $E$. Let, also, $\mathfrak{a}^T : \Gamma(E)\otimes \Gamma(E)\otimes \Gamma(E^\ast)\to \Gamma(TN)$ be the map given, for any $e_1,e_2\in \Gamma(E)$ and $\mu \in \Gamma(E^\ast)$, by
\begin{equation*}
\mathfrak{a}^T(e_1\otimes e_2 \otimes \mu) : = \phi^\ast(\langle e_1, e_2\rangle)\varrho({\mathcal{G}^\flat}^{-1}\mu).
\end{equation*}
Then, for any $f\in \CN$, $\langle df, \mathfrak{a}^T(e_1\otimes e_2 \otimes \mu)\rangle = \phi^\ast (\langle \mu, ({\mathcal{G}_{_E}^\flat}^{-1}\circ \varrho^T)df\rangle \langle e_1, e_2\rangle_{_E})$. A straightforward computation shows that the pair $(\varrho,\mathfrak{a}^T)$ defines an action of $E$ on $\phi$. Hence, according to Theorem \ref{theorem-action lod alg - new}, $(\varrho,\mathfrak{a}^T)$ defines a Loday algebroid structure
$(\lbr \cdot, \cdot \rbr, \varrho_{_N}, \mathfrak{a}_{_N})$ on $\phi^\ast E$. By introducing on $\phi^\ast E$ the fiberwise nondegenerate symmetric bilinear form $\langle \cdot, \cdot \rangle_{_{\phi^\ast E}}$ given, for any $f_1,f_2\in \CN$, $e_1,e_2\in \Gamma(E)$ and $y\in N$, by
\[\langle f_1\phi^\ast e_1, f_2\phi^\ast e_2 \rangle_{_{\phi^\ast E}}(y): = f_1(y)f_2(y)\phi^\ast \langle e_1, e_2 \rangle_{_E} (y),\]
it is straightforward to show that the action algebroid $(\phi^\ast E, \lbr \cdot, \cdot \rbr, \varrho_{_N}, \mathfrak{a}_{_N})$ is a Courant algebroid. For the proof of the above assertions we use, in addition to the supposed conditions, the commutativity of the following diagrams:
\begin{equation*}
\begin{tikzcd}
\Gamma(\phi^\ast E) \arrow{d}[swap]{\varrho_{_N}}] & \Gamma(E) \arrow[dl,swap, "\varrho"]\arrow{d}{\rho_{_E}} \arrow{l}[swap]{\phi^\ast}   \\
\Gamma(TN) \arrow[swap]{r}{d\phi} & \Gamma(TM),
\end{tikzcd}
\quad \quad
\begin{tikzcd}
\Gamma(\phi^\ast E^\ast ) & \arrow{l}[swap]{\phi^\ast}\Gamma(E^\ast) \\
\Gamma(T^\ast N) \arrow{u}{\varrho_{_N}^T} \arrow{ur}{\varrho^T} & \arrow{l}{d\phi^T}\arrow[swap]{u}{\rho^T}\Gamma(T^\ast M),
\end{tikzcd}
\quad \quad
\begin{tikzcd}
\Gamma(E ) \arrow[r, shift left, "\mathcal{G}_{_E}^\flat"] \arrow{d}[swap]{\phi^\ast} & \arrow{d}{\phi^\ast} \arrow[l, shift left, "\mathcal{G}_{_E}^{\flat^{-1}}"]\Gamma(E^\ast) \\
\Gamma(\phi^\ast E) \arrow[r, shift left, "\mathcal{G}_{_{\phi^\ast E}}^\flat"] & \arrow[l, shift left, "\mathcal{G}_{_{\phi^\ast E}}^{\flat^{-1}}"]\Gamma(\phi^\ast E^\ast),
\end{tikzcd}
\end{equation*}
where $\mathcal{G}_{_{\phi^\ast E}}^\flat : \Gamma(\phi^\ast E) \to \Gamma(\phi^\ast E^\ast)$ is the map defined by the product $\langle \cdot, \cdot \rangle_{_{\phi^\ast E}}$. \\
\noindent
We note that the notion of an action of a Courant algebroid was first introduced in \cite{li-bland-meinr} by Li-Bland and Meinrenken. Their construction differs from the one presented above.
\end{enumerate}

\begin{remark}\label{Loday alg action - Lie alg action}
\emph{We are in the setting of Example \ref{Loday to Lie}.
Then, we can easily verify that an action $(\varrho,\mathfrak{a}^T)$ of $E$ on $\phi\colon N\to M$ induces an action
$\bar{\varrho} : \bar{\mathcal{E}} \to \Gamma(TN)$, $\bar{e} \mapsto \bar{\varrho}(\bar{e}): = \varrho(e)$, of the Lie algebroid $\bar{E}$ on $\phi$.
Hence, according to \cite{hig-mck}, there is an associated Lie algebroid structure on the pullback vector bundle $\phi^\ast \bar{E}$
whose module of smooth sections is $\Gamma(\phi^\ast \bar{E})\cong \CN \otimes \Gamma(\bar{E})$.
Regarding $\CN$ as a $\C$-module under $f\cdot h = (f\circ \phi)h$, $f\in \C$ and $h\in \CN$,
we obtain, from \eqref{exact seq loday}, the split exact sequence}
\begin{equation*}\label{exact seq phi - loday}
0 \longrightarrow \CN \otimes \bar{\mathcal{E}_0} \longrightarrow \CN \otimes \mathcal{E} \longrightarrow \CN \otimes \bar{\mathcal{E}} \longrightarrow 0,
\end{equation*}
\emph{which is equivalent to the split exact sequence}
\begin{equation*}\label{exact seq phi - loday1}
0 \longrightarrow \Gamma(\phi^\ast E_0) \longrightarrow \Gamma(\phi^\ast E) \longrightarrow \Gamma(\phi^\ast \bar{E}) \longrightarrow 0.
\end{equation*}
\emph{Therefore, since $\Gamma(\phi^\ast E_0)$ coincides with the $\CN$-submodule of $\Gamma(\phi^\ast E)$ generated by $\lbr x,x\rbr$, where $x\in\Gamma(\phi^\ast E)$, the action Lie algebroid structure on $\phi^\ast\bar{E}$ constructed from $\bar{\varrho}$ is precisely the Lie algebroid structure induced by the action Loday algebroid
structure $(\lbr\cdot,\cdot\rbr,\varrho_{_N},\mathfrak{a}_{_N})$ on $\phi^\ast E$.}

\vspace{1mm}
\noindent
\emph{Setting $\mathcal{S}_{N}:=\CN\otimes\bar{\mathcal{E}}_0$, which is a submodule of $\Gamma(\phi^\ast E)$, the action Loday algebroid $(\phi^\ast E,\lbr\cdot,\cdot\rbr,\varrho_{_N},\mathfrak{a}_{_N})$ is a $\mathcal{S}_{N}$-algebroid in the sense of \cite{popescu}. The same conclusion holds if we remove from Definition \ref{def-action-loday alg - NEW} the compatibility conditions \eqref{cond-lemma-rq-1 - action}--\eqref{cond-lemma-rq-3 - action} relating $\mathfrak{a}^T$ to the bracket $\lcf\cdot,\cdot\rcf$ on $\Gamma(E)$. In this case, the resulting action algebroid is only a $\mathcal{S}_{N}$-algebroid.}
\end{remark}

\subsection{Loday algebroid comorphisms}\label{subsection - Loday alg comorphisms}
We introduce the concept of a \emph{Loday algebroid comorphism} covering maps between manifolds
by adapting the notion of Lie algebroid (co)morphism given in \cite{rui}, which is based on the following observation. Let $\pi_{_E} : E\to M$ and $\pi_{_F} : F\to N$ be
two vector bundles with dual bundles $\pi_{_{E^\ast}} : E^\ast \to M$ and $\pi_{_{F^\ast}} : F^\ast \to N$, respectively.
Any vector bundle map $(\hat{\Phi}, \phi)$,
\begin{equation*}
\begin{tikzcd}
F^\ast \arrow{r}{\hat{\Phi}} \arrow{d}[swap]{\pi_{_{F^\ast}}} & E^\ast \arrow{d}{\pi_{_{E^\ast}}}\\
N \arrow{r}[swap]{\phi}& M,
\end{tikzcd}
\end{equation*}
induces a map $\Phi : \Gamma(E) \to \Gamma(F)$ which assigns to each section $e\in \Gamma(E)$ the section $\Phi(e) \in \Gamma(F)$ defined, for any $y\in N$, by
\[\Phi(e)_y = \hat{\Phi}^T_ye_{\phi(y)},\]
where we have denoted by $\hat{\Phi}^T_y : E_{\phi(y)} \to F_y$ the transpose of the linear map $\hat{\Phi}_y : F^\ast_y \to E^\ast _{\phi(y)}$ on the fibers.

\vspace{1mm}
\noindent
Since any ring homomorphism $\phi^\ast:\C\longrightarrow \CN$, defines a map $\phi : N\to M$ \cite{{jet-nestr}},
the above identity provides a one-to-one correspondence between the set of module morphisms
$\Phi : \Gamma(E)\longrightarrow\Gamma(F)$ over a ring homomorphism $\phi^\ast :\C\longrightarrow \CN$,
and the set of vector bundle morphisms $(\hat{\Phi},\phi):F^*\longrightarrow E^*$.

\begin{definition}\label{def-Loday-comorphism}
Let $(E, \lcf \cdot, \cdot \rcf_{_E}, \rho_{_E}, \alpha_{_E})$ and $(F, \lcf \cdot, \cdot \rcf_{_F}, \rho_{_F}, \alpha_{_F})$ be two Loday algebroids over $M$ and $N$, respectively.
A \emph{Loday algebroid comorphism} from $E$ to $F$ is a vector bundle map $\hat{\Phi} : F^\ast \to E^\ast$ over $\phi : N \to M$ such that:
\begin{enumerate}
\item
The induced map $\Phi : \Gamma(E) \to \Gamma(F)$ is a Loday algebra homomorphism: For any $e_1,e_2 \in \Gamma(E)$,
\begin{equation}\label{def - Loday algebra homomorphism}
\Phi(\lcf e_1, e_2 \rcf_{_E}) = \lcf \Phi(e_1), \Phi(e_2) \rcf_{_F}.
\end{equation}
\item
For any $e\in \Gamma(E)$, the vector fields $\rho_{_F}(\Phi(e))$ and $\rho_{_E}(e)$ are $\phi$-related: For any $f\in \C$,
\begin{equation}\label{comorphism - left anchor}
\rho_{_F}(\Phi(e)) (f\circ \phi) = \rho_{_E}(e)(f)\circ \phi.
\end{equation}
\end{enumerate}
\end{definition}

\vspace{1mm}
\noindent
An immediate consequence of \eqref{def - Loday algebra homomorphism} and \eqref{comorphism - left anchor} concerning the maps $\alpha_{_E}$ and $\alpha_{_F}$ is the following.
Taking into account that $(\Phi, \phi^\ast)$ is an homomorphism from the $\C$-bimodule $\Gamma(E)$ to the $\CN$-bimodule $\Gamma(F)$, namely, for any $e\in \Gamma(E)$ and $f\in \C$,
\begin{equation}\label{homomorphism Phi}
\Phi(fe) = \phi^\ast f\Phi(e),
\end{equation}
and applying the identity \eqref{def - Loday algebra homomorphism} to the pair $(fe_1,e_2)$, $f\in \C$ and $e_1, e_2 \in \Gamma(E)$,
we get that the module maps $\alpha_{_E}^T : \Gamma(T^\ast M) \otimes \Gamma(E) \otimes \Gamma(E) \to \Gamma(E)$ and
$\alpha_{_F}^T : \Gamma(T^\ast N) \otimes \Gamma(F) \otimes \Gamma(F) \to \Gamma(F)$ induced,
respectively, by the corresponding vector bundle maps $\alpha^T_{_E}$ and $\alpha^T_{_F}$ (see \eqref{alpha - alpha t}), are related as follows:
\begin{equation*}
\alpha_{_F}^T (\phi^\ast df \otimes \Phi(e_1)\otimes \Phi(e_2)) = \Phi(\alpha^T_{_E}(df \otimes e_1 \otimes e_2)).
\end{equation*}
The last equation means that the following diagram is commutative:
\begin{equation*}
\begin{CD}
\Gamma(T^\ast M) \otimes \Gamma(E) \otimes \Gamma(E) @>\phi^\ast \otimes \Phi \otimes \Phi >> \Gamma(T^\ast N)\otimes \Gamma(F) \otimes \Gamma(F)\\
@V \alpha_{_E}^T VV @VV \alpha_{_F}^T V\\
\Gamma(E) @>>\Phi> \Gamma(F).
\end{CD}
\end{equation*}

\vspace{1mm}
\noindent
In the particular case where $(E, \lcf \cdot, \cdot \rcf_{_E}, \rho_{_E}, \alpha_{_E})$ and $(F, \lcf \cdot, \cdot \rcf_{_F}, \rho_{_F}, \alpha_{_F})$
are Loday algebroids over the same manifold $M$ or $\phi : N \to M$ is a diffeomorphism,
then a comorphism $\hat{\Phi} : F^\ast \to E^\ast$ over the identity or over $\phi$ induces a map between the spaces
of smooth sections $\Gamma(F^\ast) \to \Gamma(E^\ast)$, also denoted by $\hat{\Phi}$, and the induced map $\Phi : \Gamma(E) \to \Gamma(F)$ is the transpose of $\hat{\Phi}$.

\vspace{1mm}
\noindent
Regarding the composition of Loday algebroid comorphisms, we can easily establish the following proposition.

\begin{proposition}\label{composition Loday comorphisms}
Let $(E_i, \lcf \cdot, \cdot \rcf_{_{E_i}}, \rho_{_{E_i}}, \alpha_{_{E_i}})$ be a Loday algebroid over a smooth manifold $M_i$ for any $i=1,2,3$. If $(\hat{\Phi}, \phi)$, $\hat{\Phi} : E_2^\ast \to E_1^\ast$, is a comorphism of Loday algebroids over $\phi: M_2 \to M_1$, and $(\hat{\Psi}, \psi)$, $\hat{\Psi} : E_3^\ast \to E_2^\ast$, is a comorphism of Loday algebroids over $\psi: M_3 \to M_2$, then $(\hat{\Phi}\circ \hat{\Psi}, \phi\circ \psi)$, $\hat{\Phi}\circ \hat{\Psi} : E_3^\ast \to E_1^\ast$, is also a comorphism of Loday algebroids over $\phi\circ \psi$.
\end{proposition}

\begin{proposition}
Let $(E, \lcf \cdot, \cdot \rcf_{_E}, \rho_{_E}, \alpha_{_E})$ and $(F, \lcf \cdot, \cdot \rcf_{_F}, \rho_{_F}, \alpha_{_F})$
be two Loday algebroids over $M$ and $N$, respectively, $(\bar{E}, [\cdot, \cdot]_{\bar{E}}, \bar{\rho}_{_E})$ and $(\bar{F}, [\cdot, \cdot]_{\bar{F}}, \bar{\rho}_{_F})$
the corresponding Lie algebroids (under the appropriate assumptions, see paragraph \ref{Loday to Lie}),
and $(\hat{\Phi},\phi)$ a Loday algebroid comorphism. Then, $\Phi : \Gamma(E) \to \Gamma(F)$ induces a Lie
algebroid (co)morphism $\bar{\Phi} : \Gamma(\bar{E}) \to \Gamma(\bar{F})$ over $\phi$, in the sense of \cite{rui}.
\end{proposition}
\begin{proof}
Because $\Phi(\bar{\mathcal{E}}_0) = \Phi(\C \cdot \mathcal{E}_0 ) = \phi^\ast (\C)\cdot \Phi(\mathcal{E}_0) \subset \bar{\mathcal{F}}_0$, $\Phi$ induces a morphism $\bar{\Phi} : \Gamma(\bar{E}) \to \Gamma(\bar{F})$, $\bar{\Phi}(\bar{e}): = \overline{\Phi(e)}$, which is a homomorphism of Lie algebras:
\[\bar{\Phi}([\bar{e}_1, \bar{e}_2]_{\bar{E}}) = [\bar{\Phi}(\bar{e}_1), \bar{\Phi}(\bar{e}_2)]_{\bar{F}},\quad \quad \bar{e}_1, \bar{e}_2\in \Gamma(\bar{E}),\]
and
\[d\phi \circ \bar{\rho}_{_F} \circ \bar{\Phi} = \bar{\rho}_{_E}.\]
\end{proof}

\vspace{2mm}
\noindent
In what follows, we work within the framework of Theorem \ref{theorem-action lod alg - new}.
The pullback map $\phi^* : \Gamma(E)\to \Gamma(\phi^*E)$ (over the ring homomorphism $\phi^*:\C\to \CN$) corresponds to the vector bundle morphism
$\mathrm{pr} : \phi^*E^*\to E^*$ over $\phi : N\to M$:
\begin{equation*}
\begin{tikzcd}
\phi^\ast E^\ast \arrow{r}{\mathrm{pr}}\arrow[swap]{d} \arrow{d}& E^\ast \arrow{d}\\
N \arrow{r}[swap]{\phi}& M,
\end{tikzcd}
\end{equation*}
where $\mathrm{pr}: \phi^\ast E^\ast \to E^\ast $ is the canonical projection of $\phi^\ast E^\ast \subseteq N\times E^\ast$ on the second factor.
This vector bundle map is a Loday algebroid comorphism.

\vspace{2mm}
\noindent
We continue to work in the above framework and consider:
\begin{enumerate}
\item
a Loday algebroid $(F, \lcf \cdot, \cdot \rcf_{_F}, \rho_{_F}, \alpha_{_F})$ over $N$ and the dual bundle $F^\ast$ of $F$;
\item
a vector bundle map $\hat{\Phi} : F^\ast \to E^\ast$ over $\phi : N \to M$ and the corresponding map between the modules of smooth sections $\Phi : \Gamma(E) \to \Gamma(F)$.
\end{enumerate}
By the universal property for vector bundle morphisms over $\phi : N \to M$ into $E^\ast$ \cite{mck},
there exists a unique vector bundle morphism $\hat{\Psi}: F^\ast \to \phi^\ast E^\ast$ over the identity $\mathrm{id}_{_N}$ such that $\hat{\Phi} =\mathrm{pr} \circ \hat{\Psi}$. We denote by $\Psi : \Gamma(\phi^\ast E) \to \Gamma(F)$ the corresponding map on the spaces of smooth sections.

\begin{proposition}\label{prop-Phi-Psi}
Under the above assumptions, the vector bundle map $(\hat{\Phi}, \phi)$ is a comorphism of Loday algebroids if the vector bundle map $(\hat{\Psi}, \mathrm{id}_{_N})$
is a comorphism of Loday algebroids over the identity map of $N$.
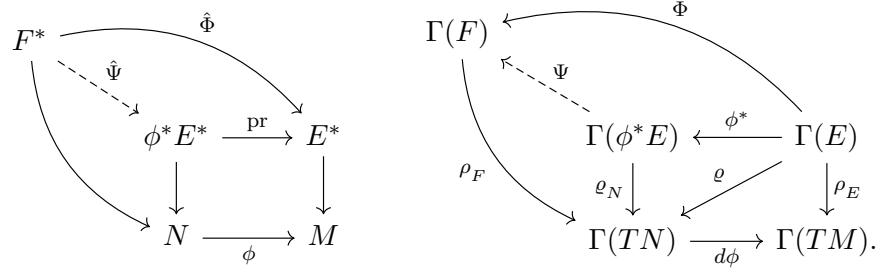
\begin{equation*}
\begin{tikzcd}
F^\ast  \arrow[ddr,bend right] \arrow[drr, bend left,"\hat{\Phi}"] \arrow[dr,dashed, "\hat{\Psi}"] \\
& \phi^\ast E^\ast \arrow{d}[swap]{} \arrow[r, "\mathrm{pr}"] & E^\ast \arrow{d}{}  \\
& N \arrow[swap]{r}{\phi} & M
\end{tikzcd}
\quad \quad
\begin{tikzcd}
\Gamma(F) \arrow[bend right,swap]{ddr}{\rho_{_F}} \\
& \Gamma(\phi^\ast E) \arrow{d}[swap]{\varrho_{_N}} \arrow[ul, swap, dashed,"\Psi"] & \Gamma(E) \arrow[dl,swap, "\varrho"]\arrow{d}{\rho_{_E}} \arrow{l}[swap]{\phi^\ast} \arrow[bend right,swap]{ull}{\Phi}  \\
& \Gamma(TN) \arrow[swap]{r}{d\phi} & \Gamma(TM).
\end{tikzcd}
\end{equation*}
\end{proposition}
\begin{proof}
We assume that $(\hat{\Psi}, \mathrm{id}_{_N})$ is a Loday algebroid comorphism and we see
the pair $(\hat{\Phi}, \phi)$ as the composition of $(\mathrm{pr}, \phi)$ with $(\hat{\Psi}, \mathrm{id}_{_N})$,
$\hat{\Phi} = \mathrm{pr}\circ \hat{\Psi}$. Hence, according to Proposition \ref{composition Loday comorphisms}, it is a Loday algebroid comorphism.
\end{proof}

\subsection{Loday algebroid morphisms in the sense of Higgins-Mackenzie}
In this paragraph, we introduce the concept of \emph{Loday algebroid morphisms} covering maps between manifolds by adapting the notion of Lie algebroid morphisms from \cite{hig-mck,mck}, which is based on the following observations. Given a vector bundle $\pi_E : E \to M$ and a smooth map $\psi : M' \to M$,
we consider the pullback vector bundle $\psi^\ast E$ of $E$ over $M'$ and the canonical projection $\mathrm{pr} : \psi^\ast E \to E$
of $\psi^\ast E \subset M' \times E$ to the second factor. It is well known that $\Gamma(\psi^\ast E) \cong C^\infty(M',\R)\otimes\Gamma(E)$ and that
the pair $(\psi^\ast E, \mathrm{pr})$
has the universal property for vector bundle morphisms over $\psi$ into $E$,
in the sense that, if $\Psi: E' \to E$ is any vector bundle morphism over $\psi$, then there is a unique vector
bundle map $\Psi^\ast : E' \to \psi^\ast E$ over the $\mathrm{id}_{M'}$ such that $\Psi = \mathrm{pr}\circ \Psi^\ast$.
Thus, given $e'\in \Gamma(E')$, for suitable $h^i\in C^\infty(M',\R)$ and $e_i\in \Gamma(E)$, we can write
\begin{equation}\label{Psi-decomposition}
\Psi^\ast(e') = \sum_{i} h^i\otimes e_i.
\end{equation}
The equation \eqref{Psi-decomposition} is referred as the \emph{$\Psi$-decomposition of $e'$}.
While, if $\Psi^\ast(e') = 1 \otimes e$, for $e\in \Gamma(E)$, we say that \emph{$e'$ is $\Psi$-related to $e$}.

\vspace{2mm}
\noindent
Let $(E, \lcf \cdot, \cdot \rcf, \rho, \alpha)$ and $(E', \lcf \cdot, \cdot \rcf', \rho', \alpha')$ be Loday
algebroids on base manifolds $M$ and $M'$, respectively. We construct, as in the paragraph \emph{The Lie algebroid of a Loday algebroid} \ref{Loday to Lie},
the submodules $\bar{\mathcal{E}}_0$ of $\Gamma(E)$ and $\bar{\mathcal{E}}_0'$ of $\Gamma(E')$, and we view
$(E, \lcf \cdot, \cdot \rcf, \rho, \alpha)$ as a $\bar{\mathcal{E}}_0$-algebroid and
$(E', \lcf \cdot, \cdot \rcf', \rho', \alpha')$ as a $\bar{\mathcal{E}}_0'$-algebroid. Then, in this framework, we define as in \cite{popescu}:

\begin{definition}\label{def-Loday morphism Mac-Higg}
A \emph{morphism of Loday algebroids} is a vector bundle map $(\Psi, \psi)$,
\begin{equation*}
\begin{tikzcd}
E' \arrow{r}{\Psi}\arrow[swap]{d} \arrow{d}[swap]{\pi_{_{E'}}} & E \arrow{d}{\pi_{_E}}\\
M' \arrow{r}[swap]{\psi}& M,
\end{tikzcd}
\end{equation*}
such that
\begin{enumerate}
\item
$d\psi \circ \rho' = \rho \circ \Psi$,
\item
for any $e_1', e_2'\in \Gamma(E')$ with $\Psi$-decompositions
\begin{equation*}
\Psi^\ast(e_1') = \sum_{i} h_1^i\otimes e_i, \quad \quad \Psi^\ast(e_2') = \sum_{j} h_2^j\otimes e_j,
\end{equation*}
\begin{equation*}\label{mor-Loday Psi Mac-Higg}
\Psi^\ast(\lcf e_1', e_2' \rcf') - \sum_{i,j}h_1^ih_2^j\otimes \lcf e_i, e_j \rcf -
\sum_{j}\rho'(e_1')(h_2^j)\otimes e_j + \sum_{i}\rho'(e_2')(h_1^i)\otimes e_i
\in C^\infty(M', \R)\otimes\bar{\mathcal{E}}_0,
\end{equation*}
requiring,
\item
for any $f\in \C$ and $e_1', e_2' \in \Gamma(E')$ which are $\Psi$-related to $e_1, e_2 \in \Gamma(E)$, respectively,
\begin{equation*}\label{requir - Psi - mor HM}
\Psi^\ast(\alpha'(d(\psi^\ast f) \otimes e_1')(e_2')) = \psi^\ast (\alpha(df \otimes e_1)(e_2)).
\end{equation*}
\end{enumerate}
\end{definition}

\begin{remarks}
\end{remarks}
\begin{enumerate}
\item
If the submodules $\bar{\mathcal{E}}_0$ and $\bar{\mathcal{E}}_0'$ are
the modules of sections of vector subbundles $E_0$ and $E_0'$ of $E$ and $E'$, respectively, and if $\Psi(E_0') \subset E_0$, we have that $\Psi$
induces a vector bundle map $\bar{\Psi} : \bar{E}' \to \bar{E}$, where $\bar{E}' = E'/E_0'$ and $\bar{E}=E/E_0$,
which is a Lie algebroid morphism, from $(\bar{E}', [\cdot, \cdot]', \bar{\rho}')$ to $(\bar{E}, [\cdot, \cdot], \bar{\rho})$,
in the sense of Higgins-Mackenzie \cite{hig-mck}.
\item
If $\Psi: E'\to E$ is a fiberwise bijection over $\psi : M'\to M$, it has the universal property, hence $E'$
is identified with $\psi^\ast E$ and the vector bundle map $\Psi$ is identified with $\mathrm{pr}: \psi^\ast E \to E$,
which, in the context of Theorem \ref{theorem-action lod alg - new}, is a Loday algebroid comorphism.
\end{enumerate}

\subsection{Loday algebroid cohomology}\label{subsection Loday cohomology}
Let $(E, \lcf \cdot, \cdot \rcf, \rho, \alpha)$ be a Loday algebroid over a smooth manifold $M$ and $\mathcal{D}^p(\mathcal{E})$, $p\in \N$, the space of $p$-differential operators on the $\C$-bimodule $\mathcal{E} = \Gamma(E)$ with values in $\C$:
\begin{equation*}
D : \underbrace{\mathcal{E} \times \cdots \times \mathcal{E}}_{p - times} \to \C.
\end{equation*}
By convention, $\mathcal{D}^0(\mathcal{E})=\C$. We endow the space $\mathcal{D}(\mathcal{E}): = (\mathcal{D}^p(\mathcal{E}))_{p\in \N}$
with the \emph{suffle product} $\diamond$ defined in \cite{eil-ml} as follows.
For any $D\in \mathcal{D}^p(\mathcal{E})$ and $D'\in \mathcal{D}^q(\mathcal{E})$, we define the $(p+q)$-differential operator
$D \diamond D'$ by setting, for any $e_1,\ldots , e_{p+q} \in \mathcal{E}$,
\begin{equation*}\label{diamond prod}
(D \diamond D')(e_1, \ldots, e_{p+q}) = \sum_{\tau \in Sh(p,\,q)} (-1)^{|\tau|}D(e_{ \tau(1)},\ldots, e_{ \tau(p)})D'(e_{ \tau(p+1)},\ldots, e_{ \tau(p+q)}),
\end{equation*}
where $Sh(p,q)$ is the set of $(p,q)$-shuffle permutations of $(1,\ldots, p+q)$, i.e., of permutations $\tau$
such that $\tau(1)< \ldots < \tau(p)$ and $ \tau(p+1)<\ldots < \tau(p+q)$, and $(-1)^{|\tau|}$ is the signature of $\tau$.

\begin{proposition}\cite{grab-k-pon}
The pair $(\mathcal{D}(\mathcal{E}), \diamond)$ is a graded commutative associative unital $\R$-algebra, referred as \emph{shuffle algebra of $E$.}
\end{proposition}

\vspace{2mm}
\noindent
In the above framework, the data $(\lcf \cdot, \cdot \rcf, \rho)$ defines the map
\[\partial_{E} : \mathcal{D}^{\bullet}(\mathcal{E}) \to \mathcal{D}^{\bullet +1}(\mathcal{E})\]
given, for any $e_1, \ldots, e_{p+1}\in \Gamma(E)$ and $D\in \mathcal{D}^p(E)$, by
\begin{eqnarray} \label{def:Loday:cohomology:operator}
\partial_E D (e_1, \ldots, e_{p+1}) & = & \sum_{i=1}^{p+1}(-1)^{i+1}\rho (e_i)D(e_1, \ldots, \hat{e}_i, \ldots, e_{p+1}) \nonumber  \\
& & + \, \sum_{i<j}(-1)^i D(e_1, \ldots, \hat{e}_i, \ldots, \hat{e}_j, \lcf e_i, e_j\rcf, e_{j+1}, \ldots, e_{p+1}),
\end{eqnarray}
where $\hat{\cdot}$ indicates that the corresponding term is omitted. In particular, for any $f\in \mathcal{D}^0(\mathcal{E})=\C$,
\begin{equation*}\label{der_E f}
\partial_{E}f(e)=\rho(e)(f),
\end{equation*}
whence we conclude that
\begin{equation}\label{der_E f - transpose}
\partial_{E} = \rho^T \circ d : \C \to \Gamma(E^\ast) \subset \mathcal{D}^1(\mathcal{E}).
\end{equation}

\begin{remark}\label{comparaison courant algebra}
\emph{In \cite{grab-k-pon} was also introduced the notion of \emph{reduced shuffle algebra} $\mathrm{D}^\bullet(E)$ of a Loday algebroid $E$ which concerns the multidifferential operators that are
linear with respect to the last variable. In our context, $\partial_E$ leaves invariant the $\mathrm{D}(E)$. In the case where $E$ is a Courant algebroid, the reduced shuffle algebra
coincide with the Keller-Waldmann algebra of $E$ \cite{keller-waldmann}.}
\end{remark}

\begin{proposition}\cite{grab-k-pon}
The operator $\partial_E$ is a degree $1$ graded derivation of the shuffle algebra $(\mathcal{D}(\mathcal{E}), \diamond)$ of $E$ and squares to zero. I.e., for any $D\in \mathcal{D}^p(\mathcal{E})$ and $D'\in \mathcal{D}^q(\mathcal{E})$,
\begin{equation}\label{d_E cobound}
\partial_E(D\diamond D') = \partial_E D\diamond D' + (-1)^p D \diamond \partial_ED' \quad \quad \mathrm{and} \quad \quad \partial_E^2 = 0.
\end{equation}
So, $\partial_E$ is a coboundary operator.
\end{proposition}

\begin{definition} \cite{grab-k-pon}\label{def:Loday:algebroid:cohomology}
Let $(E, \lcf \cdot, \cdot \rcf, \rho, \alpha)$ be a Loday algebroid over a smooth manifold $M$. The cohomology of the Loday cochain complex $(\mathcal{D}^{\bullet}(\mathcal{E}),\partial_{E})$, associated with the Loday algebra structure $\lcf \cdot, \cdot \rcf$ on $\mathcal{E}$ represented by $\rho$ on $\C$, is called the \emph{Loday algebroid cohomology of $E$} and its $p^{th}$-cohomology group is denoted by $H^p(\mathcal{E})$.
\end{definition}
We note that, for $p=0$, $H^0(\mathcal{E})$ is the set of the functions $f\in \C$ which are constants on the leaves of the involutive distribution $\mathrm{Im}\rho$.

\vspace{2mm}
\noindent
For any two $\partial_E$-closed differential operators $D_1\in \mathcal{D}^p(\mathcal{E})$ and $D_2\in \mathcal{D}^q(\mathcal{E})$, we verify that
\[(D_1 + \partial_E D_1' )\diamond (D_2 + \partial_E D_2')\stackrel{\eqref{d_E cobound}}{=} D_1\diamond D_2 + \partial_E (D_1' \diamond D_2 + D_1'\diamond \partial_E D_2' + (-1)^p D_1 \diamond D_2'),\]
where $D_1'\in \mathcal{D}^{p-1}(\mathcal{E})$ and $D_2'\in \mathcal{D}^{q-1}(\mathcal{E})$. Thus, the cohomology class $[D_1 \diamond D_2]$ of $D_1 \diamond D_2$ depends only on $[D_1]$ and $[D_2]$. We denote it by $[D_1 \diamond D_2]: = [D_1]\diamond [D_2]$, and call it the \emph{product} of $[D_1]$ and $[D_2]$. Therefore, the product $\diamond$
defined above endows $H^{\bullet}(\mathcal{E})$ with an algebra structure.

\vspace{2mm}
\noindent
For the graded shuffle algebra $(\mathcal{D}^{\bullet}(\mathcal{E}),\partial_{E})$ we can develop the corresponding Cartan calculus. Let $[\cdot, \cdot]_{com}$
be the graded commutator on the space of graded endomorphisms of $\mathcal{D}^{\bullet}(\mathcal{E})$. If $P$ and $Q$ are two graded endomorphisms of degree $p$ and $q$, respectively, then the graded endomorphism
\[[P,Q]_{com} = P\circ Q - (-1)^{pq}Q\circ P \]
is of degree $p+q$. For any $e\in \mathcal{E}$, we introduce the operators:
\begin{itemize}
\item[-]
\emph{interior product} $i_e : \mathcal{D}^{p}(\mathcal{E}) \to \mathcal{D}^{p-1}(\mathcal{E})$, $D \mapsto i_eD$, defined, for any $e_1, \ldots, e_{p-1} \in \mathcal{E}$, by
\begin{equation*}
(i_eD)(e_1, \ldots, e_{p-1}) = D(e, e_1, \ldots, e_{p-1}),
\end{equation*}
which is a graded derivation of $(\mathcal{D}^{\bullet}(\mathcal{E}),\partial_{E})$ of degree $-1$, i.e.,
for any $D\in \mathcal{D}^{p}(\mathcal{E})$ and $D'\in \mathcal{D}^{\bullet}(\mathcal{E})$, satisfies the Leibniz rule:
\[i_e(D\diamond D') = (i_eD)\diamond D' + (-1)^p D \diamond (i_e D').\]
\item[-]
\emph{Lie derivative} $\mathcal{L}_e : \mathcal{D}^{p}(\mathcal{E}) \to \mathcal{D}^{p}(\mathcal{E})$,
$\mathcal{L}_e = [\partial_E, i_e]_{com} = \partial_E \circ i_e + i_e \circ \partial_E$, which is of degree 0.
\end{itemize}
\begin{lemma}\cite{grab-k-pon}\label{lemma Cartan}
The following Cartan's commutation relations hold in the space of graded endomorphisms of $\mathcal{D}^{\bullet}(\mathcal{E})$:
\begin{enumerate}
\item
$[\partial_E, \partial_E]_{com} = 2\partial_E^2 = 0$;
\item
$[\partial_E, \mathcal{L}_e]_{com} = 0$;
\item
$[\mathcal{L}_{e_1}, i_{e_2}]_{com}=i_{\lcf e_1, e_2\rcf}$;
\item
$[\mathcal{L}_{e_1}, \mathcal{L}_{e_2}]_{com}=\mathcal{L}_{\lcf e_1, e_2\rcf}$.
\end{enumerate}
\end{lemma}

\vspace{1mm}
\noindent
The properties of the operator $\mathcal{L}$ allow us to define
a Loday algebroid structure on the Whitney sum $E\oplus E^\ast$.
\begin{example}
\emph{Let $(E, \lcf \cdot, \cdot \rcf, \rho, \alpha)$ be a Loday algebroid over $M$ and $E^\ast$ its dual vector bundle.
Then the triple $(\lcf \cdot, \cdot \rcf_{_{E\oplus E^\ast}}, \rho_{_{E\oplus E^\ast}}, \alpha_{_{E\oplus E^\ast}})$, where, for any
$e+\eta, e_1 + \eta_1, e_2 + \eta_2 \in \Gamma(E\oplus E^\ast)$ and $f\in \C$, $\rho_{_{E\oplus E^\ast}}(e+\eta) = \rho(e)$, }
\[\lcf e_1 + \eta_1, e_2 + \eta_2 \rcf_{_{E\oplus E^\ast}} = \lcf e_1, e_2 \rcf + \mathcal{L}_{e_1}\eta_2, \]
and
\[\alpha_{_{E\oplus E^\ast}} (df \otimes (e_1 + \eta_1))(e_2 + \eta_2) = \alpha(df\otimes e_1)(e_2) +
\rho(e_2)(f)\eta_1 + \eta_2(e_1)\partial_Ef - \eta_2(\alpha(df \otimes e_1)(\cdot)),\]
\emph{defines a Loday algebroid structure on $E\oplus E^\ast$.}
\end{example}

\vspace{1mm}
\noindent
In the following, we study the behavior of the coboundary operator in relation to a comorphism of Loday algebroids.
Let $(E, \lcf \cdot, \cdot \rcf_{_E}, \rho_{_E}, \alpha_{_E})$ and $(F, \lcf \cdot, \cdot \rcf_{_F}, \rho_{_F}, \alpha_{_F})$ be two Loday algebroids over $M$ and $N$,
respectively, $(\hat{\Phi}, \phi)$ a Loday algebroid comorphism from $E$ to $F$ over a
diffeomorphism $\phi: N \to M$, and $\Phi : \Gamma(E) \to \Gamma(F)$ the induced map on the modules of smooth sections.
We restrict our attention to the case where $\phi$ is a diffeomorphism, because the \emph{pullback differential operator} can be defined only if its symbol is evaluated on $1$-forms on $N$ of the type $\eta=\phi^\ast\xi$, with $\xi\in\Gamma(T^\ast M)$. This is possible only when $\phi$ is a diffeomorphism.
We consider also the Loday cochain complexes $(\mathcal{D}^{\bullet}(\mathcal{E}),\partial_{E})$ and $(\mathcal{D}^{\bullet}(\mathcal{F}),\partial_{F})$,
where $\mathcal{F}$ is the $\CN$-bimodule $\Gamma(F)$. The map $\Phi$ induces, for any $p\in \N^{^\ast}$:
\begin{itemize}
\item[-]
a map $\Phi^p$ between the product modules $\underbrace{\mathcal{E}\times \ldots \times \mathcal{E}}_{p-times}$ and $\underbrace{\mathcal{F}\times \ldots \times \mathcal{F}}_{p-times}$ given by
\[(e_1, \ldots, e_p)\mapsto \Phi^p(e_1, \ldots, e_p): = (\Phi(e_1), \ldots, \Phi(e_p));\]
\item[-]
a map $\Phi^\star_p : \mathcal{D}^p(\mathcal{F}) \to \mathcal{D}^p(\mathcal{E})$ which assigns to each $p$-differential
operator $D\in \mathcal{D}^p(\mathcal{F})$ a $p$-differential operator $\Phi^\star_p D$
on $\mathcal{E}$ defined in such a way that the following diagram is commutative:
\begin{equation*}
\begin{tikzcd}
\underbrace{\mathcal{E}\times \ldots \times \mathcal{E}}_{p-times} \arrow{r}{\Phi^p} \arrow[swap]{d}{\Phi^\star_pD} & \underbrace{\mathcal{F}\times \ldots \times \mathcal{F}}_{p-times} \arrow{d}{D} \\
\C \arrow{r}[swap]{\phi^\ast} & \CN.
\end{tikzcd}
\end{equation*}
I.e., for any $p$-tuple $(e_1, \ldots, e_p)$ of elements of $\mathcal{E}$,
\begin{eqnarray}\label{def - Phi^star}
\lefteqn{(\phi^\ast \circ \Phi_p^\star D)(e_1, \ldots, e_p) = (D\circ \Phi^p)(e_1, \ldots, e_p) \Leftrightarrow}\nonumber \\
 & \Phi_p^\star D(e_1, \ldots, e_p)\circ \phi = D(\Phi(e_1), \ldots, \Phi(e_p)).
\end{eqnarray}
\end{itemize}
While, for $p=0$, since $\phi$ is a diffeomorphism, for any $g\in \CN$ there exists unique $f\in \C$ such that $g=\phi^\ast f$, then we
define $\Phi_0^\star : \CN \to \C$ to be the map which assigns to each $g\in \CN$ the unique $f\in \C$
such that $g = \phi^\ast f$, that is, $\Phi_0^\star = (\phi^{-1})^\ast$. Also, we can easily show that
\begin{equation*}\label{Phi^star f}
\Phi^\star _p (\phi^\ast f D) = f\Phi^\star _p D, \quad \quad f\in \C.
\end{equation*}

\vspace{2mm}
\noindent
Because of \eqref{homomorphism Phi}, $\Phi$ can be considered as a $0$-order differential operator from $\mathcal{E}$ to $\mathcal{F}$ and
its composition with a $p$-differential operator $D\in \mathcal{D}^p(\mathcal{F})$ produces a $p$-differential operator $\Phi_p^\star D$ on $\mathcal{E}$ such that:
\begin{enumerate}
\item
$\Phi_p^\star(\sigma_i^D(\phi^\ast f)) = \sigma_i^{\Phi_p^\star D}(f)$, i.e. the map $\Phi^\star_p$ sends the $i$-symbol
of $D$ to the $i$-symbol of $\Phi^\star_pD$, $i=1,\ldots,p$.
\item
$D$ and $\Phi_p^\star D$ are differential operators of the same order with respect to the corresponding arguments.
\end{enumerate}
The above assertions can be proved directly.

\begin{proposition}\label{Phi LA morph --> Phi* chain map}
The family $\Phi^\star = (\Phi_p^\star )_{p\in \N}$ of maps is a homomorphism from the graded shuffle algebra $((\mathcal{D}^p(\mathcal{F}))_{p\in \N}, \diamond)$ of $F$
to the graded shuffle algebra $((\mathcal{D}^p(\mathcal{E}))_{p\in \N}, \diamond)$ of $E$ which commute with the coboundary operators:
\begin{equation}\label{Phi^star commutation}
\partial_E \circ \Phi^\star_p = \Phi_{p+1}^\star \circ \partial_F.
\end{equation}
Thus, $\Phi^\star = (\Phi_p^\star )_{p\in \N}$ induces a sequence of maps, also denoted by $\Phi^\star = (\Phi_p^\star )_{p\in \N}$,
such that $\Phi_p^\star$ is a homomorphism from the $p^{th}$-cohomology group $H^p (\mathcal{F})$ of
$(\mathcal{D}(\mathcal{F}), \partial_F)$ to the $p^{th}$-cohomology group $H^p(\mathcal{E})$
of $(\mathcal{D}(\mathcal{E}), \partial_E)$.
\end{proposition}
\begin{proof}
The first claim is proved by a straightforward verification of the equality,
for any $D\in \mathcal{D}^p(\mathcal{F})$ and $D'\in \mathcal{D}^q(\mathcal{F})$, $p,q \in \N$,
\begin{equation*}\label{Phi* - p - q}
\Phi_{p+q}^\star(D \diamond D') = \Phi_p^\star D \diamond \Phi_q^\star D'.
\end{equation*}
The second claim follows also from a straightforward verification.
It means that $\Phi_p^\star$ maps cocycles to cocycles and coboundaries to coboundaries.
Thus, it induces a homomorphism $\Phi_p^\star: H^p (\mathcal{F}) \to H^p (\mathcal{E})$.
\end{proof}

\vspace{2mm}
\noindent
The following proposition proves the converse of the above result.

\begin{proposition}\label{Phi* chain map --> Phi LA morph}
Let $(E, \lcf \cdot, \cdot \rcf_{_E}, \rho_{_E}, \alpha_{_E})$ and $(F, \lcf \cdot, \cdot \rcf_{_F}, \rho_{_F}, \alpha_{_F})$
be two Loday algebroids over $M$ and $N$, respectively, $\hat{\Phi} : F^\ast \to E^\ast$ a vector bundle map over a
diffeomorphism $\phi : N \to M$, $\Phi : \Gamma(E) \to \Gamma(F)$ the induced map on the modules of smooth sections,
and $\Phi^\star = (\Phi^\star _p)_{p\in \N^{^\ast}}$ the corresponding map from $(\mathcal{D}(\mathcal{F}), \diamond)$ to
$(\mathcal{D}(\mathcal{E}), \diamond)$. If the sequence of maps $\Phi^\star = (\Phi^\star _p)_{p\in \N^{^\ast}}$ is such that
\eqref{Phi^star commutation} holds, then $(\hat{\Phi}, \phi)$ is a Loday algebroid morphism.
\end{proposition}
\begin{proof}
We suppose that $\Phi^\star : (\mathcal{D}(\mathcal{F}), \partial_F) \to (\mathcal{D}(\mathcal{E}), \partial_E)$
satisfies \eqref{Phi^star commutation}. Our purpose is to show that $\Phi$ satisfies the axioms of Definition \ref{def-Loday-comorphism}.

\vspace{1mm}
\noindent
From the identity $\Phi_1^\star \circ \partial_F = \partial_E \circ \Phi_0^\star$ 
we get that,
for any $f\in \C$, $\Phi_1^\star \partial_F\phi^\ast f = \partial_Ef$, and for any $e\in \mathcal{E}$,
\begin{equation*}
\Phi_1^\star \partial_F\phi^\ast f (e) \circ \phi = \partial_Ef(e) \circ \phi \stackrel{\eqref{def - Phi^star}}{\Leftrightarrow} \partial_F\phi^\ast f (\Phi(e)) =  \partial_Ef(e) \circ \phi \Leftrightarrow \rho_{_F}(\Phi(e))(f\circ \phi) = \rho_{_E}(e)(f)\circ \phi,
\end{equation*}
which means that \eqref{comorphism - left anchor} is true.

\vspace{1mm}
\noindent
Now, from the identity $\partial_E \circ \Phi^\star_1 = \Phi_{2}^\star \circ \partial_F$ we obtain that $\Phi$ is a Loday algebra homomorphism. In fact, let $D\in \mathcal{D}^1(\mathcal{F})$ and $e_1, e_2 \in \mathcal{E}$, then
\begin{eqnarray*}
(\partial_E \circ \Phi^\star_1)(D)(e_1,e_2)\circ \phi & = & (\partial_E  \Phi^\star_1D)(e_1,e_2)\circ \phi \nonumber \\
& = & \rho_{_E}(e_1) (\Phi^\star_1D(e_2))\circ \phi - \rho_{_E}(e_2) (\Phi^\star_1D(e_1))\circ \phi \nonumber \\
& & - \, \Phi^\star_1D(\lcf e_1, e_2 \rcf_{_E})\circ \phi \nonumber \\
& \stackrel{\eqref{comorphism - left anchor}, \eqref{def - Phi^star}}{=} & \rho_{_F}(\Phi(e_1)) (\Phi^\star_1D(e_2)\circ \phi ) - \rho_{_F}(\Phi(e_2)) (\Phi^\star_1D(e_1)\circ \phi ) \nonumber \\
& & - \, D(\Phi \lcf e_1, e_2 \rcf_{_E} ) \nonumber \\
& = &  \rho_{_F}(\Phi(e_1)) (D(\Phi(e_2)) - \rho_{_F}(\Phi(e_2)) (D(\Phi(e_1)) \nonumber \\
& & - \,D(\Phi \lcf e_1, e_2 \rcf_{_E}) \nonumber \\
& = & \partial_F D (\Phi(e_1),\Phi(e_2)) + D(\lcf \Phi(e_1), \Phi(e_2) \rcf_{_F}) - D(\Phi \lcf e_1, e_2 \rcf_{_E}) \nonumber \\
& = & (\Phi_{2}^\star \circ \partial_F)(D)(e_1,e_2)\circ \phi \nonumber \\
& & +\, D(\lcf \Phi(e_1), \Phi(e_2) \rcf_{_F}) - D(\Phi \lcf e_1, e_2 \rcf_{_E}).
\end{eqnarray*}
Because of our hypothesis, the last equation implies that, for any $D\in \mathcal{D}^1(\mathcal{F})$,
\[D(\lcf \Phi(e_1), \Phi(e_2) \rcf_{_F}) - D(\Phi \lcf e_1, e_2 \rcf_{_E}) = 0.\]
Hence, for any $D\in \Gamma(F^\ast) \subset \mathcal{D}^1(\mathcal{F})$, the above equality is also true. So,
\[\Phi \lcf e_1, e_2 \rcf_{_E} = \lcf \Phi(e_1), \Phi(e_2) \rcf_{_F},\]
whence we conclude that $\Phi$ is a Loday algebra homomorphism.
\end{proof}

\section{Nonlinear connections of Loday algebroids}\label{section - nonlinear connections}
\subsection{Nonlinear $E$-connections and connection differential ope\-ra\-tors}\label{LA connection}
According to \cite{cr-fer-sec-cl}, a nonlinear connection of a Lie algebroid $A$
on a vector bundle $B$, both over a smooth manifold $M$, is a $\R$-bilinear map $\nabla : \Gamma(A) \times \Gamma(B) \to \Gamma(B)$ such that,
for any $a\in \Gamma(A)$, $\nabla_a$ is a derivative endomorphism of $B$, local in $a$. The last expression means
that the value at any point $x\in M$ of $\nabla_ab$, $b\in \Gamma(B)$, depends on the jet of order $1$ of $a$ at $x$, that is, $\nabla_a$ is not $\C$-linear with respect of $a$.
Also, a representation of $A$ is a flat vector bundle $B$ with respect to an $A$-connection.
Adopting this point of view, we define a Loday algebroid connection on a vector bundle in such a way that the adjoint representation of a Loday algebroid is well defined.
Since a Loday algebroid bracket $\lcf \cdot, \cdot \rcf$ on a vector bundle $E$ is a bidifferential operator of total order $\leq 1$
and has the locality property (Remark \ref{remark-locality}), we consider nonlinear connections in the above sense.

\begin{definition}\label{def:nonlinear:connection}
Let $(E, \lcf \cdot, \cdot \rcf, \rho, \alpha)$ be a Loday algebroid over a smooth manifold $M$ and $B$ a smooth vector bundle on $M$. A \emph{nonlinear $E$-connection on $B$} is a $\R$-bilinear map
\[\nabla : \Gamma(E)\times \Gamma(B)\to \Gamma(B)\]
such that, for all $e \in \Gamma(E)$, $b\in \Gamma(B)$ and $f\in \C$, the following properties hold:
\begin{enumerate}
\item
$\nabla_{e} (fb)=f\nabla_{e}b+\rho(e)(f)\,b$,
\item
$\nabla_{fe}b = f \nabla_e b + \sigma (df \otimes e)(b)$,
\end{enumerate}
where $\sigma : \Gamma(T^\ast M) \otimes \Gamma(E) \to \mathrm{End}(\Gamma(B))$ is the homomorphism on the $\C$-modules of smooth sections induced by a vector bundle map, also denoted by $\sigma$, $\sigma : T^\ast M \otimes E \to \mathrm{End}(B)$. In the following, we will call the map $\sigma$: \emph{symbol of $\nabla$}.
\end{definition}

\begin{remarks}\label{rems-dif op con}

\vspace{1mm}
\noindent
\begin{enumerate}
\item[\emph{1}.]
\emph{The above conditions ensure that $\nabla$ is a bidifferential operator of total order $\leq 1$ on the $\C$-bimodule $\Gamma(E)\times\Gamma(B)$ with values in $\Gamma(B)$.}
\item[\emph{2}.]
\emph{If $\nabla b$ is a $0$-order differential operator on $\Gamma(E)$, i.e. $\nabla b : \Gamma(E) \to \Gamma(B)$ is $\C$-linear or, equivalently, $\sigma = 0$, then we are dealing with a linear $E$-connection on $B$.
This remark explain the term ``symbol" for the vector bundle map $\sigma$.}
\end{enumerate}
\end{remarks}

\begin{proposition}\label{set nl conn - affine}
Let $(E,B)$ be as in Definition \ref{def:nonlinear:connection}. The set $\mathfrak{C}_{nl}(E,B)$ of nonlinear $E$-connections on $B$ is not
empty and is an affine space modeled on the linear space $\mathcal{D}_1^1(\mathcal{E}, \mathrm{End}(\Gamma(B)))$.
\end{proposition}
\begin{proof}
Since $E$ is an anchored vector bundle, according to Theorems 3.4 and 5.2 in \cite{Cantr-Lan},
the subset $\mathfrak{C}_{lin}(E,B)\subset\mathfrak{C}_{nl}(E,B)$ of linear $E$-connections on $B$ is not empty. Hence, $\mathfrak{C}_{nl}(E,B)$ is not empty.
For any two nonlinear $E$-connections $\nabla^0$ and $\nabla^1$ on $B$ with symbols $\sigma^0$ and $\sigma^1$, respectively, and any $g\in \C$, the affine combination
\[\nabla: = (1-g)\nabla^0 + g\nabla^1\]
is a nonlinear $E$-connection on $B$ with symbol $\sigma = (1-g)\sigma^0 + g\sigma^1$. Thus, $\mathfrak{C}_{nl}(E,B)$ is an affine space.
The difference $A = \nabla^1 - \nabla^0$ is a differential operator of order $\leq 1$ on the first entry and linear on the second entry, that is,
it is an element of the $\C$-module $\mathcal{D}_1^1(\mathcal{E}, \mathrm{End}(\Gamma(B)))$.
\end{proof}

\begin{remark}
\emph{Let $\mathbb{A}(B)\to M$ be the Atiyah algebroid of $B$ whose sections are the derivative endomorphisms of $\Gamma(B)$ \cite{mck}.
Then, we can consider the space $\mathfrak{C}_{nl}(E,B)$ as an affine subspace of the $\C$-module $\mathcal{D}_1^1(\mathcal{E}, \mathrm{Der}(\Gamma(B)))$, where $\mathrm{Der}(\Gamma(B))=\Gamma(\mathbb{A}(B))$.}
\end{remark}

\vspace{1mm}
\noindent
An equivalent description of a nonlinear $E$-connection $\nabla$ on $\Gamma(B)$ is given by the notion of \emph{covariant derivation}. It is a map
\[d^{\nabla} : \Gamma(B)  \to  \mathcal{D}^1_1 (\mathcal{E}, \Gamma(B)), \]
called \emph{the covariant derivation of $\nabla$}, satisfying the Leibniz rule
\begin{equation}\label{leibniz rule - nabla}
d^{\nabla}(fb) = \partial_E(f) \otimes b + f d^{\nabla}b,
\end{equation}
for all $f\in \C$ and $b\in \Gamma(B)$. It is related with $\nabla$ by the identity
\[d^{\nabla}b = \nabla_{\cdot} b.\]

\vspace{1mm}
\noindent
Relative to a local frame for the vector bundle $B$, a nonlinear $E$-connection $\nabla$ on $B$ can be represented by a matrix of $1$-differential operators,
which describes $\nabla$ locally. Let $U$ be an open subset of $M$ over which the vector bundle $B$
is trivial, $(b_1, \ldots, b_s)$, $s = \mathrm{rank} B$, a local frame of smooth sections of $B$ over $U$, and $e\in \Gamma(E\vert_U)$. On $U$,
since any section $b$ of $B$ is written as a $C^\infty(U,\R)$-linear combination $b = \sum_{i =1}^{s}f^ib_i$, the section $\nabla_eb$ is
computed from $\nabla_e b_i$ by linearity and the Leibniz rule. As section of $B\vert _U$, $\nabla_e b_i$ is a $C^\infty(U,\R)$-linear
combination of $(b_1, \ldots, b_s)$ with coefficients $\theta_i^j$ depending on $e$:
\begin{equation*}
\nabla_e b_i = \sum_{j=1}^s\theta_i^j(e)b_j.
\end{equation*}
A direct computation proves that $\theta_i^j$, $i, j=1, \ldots,s$, are $1$-differential operators on $\mathcal{E}$ of first order,
i.e. $\theta_i^j\in \mathcal{D}_1^1(\mathcal{E})$, called \emph{connection $1$-differential operators of $\nabla$ relative to the frame $(b_1,\ldots,b_s)$}.
The matrix $\theta = [\begin{array}{c}
                                                                             \theta_i^j \\
                                                                           \end{array}]$
will be called \emph{connection matrix of $\nabla$ relative to the frame $(b_1,\ldots,b_s)$}. Hence

\begin{proposition}\label{prop-nabla-theta-nabla}
Under the above assumptions and notations, a nonlinear $E$-connection $\nabla$ on $B\vert_U$ determines a unique connection matrix  $\theta = [
                       \begin{array}{c}
                         \theta_i^j \\
                       \end{array}]
$ of $1$-differential operators of first order relative to the frame $(b_1, \ldots,b_s)$. Conversely, any matrix $\theta = [
                       \begin{array}{c}
                         \theta_i^j \\
                       \end{array}]$
of such operators determines a nonlinear $E$-connection on $B\vert_U$.
\end{proposition}
\begin{proof}
The direct implication of the above proposition is immediate. We prove the converse. Given a matrix $\theta = [
                       \begin{array}{c}
                         \theta_i^j \\
                       \end{array}]
$ of elements in $\mathcal{D}^1_1(\mathcal{E})$, $e\in \Gamma(E\vert_U)$ and $b\in \Gamma(B\vert_U)$, we set
\[\nabla_e b_i = \sum_{j=1}^s\theta_i^j(e)b_j.\]
Define $\nabla_eb$ by applying linearity and the Leibniz rule to $b= \sum_{i=1}^sf^ib_i$:
\begin{eqnarray*}
\nabla_eb & = & \nabla_e(f^ib_i) = f^i \nabla_eb_i + \rho(e)(f^i)b_i \nonumber \\
& = & f^i\theta_i^j(e)b_j + \rho(e)(f^i)b_i = [ f^i\theta_i^j(e) + \rho(e)(f^j)]b_j.
\end{eqnarray*}
On the other hand, for any $g\in \C$, we have
\[\nabla_{ge}b_i = \theta_i^j(ge)b_j = g\theta_i^j(e)b_j + \sigma_i^j(dg\otimes e)b_j,\]
where $\sigma_i^j$ is the symbol of $\theta_i^j$ viewed as a section of $(TM \otimes E^\ast)\vert_U$. Set
$\sigma = [\begin{array}{c}\sigma_i^j \\\end{array}]$, which is a matrix of linear differential operators on
$\Gamma(E\vert_U)$ and $\sigma(dg \otimes e) = [\begin{array}{c}
\sigma_i^j(dg \otimes e) \\
\end{array}]$. The last matrix can be considered as the associated one to an endomorphism of $B\vert_U$ with respect to the frame $(b_1,\ldots,b_s)$.
Under these considerations, $\nabla$ is a nonlinear $E$-connection on $B\vert_U$.
\end{proof}

\subsection{Curvature of a nonlinear $E$-connection and curvature differential operators}
Let $\mathcal{D}^p(\mathcal{E}, \Gamma(B))$, $p\in \N$, be the $\C$-module of $p$-differential operators on $\mathcal{E}$ with values in $\Gamma(B)$. We set $\mathcal{D}(\mathcal{E}, \Gamma(B)): = \big(\mathcal{D}^p(\mathcal{E}, \Gamma(B))\big)_{p\in \N}$ and we note that $\mathcal{D}^0(\mathcal{E}, \Gamma(B))\cong\Gamma(B)$.

\vspace{2mm}
\noindent
We extend the shuffle product $\diamond$ in $\mathcal{D}(\mathcal{E})$ to a product,
also denoted by $\diamond$, between elements of $\mathcal{D}(\mathcal{E})$ and of $\mathcal{D}(\mathcal{E}, \Gamma(B))$,
by setting, for any $D \in \mathcal{D}^p(\mathcal{E})$ and $\Delta \in \mathcal{D}^q(\mathcal{E}, \Gamma(B))$, $D \diamond \Delta$
to be the element in $\mathcal{D}^{p+q}(\mathcal{E}, \Gamma(B))$ given, for any $e_1, \ldots, e_{p+q} \in \Gamma(E)$, by
\begin{equation*}
D \diamond \Delta (e_1, \ldots, e_{p+q}) = \sum_{\tau \in Sh(p,q)}(-1)^{|\tau|}D(e_{\tau(1)}, \ldots, e_{\tau(p)})\Delta(e_{\tau(p+1)}, \ldots, e_{\tau(p+q)}).
\end{equation*}
In the above definition we take into account that $\Gamma(B)$ is a $\C$-bimodule.
The extended product $\diamond : \mathcal{D}(\mathcal{E}) \times \mathcal{D}(\mathcal{E}, \Gamma(B)) \to \mathcal{D}(\mathcal{E}, \Gamma(B))$
endows $\mathcal{D}(\mathcal{E}, \Gamma(B))$ with the structure of a $\mathcal{D}(\mathcal{E})$-module.

\vspace{2mm}
\noindent
As in the classical theory of linear connections, the covariant derivative $d^\nabla$ of a nonlinear $E$-connection
$\nabla$ on $B$ can be uniquely extended to an operator of degree $+1$, also denoted by $d^{\nabla}$,
\begin{equation*}
d^{\nabla} : \mathcal{D}^{\bullet}(\mathcal{E}, \Gamma(B)) \to \mathcal{D}^{\bullet + 1}(\mathcal{E}, \Gamma(B)),
\end{equation*}
given, for any $\Delta \in \mathcal{D}^{p}(\mathcal{E}, \Gamma(B))$ and $e_1, \ldots, e_{p+1} \in \Gamma(E)$, by
\begin{eqnarray*}
d^{\nabla}\Delta (e_1, \ldots, e_{p+1}) & = & \sum_{i=1}^{p+1}(-1)^{i+1}\nabla_{e_i}(\Delta (e_1, \ldots, \hat{e}_i, \ldots, e_{p+1})) \nonumber \\
& & +\, \sum_{i<j} (-1)^i \Delta (e_1, \ldots, \hat{e}_i, \ldots, \hat{e}_j, \lcf e_i, e_j\rcf, e_{j+1}, \ldots, e_{p+1}),
\end{eqnarray*}
where $\hat{\cdot}$ indicates that the corresponding term is omitted. The above operator $d^{\nabla}$ is a generalized derivation of $\mathcal{D}(\mathcal{E}, \Gamma(B))$. For any $D \in \mathcal{D}^p(\mathcal{E})$ and $\Delta \in \mathcal{D}(\mathcal{E}, \Gamma(B))$, it satisfies the generalized Leibniz rule:
\begin{equation}\label{general-Leibniz}
d^{\nabla} (D \diamond \Delta) = \partial_E D \diamond \Delta + (-1)^p D\diamond d^{\nabla} \Delta.
\end{equation}

\vspace{1mm}
\noindent
The map $(d^{\nabla})^2 : \mathcal{D}^{\bullet}(\mathcal{E}, \Gamma(B)) \to \mathcal{D}^{\bullet + 2}(\mathcal{E}, \Gamma(B))$ is $\C$-linear
on the sections of $B$, so it can be considered as an element $R^{\nabla}$ of the space $\mathcal{D}^2(\mathcal{E}, \mathrm{End}(\Gamma(B)))$
of $2$-differential operators on $\mathcal{E}$ with values in $\mathrm{End}(\Gamma(B))$.

\begin{definition}
The \emph{curvature} of a nonlinear $E$-connection $\nabla$ on $B$ is the operator $(d^{\nabla})^2 \in \mathcal{D}^2(\mathcal{E}, \mathrm{End}(\Gamma(B)))$.
If the curvature vanishes identically, then $\nabla$ is said to be \emph{flat}.
\end{definition}
If $\nabla$ is flat, $d^{\nabla}$ is a differential on $\mathcal{D}^{\bullet}(\mathcal{E}, \Gamma(B))$ and the cohomology $H^{\bullet}(\mathcal{E}, \Gamma(B))$ of
the cochain complex $\big(\mathcal{D}^{\bullet}(\mathcal{E}, \Gamma(B)), d^{\nabla} \big)$ is called \emph{Loday algebroid cohomology with coefficients in $\Gamma(B)$}. In this case, we say that \emph{$B$ is a representation of the Loday algebroid $E$}.

\vspace{2mm}
\noindent
On $\mathcal{D}^{\bullet}(\mathcal{E}, \Gamma(B))$, it is also defined the operator \emph{interior product of a differential operator by a section of $E$} given, for any $e\in \Gamma(E)$, by
\begin{eqnarray*}
i_e : \mathcal{D}^{\bullet}(\mathcal{E}, \Gamma(B)) & \to & \mathcal{D}^{\bullet - 1}(\mathcal{E}, \Gamma(B)) \\
 \Delta &\mapsto & i_e \Delta : = \Delta(e, \ldots).
\end{eqnarray*}
Hence, by considering the graded commutator on the space of graded endomorphisms of $\mathcal{D}^{\bullet}(\mathcal{E}, \Gamma(B))$,
\begin{equation*}
[P,Q]_{com} = P\circ Q - (-1)^{pq}Q\circ P,
\end{equation*}
where $P$ and $Q$ are graded endomorphisms of $\mathcal{D}^{\bullet}(\mathcal{E}, \Gamma(B))$ of degree $p$ and $q$, respectively, we obtain the operator
\begin{equation*}
\nabla_e  =  [i_e, d^{\nabla}]_{com} = i_e \circ d^{\nabla} + d^{\nabla} \circ i_e
\end{equation*}
and the identity, for any $e_1, e_2 \in \Gamma(E)$,
\begin{equation*}
[\nabla_{e_1}, i_{e_2}]_{com}  =  \nabla_{e_1} \circ i_{e_2} - i_{e_2} \circ \nabla_{e_1} = i_{\lcf e_1, e_2\rcf}.
\end{equation*}

\begin{proposition}
The curvature $(d^{\nabla})^2 : \Gamma(B) \to \mathcal{D}^{2}(\mathcal{E}, \Gamma(B))$ of a nonlinear $E$-connection $\nabla$ on $B$ satisfies the identity
\begin{equation}\label{curvat-d2}
i_{e_2}\circ i_{e_1} ((d^{\nabla})^2 b) = \nabla_{e_1}\nabla_{e_2}b -  \nabla_{e_2}\nabla_{e_1}b- \nabla_{\lcf e_1, e_2 \rcf}b.
\end{equation}
As a bidifferential operator on $\mathcal{E}$ with values in $\mathrm{End}(\Gamma(B))$, it is of second order with respect to the first argument
and of first order with respect to the second argument. For this reason, we write $R^{\nabla}\in \mathcal{D}^2_{(2,1)}(\mathcal{E}, \mathrm{End}(\Gamma(B)))$.
\end{proposition}
\begin{proof}
A simple computation shows that \eqref{curvat-d2} is true, from which we obtain the well known curvature formula:
\begin{equation}\label{expr-curvature}
R^{\nabla}(e_1,e_2) = \nabla_{e_1}\nabla_{e_2} - \nabla_{e_2}\nabla_{e_1}- \nabla_{\lcf e_1, e_2 \rcf}.
\end{equation}
The assertion that $R^{\nabla}\in \mathcal{D}^2_{(2,1)}(\mathcal{E}, \mathrm{End}(\Gamma(B)))$ can be justified by computing its symbols.
\end{proof}

\begin{remark}
\emph{Viewing $\nabla_e$ as a derivative endomorphism of $\Gamma(B)$, the curvature $R^{\nabla}$ can be written as
\[R^{\nabla}(e_1,e_2) = [\nabla_{e_1},\nabla_{e_2}]_{com} - \nabla_{\lcf e_1, e_2 \rcf},\]
where $[\cdot, \cdot]_{com}$ is the Lie algebra bracket on the space $\mathrm{Der}(\Gamma(B))$. The flatness of $\nabla$ implies that
\[\nabla_{\lcf e_1, e_2 \rcf} = [\nabla_{e_1},\nabla_{e_2}]_{com},\]
which means that $\nabla$ is a Loday algebra homomorphism from $(\Gamma(E), \lcf \cdot, \cdot \rcf)$ to $(\mathrm{Der}(\Gamma(B)), [\cdot, \cdot]_{com})$.
For the last conclusion, we use the fact that every Lie algebra is a Loday algebra.}
\end{remark}

\vspace{1mm}
\noindent
As we have seen in Subsection \ref{LA connection}, with respect to a local frame of smooth sections of $B$, a nonlinear $E$-connection $\nabla$ on $B$ can be represented by a matrix $\theta$ of first order differential operators, which describes $\nabla$ locally.
In the following, we will see that the curvature of $\nabla$ is represented by a matrix of $2$-differential operators and that the relationship between these two matrices yields the structural equation of $\nabla$.
Using the notation introduced in the corresponding paragraph of Subsection \ref{LA connection}, we have that,
for any $e_1, e_2 \in \Gamma(E\vert_U)$, the section $R^{\nabla}(e_1,e_2)b_i$ of $B\vert_U$ is a $C^\infty(U,\R)$-linear combination of $(b_1, \ldots, b_s)$ with coefficients $\Theta_i^j$ depending on $(e_1,e_2)$:
\begin{equation}\label{curvat-Theta}
R^{\nabla}(e_1, e_2)b_i = \sum_{j=1}^{s}\Theta_i^j(e_1, e_2)b_j.
\end{equation}
From \eqref{expr-curvature} and the properties of $\nabla$, we obtain
\begin{equation}\label{second-struct-eq}
\Theta_i^j = \partial_E \theta_i^j - \sum_{m=1}^s\theta^m_i \diamond \theta^j_m, \quad \quad i,j =1, \ldots, s.
\end{equation}
So, $\Theta_i^j$, $i,j =1, \ldots, s$, are $2$-differential operators on $\Gamma(E\vert_U)$ of second order, called the \emph{curvature $2$-differential operators of $\nabla$}. Precisely, $\Theta_i^j \in \mathcal{D}^2_{(2,1)}(\Gamma(E\vert_U))$. The matrix $\Theta = [\begin{array}{c}
                                                                             \Theta_i^j \\
                                                                           \end{array}]$
is called the \emph{curvature matrix of $\nabla$ relative to the frame $(b_1, \ldots, b_s)$} and the relations \eqref{second-struct-eq}
are called \emph{second structural equations of $\nabla$}. The set of equations \eqref{second-struct-eq} can be written in matrix form
\begin{equation}\label{THETA}
\Theta = \partial_E\theta - \theta \diamond \theta,
\end{equation}
where $\partial_E\theta$ is defined to be $[\begin{array}{c}
                                               \partial_E \theta_i^j \\
                                             \end{array}]$
and $\theta \diamond \theta$ is defined to be the matrix of differential operators whose $(i,j)$-entry is $(\theta \diamond \theta)_i^j = \sum_{m=1}^s \theta^m_i \diamond \theta^j_m$.

\vspace{2mm}
\noindent
At this point, it is of interest to derive the transformation rules for the local expressions of $\nabla$ and its curvature $R^{\nabla}$.

\begin{proposition}\label{prop - theta - theta'}
Suppose that $(b_1, \ldots, b_s)$ and $(b_1', \ldots, b_s')$ are two local frames of the vector bundle $B$ over the open sets $U$ and $U'$ of $M$, respectively, such that $b_i' = a_i^jb_j$, $a_i^j \in C^\infty (U\cap U', \R)$, for some nondegenerate matrix $a=[\begin{array}{c}
                                               a_i^j \\
                                             \end{array}]$. If $\theta$ and $\theta'$ are the connection matrices and $\Theta$ and $\Theta'$ are the curvature matrices of a nonlinear $E$-connection $\nabla$ on $B$ relative to $(b_1, \ldots, b_s)$ and $(b_1', \ldots, b_s')$, respectively, on $U\cap U'\neq \emptyset$, then
\begin{enumerate}
\item
\begin{equation}\label{theta - transf - rule}
\theta' = a\diamond \theta \diamond a^{-1} + \partial_E a \diamond a^{-1},
\end{equation}
\item
\begin{equation*}\label{THETA - transf - rule}
\Theta' = a \diamond \Theta \diamond a^{-1}.
\end{equation*}
\end{enumerate}
\end{proposition}
\begin{proof}
Simple computations. We use the Leibniz rule \eqref{leibniz rule - nabla} to deduce the transformation rule for the connection matrix under a change of frame of $B$. In order to establish the transformation rule for the curvature matrix, in addition to the Leibniz rule \eqref{general-Leibniz}, we also use the second structural equation \eqref{THETA}.
\end{proof}

\subsection{Induced nonlinear $E$-connections}
\subsubsection{Connections on the dual bundle}
Let $\nabla$ be a nonlinear $E$-connection on $B$ and $B^\ast$ the dual bundle of $B$. We consider the map
\begin{equation*}
\nabla^{\ast} : \Gamma(E) \times \Gamma(B^{\ast}) \to \Gamma(B^{\ast}),
\end{equation*}
which is completely determined by the relation
\begin{equation}\label{def dual connection}
\rho(e)\langle b^\ast, b\rangle = \langle  \nabla^\ast_e b^\ast, b \rangle + \langle b^\ast, \nabla_eb \rangle, \quad \quad e\in \Gamma(E), \quad b\in \Gamma(B), \quad b^\ast \in \Gamma(B^\ast),
\end{equation}
$\langle \cdot, \cdot \rangle$ being the duality pairing between elements of $\Gamma(B^\ast)$ and of $\Gamma(B)$.

\begin{proposition}
The map $\nabla^{\ast}$ defines a nonlinear $E$-connection on $B^\ast$. For each $e\in \Gamma(E)$, $b^{\ast}\in \Gamma(B^{\ast})$ and $f\in \C$, we have
\begin{enumerate}
\item
$\nabla^{\ast}_{e}(fb^\ast) = f\nabla^{\ast}_{e}b^\ast + \rho(e)(f)b^\ast$,
\item
$\nabla^{\ast}_{fe}b^\ast = f\nabla^{\ast}_{e}b^\ast - \sigma(df \otimes e)^{t} (b^\ast)$,
\end{enumerate}
where $\sigma(df \otimes e)^{t}$ is the transpose map of the endomorphism $\sigma(df \otimes e)$ of $\Gamma(B)$. The differential operator $\nabla^{\ast}$ is also of total order $\leq 1$. The curvature $R^{\nabla^\ast}$ of $\nabla^\ast$ verifies the relation
\begin{equation}\label{dual curvature}
\langle R^{\nabla^\ast}(e_1,e_2)b^\ast, b\rangle + \langle b^\ast, R^{\nabla}(e_1,e_2)b\rangle = 0.
\end{equation}
\end{proposition}
\begin{proof}
Simple computations.
\end{proof}

\subsubsection{Connections on the bundle of endomorphisms}
The pair $(\nabla, \nabla^\ast)$ of nonlinear $E$-connections on $B$ and $B^\ast$, respectively, induces tensor product nonlinear $E$-connections on any tensor bundle constructed from $B$ and $B^\ast$. In particular, it induces on $B^\ast \otimes B \cong \mathrm{End}(B)$ the nonlinear $E$-connection
$\tilde{\nabla} = \nabla^\ast \otimes \mathrm{Id}_{\Gamma(B)} + \mathrm{Id}_{\Gamma(B^\ast)}\otimes \nabla$. The last, in an alternative way, is given, for any $e\in \Gamma(E)$ and $\Phi \in \mathrm{End}(\Gamma(B))$, by
\begin{equation*}
\tilde{\nabla}_e\Phi = [\nabla_e, \Phi]_{com},
\end{equation*}
where $[\cdot, \cdot]_{com}$ is the commutator bracket on $\mathrm{Der}(\Gamma(B))$\footnote{We are obliged to
consider the space $\mathrm{Der}(\Gamma(B))$ of derivative endomorphisms of $\Gamma(B)$, which contains the space $\mathrm{End}(\Gamma(B))$,
because $\nabla_e \in \mathrm{Der}(\Gamma(B))$. But, the bracket of $\nabla_e$ with an endomorphism of $\Gamma(B)$ is
an endomorphism of $\Gamma(B)$.}. For any $f\in \C$, $\tilde{\nabla}$ verifies the relations
\begin{enumerate}
\item
$\tilde{\nabla}_e(f\Phi) = f \tilde{\nabla}_e\Phi + \rho(e)(f)\Phi$,
\item
$\tilde{\nabla}_{fe}\Phi = f\tilde{\nabla}_e\Phi + [\sigma(df \otimes e), \Phi]_{com}$,
\end{enumerate}
whence we conclude that it is also a bidifferential operator of total order $\leq 1$.

\vspace{1mm}
\noindent
For any $\Phi, \Psi \in \mathrm{End}(\Gamma(B))$ and $e\in \Gamma(E)$,
\begin{equation}\label{formula der circ}
\tilde{\nabla}_e (\Phi \circ \Psi) = (\tilde{\nabla}_e \Phi) \circ \Psi + \Phi \circ (\tilde{\nabla}_e \Psi).
\end{equation}

\vspace{2mm}
\noindent
Let $\mathcal{D}^{p}(\mathcal{E}, \mathrm{End}(\Gamma(B)))$ be the space
of multidifferential operators, with $p$ entries, on $\mathcal{E}$ with values in $\mathrm{End}(\Gamma(B))$.
The nonlinear $E$-connection $\tilde{\nabla}$ defines a covariant derivation operator $d^{\tilde{\nabla}}$ of degree $+1$
on $\mathcal{D}^{\bullet}(\mathcal{E}, \mathrm{End}(\Gamma(B)))$ defined,
for any $\Delta \in \mathcal{D}^{p}(\mathcal{E}, \mathrm{End}(\Gamma(B)))$ and $e_1, \ldots, e_{p+1} \in \mathcal{E}$, by
\begin{eqnarray}\label{dif napla tilde}
d^{\tilde{\nabla}}\Delta (e_1, \ldots, e_{p+1}) & = & \sum_{i=1}^{p+1}(-1)^{i+1} \tilde{\nabla}_{e_i}(\Delta(e_1, \ldots, \hat{e}_i, \ldots, e_{p+1})) \nonumber \\
& & + \, \sum_{i<j}(-1)^i \Delta (e_1, \ldots, \hat{e}_i,\ldots, \hat{e}_j, \lcf e_i, e_j\rcf, e_{j+1}, \ldots, e_{p+1}).
\end{eqnarray}
The curvature $R^{\tilde{\nabla}} = (d^{\tilde{\nabla}})^2$ of $\tilde{\nabla}$ is given, for any $e_1,e_2 \in \Gamma(E)$ and $\Phi \in \mathrm{End}(\Gamma(B))$, by
\begin{equation*}
R^{\tilde{\nabla}}(e_1,e_2)\Phi = R^{\nabla}(e_1,e_2) \circ \Phi - \Phi \circ R^{\nabla}(e_1,e_2)
\end{equation*}
or, alternatively, by
\[R^{\tilde{\nabla}}(e_1,e_2)\Phi = [R^{\nabla}(e_1,e_2), \Phi]_{com}.\]
Hence, the flatness of $\nabla$ implies the flatness of $\tilde{\nabla}$ and, in this case, $d^{\tilde{\nabla}}$ is a
differential. Then, the pair $\big(\mathcal{D}(\mathcal{E}, \mathrm{End}(\Gamma(B))), \,  d^{\tilde{\nabla}}\big)$,
where $\mathcal{D}(\mathcal{E}, \mathrm{End}(\Gamma(B))) = \big(\mathcal{D}^{p}(\mathcal{E}, \mathrm{End}(\Gamma(B)))\big)_{p\in \N}$,
is a cochain complex with cohomology $H^{\bullet}(\mathcal{E}, \mathrm{End}(\Gamma(B)))$.

\vspace{2mm}
\noindent
In the above framework, the \emph{Bianchi identity}
\begin{equation}\label{bianchi id}
d^{\tilde{\nabla}}R^{\nabla} =0
\end{equation}
holds.

\vspace{1mm}
\noindent
In terms of the connection matrix $\theta$ and the curvature matrix $\Theta$ of $\nabla$, the Bianchi identity can be written as
\begin{equation}\label{bianchi id - theta}
\partial_E \Theta - \theta \diamond \Theta + \Theta \diamond \theta = 0,
\end{equation}
where $\theta \diamond \Theta$ (resp. $\Theta \diamond \theta$) is the matrix of which the $(i,j)$-entry
is the $3$-differential operator $\theta_i^l\diamond \Theta_l^j$ (resp. $\Theta_i^l \diamond \theta_l^j$). It is an immediate consequence of \eqref{THETA}.

\vspace{3mm}
\noindent
Also, for use in Section \ref{section char classes}, we note that the shuffle product $\diamond$ in $\mathcal{D}(\mathcal{E})$ can be extended to a product,
also denoted by $\diamond$, on the space $\mathcal{D}(\mathcal{E}, \mathrm{End}(\Gamma(B)))=\big( \mathcal{D}^p(\mathcal{E}, \mathrm{End}(\Gamma(B)))\big)_{p\in \N} $
by setting, for any $\Delta \in \mathcal{D}^p(\mathcal{E}, \mathrm{End}(\Gamma(B)))$ and $\Delta' \in \mathcal{D}^q(\mathcal{E}, \mathrm{End}(\Gamma(B)))$,
\begin{equation}\label{sh-prod-values endomorphisms}
\Delta \diamond \Delta'(e_1, \ldots, e_{p+q}) = \sum_{\tau \in Sh(p,q)}(-1)^{|\tau|}\Delta(e_{\tau(1)}, \ldots, e_{\tau(p)})\circ \Delta'(e_{\tau(p+1)}, \ldots, e_{\tau(p+q)}).
\end{equation}
Then the next result is true.
\begin{proposition}\label{prop-grad-end-alg}
The space $\big(\mathcal{D}(\mathcal{E}, \mathrm{End}(\Gamma(B))), \diamond \big)$ is a graded associative unital $\R$-algebra.
\end{proposition}

\begin{lemma}
Under the above assumptions, we have
\begin{equation}\label{eq der diamond}
d^{\tilde{\nabla}}(\Delta \diamond \Delta') = (d^{\tilde{\nabla}} \Delta)\diamond \Delta' + (-1)^p \Delta \diamond (d^{\tilde{\nabla}} \Delta').
\end{equation}
\end{lemma}
\begin{proof}
Simple computations, using (\ref{formula der circ}).
\end{proof}

\vspace{2mm}
\noindent
Moreover, we have that the trace map $\tr : \mathrm{End}(\Gamma(B)) \to \C$ naturally induces a trace map
\[\tr: \mathcal{D}^{p}(\mathcal{E}, \mathrm{End}(\Gamma(B))) \to \mathcal{D}^{p}(\mathcal{E}), \quad \quad p\in \N,\]
defined, for any $\Delta \in \mathcal{D}^{p}(\mathcal{E}, \mathrm{End}(\Gamma(B)))$ and $e_1, \ldots, e_p \in \mathcal{E}$, by
\begin{equation}\label{def-Tr}
\tr \Delta(e_1,\ldots,e_p) = \tr (\Delta(e_1,\ldots,e_p)).
\end{equation}

\vspace{2mm}
\noindent
For any $\Phi \in \mathrm{End}(B)$ and any $e\in \Gamma(E)$,
\begin{equation}\label{tr-rho}
\tr (\tilde{\nabla}_e \Phi) = \rho(e)(\tr \Phi).
\end{equation}

\begin{proposition}
The trace map $\tr : \mathcal{D}^{\bullet}(\mathcal{E}, \mathrm{End}(\Gamma(B))) \to \mathcal{D}^{\bullet}(\mathcal{E})$  satisfies, for any $\Delta \in \mathcal{D}^{\bullet}(\mathcal{E}, \mathrm{End}(\Gamma(B)))$, the identity
\begin{equation}\label{dif - Tr}
\tr d^{\tilde{\nabla}}\Delta = \partial_E (\tr\Delta).
\end{equation}
\end{proposition}
\begin{proof}
Taking into account the expressions \eqref{dif napla tilde} and \eqref{def:Loday:cohomology:operator}
of the differentials $d^{\tilde{\nabla}}$ and $\partial_E$, respectively, and the relations \eqref{def-Tr} and \eqref{tr-rho},
we can easily establish the identity \eqref{dif - Tr}.
\end{proof}

\vspace{2mm}
\noindent
Furthermore, a straightforward computation gives:
If $(\nabla_t)_{t\in \R}$ is an $1$-parameter family of nonlinear $E$-connections on $B$ and $R^{\nabla_t}$ the curvature of $\nabla_t$, we have
\begin{equation}\label{der R^t}
\frac{d R^{\nabla_t}}{dt} = d^{\tilde{\nabla}_t}\frac{d \nabla_t}{dt}.
\end{equation}
\
\subsection{Examples of nonlinear $E$-connections}
Next, we present some characteristic examples of nonlinear $E$-connections on a vector bundle $B$ over $M$, $(E, \lcf \cdot, \cdot \rcf, \rho, \alpha)$ being a
Loday algebroid over $M$.

\begin{examples}\label{examples - nonlinear connections}

\vspace{1mm}
\noindent
\begin{enumerate}
\item[\emph{1}.]
\emph{By the Jacobi identity, the Loday bracket on $\Gamma(E)$ defines a flat nonlinear $E$-connection $\nabla^{\mathrm{ad}}$ on $E$
called the \emph{adjoint representation of $E$}:}
\[\nabla^{\mathrm{ad}}_{e_1}e_2 : = \lcf e_1, e_2 \rcf, \quad \quad e_1, e_2 \in \Gamma(E).\]
\emph{Its symbol is the vector bundle map $\alpha - \rho^{t} \otimes \mathrm{Id}_{E} : T^\ast M \otimes E \to \mathrm{End}(E)$,
while its connection matrix $\theta^{\mathrm{ad}}$ satisfies the identity}
\begin{equation}\label{curv ad-connection}
\partial_E\theta^{\mathrm{ad}} = \theta^{\mathrm{ad}} \diamond \theta^{\mathrm{ad}}.
\end{equation}
\item[\emph{2}.]
\emph{The adjoint representation can be extended, by derivation, to a representation $\Lie$ of $E$ on $\bigwedge^{k}E$ given, for any $e, e_1,\ldots,e_k \in \Gamma(E)$,
$k=1, \ldots, r=\mathrm{rank}\,E$, by}
\begin{equation*}
\Lie_e (e_1\wedge \ldots \wedge e_k) : = \sum_{i=1}^ke_1\wedge \ldots \wedge \lcf e, e_i\rcf \wedge \ldots \wedge e_k.
\end{equation*}

\vspace{1mm}
\noindent
\emph{Suppose that $E$ is a trivial vector bundle and consider the action $\nabla^{\mathrm{top}} : \Gamma(E) \times \Gamma(\bigwedge^{\mathrm{top}}E) \to \Gamma(\bigwedge^{\mathrm{top}}E)$ given, for any $e\in \Gamma(E)$ and $s \in \Gamma(\bigwedge^{\mathrm{top}}E)$, by
\[\nabla^{\mathrm{top}}_e s = \Lie_es.\]
Let $(e_1, \ldots, e_r)$ be a local frame of smooth sections of $E$ and $\theta^{\mathrm{ad}}$ the connection matrix of $\nabla^{\mathrm{ad}}$ relative to the above frame. A simple computation yields}
\begin{equation*}
\nabla^{\mathrm{top}}_e (e_1\wedge \ldots \wedge e_r) = \tr (\theta^{\mathrm{ad}}(e))e_1\wedge \ldots \wedge e_r,
\end{equation*}
\emph{whence we get $\theta^{\mathrm{top}} =  \tr(\theta^{\mathrm{ad}})$ and, taking into account \eqref{curv ad-connection}, we conclude that $\Theta^{\mathrm{top}}=0$. So, $\bigwedge^{\mathrm{top}}E$ is a representation of $E$ relatively to $\nabla^{\mathrm{top}}$. Also, we have}
\begin{equation*}
\nabla^{\mathrm{top}}_{fe} (e_1\wedge \ldots \wedge e_r) = f\nabla^{\mathrm{top}}_e (e_1\wedge \ldots \wedge e_r) + [\tr(\alpha(df \otimes e))-\rho(e)(f)]e_1\wedge \ldots \wedge e_r.
\end{equation*}
\emph{The last means that the symbol of $\nabla^{\mathrm{top}}$ is the vector bundle map $\sigma : T^\ast M \otimes E \to \mathrm{End}(\bigwedge^{\mathrm{top}}E)$
given by $\sigma(df \otimes e) = [\tr(\alpha(df \otimes e))-\rho(e)(f)]\otimes \mathrm{Id}_{\bigwedge^{\mathrm{top}}E}$.}
\item[\emph{3}.]
\emph{In the same spirit, if $E$ is a trivial vector bundle and $\nabla : \Gamma(E) \times \Gamma(E) \to \Gamma(E)$ a nonlinear $E$-connection on $E$
with connection matrix $\theta = [\begin{array}{c}\theta_i^j \\ \end{array}]$,
curvature matrix $\Theta =  [\begin{array}{c}\Theta_i^j \\ \end{array}]$,
and symbol $\sigma$, then its extension $\nabla' : \Gamma(E) \times \Gamma(\bigwedge^{\mathrm{top}}E) \to \Gamma(\bigwedge^{\mathrm{top}}E)$ on $\bigwedge^{\mathrm{top}}E$,
given by}
\begin{equation*}
\nabla'_e (e_1\wedge \ldots \wedge e_r) = \sum_{i=1}^r e_1 \wedge \ldots \wedge e_{i-1} \wedge \nabla_ee_i \wedge e_{i+1}\wedge \ldots \wedge e_r,
\end{equation*}
\emph{is also a nonlinear $E$-connection on $\bigwedge^{\mathrm{top}}E$. The connection matrix of $1$-differential operators of $\nabla'$
with respect to the frame $(s)=(e_1\wedge \ldots \wedge e_r)$ of $\bigwedge^{\mathrm{top}}E$ is $\theta' = [\begin{array}{c}\tr \theta \\ \end{array}]$ and
its symbol is the vector bundle map
$\sigma' : T^{\ast} M \otimes E \to \mathrm{End}(\bigwedge^{\mathrm{top}}E)$, $\sigma' = (\tr\circ \sigma)\otimes  \mathrm{Id}_{\bigwedge^{\mathrm{top}}E}$.
The curvature matrix of $\nabla'$ is:}
\[\Theta' = \partial_E \theta'= \partial_E(\tr \theta)=\tr (\partial_E\theta)=\tr \Theta.\]
\emph{Hence, if $E$ is a representation of $E$ with respect of $\nabla$, $\bigwedge^{\mathrm{top}}E$ is also a representation of $E$ relatively to $\nabla'$. }
\item[\emph{4}.]
\emph{The cotangent bundle $T^\ast M$ of $M$ is a representation of $E$ relatively to the nonlinear $E$-connection $\nabla : \Gamma(E) \times \Gamma(T^\ast M) \to \Gamma(T^\ast M)$ given, for any $e\in \Gamma(E)$ and $\eta \in \Gamma(T^\ast M)$ by}
\begin{equation}\label{conn on T*M}
\nabla_e \eta : = \mathcal{L}_{\rho(e)}\eta,
\end{equation}
\emph{where $\mathcal{L}$ denotes the classical Lie derivative of differential forms along a vector field.
Its symbol is the vector bundle map
$\sigma = \mathrm{Id}_{T^\ast M}\otimes \rho : T^\ast M \otimes E \to \mathrm{End}(T^\ast M)$ with $\sigma(df \otimes e)(\eta) = \langle \eta, \rho(e)\rangle df$.
While, its connection matrix at a section $e\in \Gamma(E)$, relatively to a coordinate system $(x^1, \ldots, x^n)$ on $M$, is the matrix $\theta(e) = [
  \begin{array}{c}
   \displaystyle{\frac{\partial \rho^i(e)}{\partial x^j}}\\
    \end{array}
   ]$, where $\rho^i(e)$ denotes the $i$-coordinate function of the vector field $\rho(e)$.}
\item[\emph{5}.]
\emph{Let $M$ be an orientable manifold of dimension $n$ and $(x^1, \ldots, x^n)$ a local coordinate system of $M$.
The extension $\nabla^M$ of \eqref{conn on T*M} on $\bigwedge^{n}T^\ast M$ given, for any volume form $\mu \in \Gamma(\bigwedge^{n}T^\ast M)$ and $e \in \Gamma(E)$, by}
\begin{equation*}
\nabla^M_e\mu : = \mathcal{L}_{\rho(e)}\mu,
\end{equation*}
\emph{defines also a flat nonlinear $E$-connection on the trivial real line bundle $\bigwedge^{n}T^\ast M$. Its symbol is the vector bundle map
$\sigma : T^\ast M \otimes E \to \mathrm{End}(\bigwedge^{n}T^\ast M)$, $\sigma(df\otimes e) = \rho(e)(f)\otimes \mathrm{Id}_{\bigwedge^n T^\ast M}$.
Since, for each volume form $\mu \in \Gamma(\bigwedge^{n}T^\ast M)$, $\mathcal{L}_{\rho(e)}\mu = \mathrm{div}_{\mu}(\rho(e))\mu$, where $\mathrm{div}_{\mu}(\rho(e))$ is the divergence of the vector field $\rho(e)$ with respect to the volume form $\mu$, the connection matrix $\theta^M$ of $\nabla^M$ relatively to $(\mu)$ is the matrix
$\theta^M_{\mu} = [\begin{array}{c}\mathrm{div}\circ \rho \\\end{array}]$.}

\vspace{1mm}
\noindent
\emph{The flatness of $\nabla^M$ arises from the fact that,
for any $e_1, e_2 \in \Gamma(E)$, $[\mathcal{L}_{\rho(e_1)}, \mathcal{L}_{\rho(e_2)}]_{com}= \mathcal{L}_{[\rho(e_1), \rho(e_2)]}= \mathcal{L}_{\rho(\lcf e_1, e_2 \rcf)}$. As a consequence, we have that $\bigwedge^{n}T^\ast M$ is a representation of $E$ relatively to $\nabla^M$. Furthermore, the operator $\mathrm{div}_{\mu} \circ \rho$ is a $1$-cocycle of the Loday cochain complex $(\mathcal{D}^{\bullet}(\mathcal{E}), \partial_E)$. Indeed, the curvature matrix of $\nabla^M$ is $\Theta_{\mu}^M = \partial_E \theta^M_{\mu} = [\begin{array}{c}
\partial_E (\mathrm{div}_{\mu}\circ \rho)\\
 \end{array}]$. So, $\partial_E (\mathrm{div}_{\mu}\circ \rho) =0$.}
 \item[\emph{6}.]
\emph{Let $\triangle : \Gamma(E)\times \Gamma(E) \to \Gamma(E)$ be a linear $E$-connection on $E$ (see Remarks \ref{rems-dif op con}). Then the map}
\begin{eqnarray*}
\nabla : \Gamma(E) \times \Gamma(E) & \to & \Gamma(E) \nonumber \\
(e_1, e_2) & \mapsto & \nabla_{e_1}e_2 : = \lcf e_1, e_2\rcf + \triangle_{e_2}e_1
\end{eqnarray*}
\emph{defines a nonlinear $E$-connection on $E$ whose symbol is the vector bundle map $\alpha$.}
\item[\emph{7}.]
Some (non)linear $E$-connections constructed with use of a linear $TM$-connection on $E$.
\emph{Let $\triangle : \Gamma(TM)\times \Gamma(E) \to \Gamma(E)$ be a linear $TM$-connection on $E$.}
\begin{enumerate}
\item
\emph{The map}
\begin{eqnarray*}
\nabla : \Gamma(E) \times \Gamma(E) & \to & \Gamma(E) \nonumber \\
(e_1, e_2) & \mapsto & \nabla_{e_1}e_2 : = \lcf e_1, e_2\rcf + \triangle_{\rho(e_2)}e_1
\end{eqnarray*}
\emph{is a nonlinear $E$-connection on $E$ whose symbol is the vector bundle map $\alpha$.}
\item
\emph{Viewing, for any $e\in \Gamma(E)$, $\triangle_{\cdot} e$ as a section of $T^\ast M \otimes E$, we take that the map}
\begin{eqnarray*}
\nabla' : \Gamma(E) \times \Gamma(E) & \to & \Gamma(E)  \nonumber \\
(e_1, e_2) & \mapsto & \nabla'_{e_1}e_2 :=  \nabla_{e_1}e_2 - \alpha (\triangle_{\cdot} e_1)(e_2) \nonumber \\
& & \quad \quad \quad =  \lcf e_1, e_2\rcf + \triangle_{\rho(e_2)}e_1 - \alpha (\triangle_{\cdot} e_1)(e_2)
\end{eqnarray*}
\emph{is a linear $E$-connection on $E$.}
\item
\emph{The map}
\begin{eqnarray*}
\nabla : \Gamma(E) \times \Gamma(TM) & \to & \Gamma(TM) \nonumber \\
(e, X) & \mapsto & \nabla_eX : = [ \rho(e), X] + \rho(\triangle_Xe)
\end{eqnarray*}
\emph{is a linear $E$-connection on $TM$.}
\end{enumerate}
\item[\emph{8}.]
\emph{The Dorfman connections presented in \cite{bat-pet} when $E$ is a Courant algebroid.}
\end{enumerate}
\end{examples}

\begin{example}[The Lie algebroid of a Loday algebroid as a representation]\emph{We \\work in the setting of the paragraph \emph{The Lie algebroid of the Loday algebroid} \ref{Loday to Lie}. Under the assumptions of that paragraph, the Lie algebroid $\bar{E}=E/E_0$ is a representation of $E$ with respect to the nonlinear $E$-connection (Bott connection)}
\begin{eqnarray*}
\nabla : \Gamma(E) \times \Gamma(\bar{E}) & \to & \Gamma(\bar{E}) \nonumber \\
(e, \bar{e'}) & \mapsto & \nabla_e \bar{e'}: = \overline{\lcf e, e'\rcf}.
\end{eqnarray*}
\emph{Its symbol is the vector bundle map $\sigma : T^\ast M \otimes E \to \mathrm{End}(\bar{E})$ given, for any $e, e'\in \Gamma(E)$ and $f\in \C$,
\[\sigma(df \otimes e)(\bar{e'}) = - \mathrm{pr}\big((\bar{\rho}^T \otimes \mathrm{Id}_E)(df \otimes e)(\bar{e'})\big) = - \bar{\rho}(\bar{e'})(f)\bar{e} = -\rho(e')(f)\bar{e}, \]
where $\mathrm{pr}: E \to \bar{E}$ is the natural projection map.}
\end{example}

\subsection{Nonlinear connections and Loday algebroids comorphisms}
Let $(E, \lcf \cdot, \cdot \rcf_{_E}, \rho_{_E}, \alpha_{_E})$ and $(F, \lcf \cdot, \cdot \rcf_{_F}, \rho_{_F}, \alpha_{_F})$
be two Loday algebroids over $M$ and $N$, respectively, $(\hat{\Phi},\phi)$ a Loday algebroid comorphism with $\phi$ a diffeomorphism, and $B$ a vector bundle over $M$.
Suppose that $M$ is covered by framed open sets $(U, b_1, \ldots , b_s)$ for $B$, $s= \mathrm{rank}B$, and let $(\phi^{-1}(U), \phi^\ast b_1, \ldots, \phi^\ast b_s)$
be the corresponding covering of $N$ by framed open sets for $\phi^\ast B$.

\begin{proposition}\label{from F-conn to E-conn}
Under the above assumptions and notations, a nonlinear $F$-connection $\nabla^F$ on $\phi^\ast B$ defines a (unique) nonlinear $E$-connection $\nabla^E$ on $B$ such that,
for any $e\in \Gamma(E)$ and $b\in \Gamma(B)$,
\begin{equation}\label{compatible E-F-connections}
\phi^\ast \nabla^E_eb = \nabla^F_{\Phi(e)}\phi^\ast b.
\end{equation}
The corresponding connection and curvature matrices, with respect to suitable frames on open sets of $N$ and $M$, respectively, are related as follows:
\[ \Phi^\star_1 \theta^F = \theta^E \quad \quad \mathrm{and} \quad \quad \Phi^\star_2 \Theta^F = \Theta^E.\]
While, for the symbols $\sigma^E$ and $\sigma^F$ of $\nabla^E$ and $\nabla^F$, respectively, we have:
\[\phi^\ast(\sigma^E(df \otimes e)(b)) = \sigma^F(d\phi^\ast f \otimes \Phi(e))(\phi^\ast b), \quad \quad f\in \C.\]
\end{proposition}
\begin{proof}
Since $\phi$ is supposed to be a diffeomorphism, $\Gamma(\phi^\ast B) \cong \phi^\ast (\C)\otimes \Gamma(E)$.
In this framework, it is easy to verify that the operator $\nabla^E$ defined by \eqref{compatible E-F-connections}
is a nonlinear $E$-connection on $B$ with the mentioned properties.
\end{proof}

\begin{definition}
We say that a nonlinear $E$-connection $\nabla^E$ on $B$ and a nonlinear $F$-connection $\nabla^F$ on $\phi^\ast B$
are $(\hat{\Phi},\phi)$-compatible if, for any $e \in \Gamma(E)$ and $b\in \Gamma(B)$, \eqref{compatible E-F-connections} holds.
\end{definition}

\noindent
Also, the map $(\hat{\Phi}, \phi)$, where $\phi$ is a diffeomorphism, defines a family of maps, also denoted by $ \Phi^\star = (\Phi_p^\star)_{p\in \N}$,
between the spaces $\mathcal{D}^p(\mathcal{F}, \Gamma(\phi^\ast B))$ and $\mathcal{D}^p(\mathcal{E}, \Gamma(B))$ which assigns,
to each $p$-differential operator $\Delta$ on $\mathcal{F}$ with values in $\Gamma(\phi^\ast B)$, $p\in \N$,
a differential operator $\Phi^\star_p \Delta$ on $\mathcal{E}$ with values in $\Gamma(B)$ such that the next diagram is commutative:
\begin{equation*}
\begin{tikzcd}
\underbrace{\mathcal{E}\times \ldots \times \mathcal{E}}_{p-times} \arrow{r}{\Phi^p} \arrow[swap]{d}{\Phi^\star_p \Delta} & \underbrace{\mathcal{F}\times \ldots \times \mathcal{F}}_{p-times} \arrow{d}{\Delta} \\
\Gamma(B) \arrow{r}{\cong}[swap]{\phi^\ast} & \Gamma(\phi^\ast B).
\end{tikzcd}
\end{equation*}
Namely, for any $p$-tuple $(e_1,\ldots, e_p)$ of $\mathcal{E}$,
\begin{equation*}\label{def - Phi^star - B}
\phi^\ast(\Phi_p^\star \Delta(e_1, \ldots, e_p)) = \Delta(\Phi(e_1), \ldots, \Phi(e_p)).
\end{equation*}
Then, using the compatibility condition \eqref{compatible E-F-connections} between $\nabla^E$ and $\nabla^F$,
a straightforward computation yields the following relation between the covariant derivatives
$d^{\nabla^E} : \mathcal{D}^\bullet (\mathcal{E}, \Gamma(B)) \to \mathcal{D}^{\bullet+1} (\mathcal{E}, \Gamma(B))$ and $d^{\nabla^F} :
\mathcal{D}^\bullet (\mathcal{F}, \Gamma(\phi^\ast B)) \to \mathcal{D}^{\bullet+1} (\mathcal{F}, \Gamma(\phi^\ast B))$.
For any $\Delta \in \mathcal{D}^p(\mathcal{F}, \Gamma(\phi^\ast B))$, $p\in \N^{^\ast}$,
\begin{equation}\label{compatibility d^E-d^F}
d^{\nabla^E} \Phi^\star_p \Delta = \Phi^\star_{p+1}d^{\nabla^F}\Delta.
\end{equation}

\section{Characteristic classes of Loday algebroids}\label{section char classes}
In this section, as an application of our earlier results on Loday algebroid cohomology, we develop a theory of characteristic classes for Loday algebroids, following an approach inspired by previous works on the characteristic classes of Lie \cite{rui, cr-fer-sec-cl, balcerzak} and Courant \cite{CM2021} algebroids.

\subsection{The modular class of Loday algebroids and of Loday algebroid comorphisms}
\subsubsection{The modular class of Loday algebroids}
The concept of the \emph{modular class of a Loday algebroid} first appeared in \cite{stienon-xu},
but in a (completely) different setting from ours. Sti\'enon and Xu \cite{stienon-xu} following the classical procedure of \cite{evens-lu-weinstein},
consider linear connections of Loday algebroids (in their sense) $E$
and work in the naive cochain complex $(\Gamma(\bigwedge^{\bullet}\ker \rho), \check{d})$ of $E$.
The latter is a complex that is unrelated to the one considered in this paper.
Here, we also follow the approach of \cite{evens-lu-weinstein},
but our construction uses nonlinear representations of $E$ on (real) line bundles and defines a characteristic class in the first cohomology group
$H^1(\mathcal{E})$ of the Loday algebroid cohomology of $E$.

\vspace{2mm}
\noindent
Let $(E, \lcf \cdot, \cdot \rcf, \rho, \alpha)$ be a Loday algebroid over a smooth manifold $M$ and $\nabla$ a nonlinear $E$-connection on a real line bundle $\mathbb{L}$ over $M$.
Since $\mathbb{L}$ is a line bundle, the generating set of the $\C$-module $\mathrm{End}(\Gamma(\mathbb{L}))$ is $\{\mathrm{Id}_{\Gamma(\mathbb{L})}\}$ and the symbol $\sigma : \Gamma(T^\ast M\otimes E) \to \mathrm{End}(\Gamma(\mathbb{L}))$ of $\nabla$ is of type
\begin{equation*}\label{symb-line bundle}
\sigma(df \otimes e) = g_{_{df \otimes e}}\otimes \mathrm{Id}_{\Gamma(\mathbb{L})},
\end{equation*}
where $g_{_{df \otimes e}}\in \C$.

\vspace{1mm}
\noindent
Suppose that $\mathbb{L}$ is a trivial real line bundle and choose a nonvanishing section $s$ of $\mathbb{L}$,
which consists a (global) frame of smooth sections of $\mathbb{L}$. Then, for any $e\in \Gamma(E)$,
\begin{equation*}
\nabla_e s = \theta_s(e)s,
\end{equation*}
where $\theta_s \in \mathcal{D}^1(\mathcal{E})$ is the connection $1$-differential operator of $\nabla$ relative to the frame $(s)$.

\begin{proposition}
Under the above assumptions,
\begin{enumerate}
\item
$R^{\nabla}= \partial_E \theta_s $. Thus, $\mathbb{L}$ is a representation of $E$ if and only if $\theta_s$ is a $1$-cocycle with respect to $\partial_E$.
\item
If $s'$ is another nonvanishing section of $\mathbb{L}$, then $s' = f s$, for a nonvanishing function $f\in \C$, and
\begin{equation*}
\theta_{s'} = \theta_s + \partial_E \ln|f|.
\end{equation*}
Therefore, if $\mathbb{L}$ is a representation of E, the class $[\theta_s]\in H^1(\mathcal{E})$ is well defined and independent of the chosen section $s\in \Gamma(\mathbb{L})$.
\end{enumerate}
\end{proposition}
\begin{proof} Under the above notations, we have:
\begin{enumerate}
\item
Let $\Theta_s$ be the curvature $2$-differential operator of $\nabla$ relative to $s$. Then, according to relations \eqref{curvat-Theta} and \eqref{THETA}, for any $e_1,e_2 \in \Gamma(E)$,
\[R^{\nabla}(e_1,e_2)s = \Theta_s(e_1,e_2)s\]
and
\[\Theta_s(e_1,e_2) = \partial_E \theta_s(e_1,e_2) - \theta_s \diamond\theta_s (e_1,e_2) = \partial_E \theta_s(e_1,e_2).\]
Hence, $\mathbb{L}$ is a representation of $E$ relatively to $\nabla$ if and only if $\theta_s$ is a $1$-cocycle with respect to $\partial_E$.
\item
Let $s'$ and $f$ be as in our hypothesis and $e\in \Gamma(E)$. Then
\begin{equation*}
\theta_{s'}(e)s' = \nabla_es' = \nabla_e(fs) = f\nabla_es + \rho(e)(f)s = \theta_s(e)s' + \frac{\rho(e)(f)}{f}s',
\end{equation*}
and the results follow.
\end{enumerate}
\end{proof}

\vspace{2mm}
\noindent
We call $\theta_s$ the \emph{characteristic cocycle associated to the representation $\mathbb{L}$ and the section $s$}. In the following, we omit the index $s$ and write $[\theta]$,
since the class $[\theta_s]$ does not depend on $s$ but rather depends  only on the representation $\mathbb{L}$.

\begin{definition}
We call the cohomology class $[\theta]$ of a trivial real line bundle representation $\mathbb{L}$ of a nonlinear $E$-connection $\nabla$ \emph{characteristic class of $\mathbb{L}$ relatively to $\nabla$} and write
\[\mathrm{char}^{\nabla}(\mathbb{L}) = [\theta].\]
\end{definition}
If $\mathrm{char}^{\nabla}(\mathbb{L}) = 0$, then $\theta = \partial_Ef$ for a function $f\in \C$, which means that $\theta$ is a $\C$-linear operator.

\vspace{2mm}
\noindent
Let $\mathbb{L}_1$ and $\mathbb{L}_2$ be two real line bundles representations of $E$; $\mathbb{L}_1$ with respect of a nonlinear $E$-connection $\nabla^1$ and
$\mathbb{L}_2$ with respect of a such connection $\nabla^2$. We equip their tensor product $\mathbb{L}_1 \otimes \mathbb{L}_2$
with the tensor product nonlinear $E$-connection $\nabla = \nabla^1\otimes \mathrm{Id}_{\mathbb{L}_2} + \mathrm{Id}_{\mathbb{L}_1}\otimes \nabla^2$:
\[\nabla_e (s_1 \otimes s_2) = (\nabla^1_e s_1) \otimes s_2 + s_1 \otimes (\nabla^2 s_2), \quad \quad s_1 \in \Gamma(\mathbb{L}_1), \, s_2 \in \Gamma(\mathbb{L}_2).\]

\begin{proposition}

\vspace{1mm}
\noindent
\begin{enumerate}
\item
Let $\mathbb{L}_1$ and $\mathbb{L}_2$ be trivial line bundles. Then, the characteristic class of the trivial real line bundle $\mathbb{L}_1 \otimes \mathbb{L}_2$
and those of $\mathbb{L}_1$ and $\mathbb{L}_2$ are related as follows:
\begin{equation}\label{char-tensor}
\mathrm{char}^{\nabla}(\mathbb{L}_1\otimes \mathbb{L}_2) = \mathrm{char}^{\nabla^1}(\mathbb{L}_1) + \mathrm{char}^{\nabla^2}(\mathbb{L}_2).
\end{equation}
\item
In the particular case where $(\mathbb{L}_1, \nabla^1)=(\mathbb{L}, \nabla)$ and $(\mathbb{L}_2, \nabla^2)=(\mathbb{L}^\ast, \nabla^\ast)$, we have
\begin{equation*}\label{char-dual}
\mathrm{char}^{\nabla}(\mathbb{L}) + \mathrm{char}^{\nabla^\ast}(\mathbb{L}^\ast)=0.
\end{equation*}
\item
Let $(E, \lcf \cdot, \cdot \rcf_{_E}, \rho_{_E}, \alpha_{_E})$ and $(F, \lcf \cdot, \cdot \rcf_{_F}, \rho_{_F}, \alpha_{_F})$
be two Loday algebroids over $M$ and $N$, respectively, $(\hat{\Phi},\phi)$ a Loday algebroid comorphism, $\phi$ being a diffeomorphism,
$\mathbb{L}$ a trivial vector bundle over $M$ and $\phi^\ast \mathbb{L}$ the pullback line bundle over $N$. Let, also, $\nabla^E$ and $\nabla^F$ be $(\hat{\Phi},\phi)$-compatible nonlinear connections on $\mathbb{L}$ and $\phi^\ast \mathbb{L}$, respectively. If $\phi^\ast \mathbb{L}$ is a representation of $F$ relatively to $\nabla^F$, then $\mathbb{L}$ is a representation of $E$ relatively to $\nabla^E$ and
\begin{equation}\label{Phi-characteristic classes}
\Phi_1^\star (\mathrm{char}^{\nabla^F}(\phi^\ast \mathbb{L})) = \mathrm{char}^{\nabla^E}(\mathbb{L}).
\end{equation}
\end{enumerate}
\end{proposition}
\begin{proof}

\vspace{1mm}
\noindent
\begin{enumerate}
\item
By a simple computation, we obtain that the characteristic cocycles associated with the referred representations are related by the following relation:
\[\theta_{s_1 \otimes s_2} = \theta_{s_1} + \theta_{s_2},\]
for some global sections $s_1 \in \Gamma(\mathbb{L}_1)$ and $s_2 \in \Gamma(\mathbb{L}_2)$. So, the identity \eqref{char-tensor} is true. In the particular case where $(\mathbb{L}_1, \nabla^1) = (\mathbb{L}_2, \nabla^2) = (\mathbb{L}, \nabla)$, we have\footnote{We write also $\nabla$
for denote the tensor product nonlinear $E$-connection $\nabla \otimes \mathrm{Id}_{\mathbb{L}} + \mathrm{Id}_{\mathbb{L}}\otimes \nabla$.}
\[\mathrm{char}^{\nabla}(\mathbb{L}) = \frac{1}{2}\mathrm{char}^{\nabla}(\mathbb{L}\otimes \mathbb{L}).\]
\item
It is an immediate consequence of the relations \eqref{char-tensor} and \eqref{def dual connection}.
\item
By Proposition \ref{from F-conn to E-conn}, if $\phi^\ast \mathbb{L}$ is a representation of $F$ relatively to $\nabla^F$, then $\mathbb{L}$ is a representation of $E$ relatively to $\nabla^E$. Let $s$ be a nonvanishing section of $\mathbb{L}$, $\phi^\ast s$ the corresponding generator of $\Gamma(\phi^\ast \mathbb{L})$, and $\theta^F_{\phi^\ast s}$ a representative of $\mathrm{char}^{\nabla^F}(\phi^\ast \mathbb{L})$. Then, $\Phi_1^\star \theta^F_{\phi^\ast s}$ is a representative of $\mathrm{char}^{\nabla^E}(\mathbb{L})$ since, for any $e\in \Gamma(E)$,
\begin{equation*}
\phi^\ast(\nabla^E_es) = \nabla^F_{\Phi(e)} \phi^\ast s = \theta^F_{\phi^\ast s}(\Phi(e)) \phi^\ast s =  \phi^\ast (\Phi_1^\star \theta^F (e)s)
\end{equation*}
and, taking into account the Proposition \ref{Phi LA morph --> Phi* chain map}, we have that \eqref{Phi-characteristic classes} holds.
\end{enumerate}
\end{proof}

\vspace{2mm}
\noindent
If $\mathbb{L}$ is a non trivial real line bundle representation of $E$, because it is isomorphic to its dual line
bundle $\mathbb{L}^\ast$\footnote{This is true under the common assumption that $M$, as topological manifold, is paracompact and Hausdorff \cite{husemoller}.},
we have that $\mathbb{L}\otimes \mathbb{L} \cong \mathbb{L}^\ast \otimes \mathbb{L} \cong \mathrm{End}(\mathbb{L}) = \langle \mathrm{Id}_{\mathbb{L}}\rangle$,
which is a trivial real line bundle over $M$. Thus, its characteristic class $\mathrm{char}^{\nabla}(\mathbb{L}\otimes \mathbb{L})$ is well defined. Therefore, in this case, we define the characteristic class of $\mathbb{L}$ by setting
\[\mathrm{char}^{\nabla}(\mathbb{L}): = \frac{1}{2}\mathrm{char}^{\nabla}(\mathbb{L}\otimes \mathbb{L}).\]

\begin{example}
The characteristic class of an orientable smooth manifold $M$. \emph{Let $(E, \lcf \cdot, \cdot \rcf, \rho, \alpha)$
be a Loday algebroid, $\mu$ a volume form on $M$ and $\mathbb{L} = \bigwedge^n T^\ast M$ the real line bundle representation of $E$
with respect to the nonlinear $E$-connection $\nabla^M = \mathcal{L}_{\rho(\cdot)}$ (see paragraph \ref{examples - nonlinear connections} - Example 5).
The characteristic cocycle of $E$ is the first order differential operator $\theta \in \mathcal{D}^1(\mathcal{E})$, $\theta(e) = (\mathrm{div}\circ \rho)(e)$,
and $\mathrm{char}^{\nabla^M}(\bigwedge^n T^\ast M) = [\mathrm{div}\circ \rho]$.}
\end{example}

\vspace{1mm}
\noindent
As we have shown in Example 2 of Examples \ref{examples - nonlinear connections}, the real line vector bundle $\mathbb{L}= \bigwedge^{\mathrm{top}}E$ is a
representation of $E$ with respect to the extended adjoint representation
$\nabla^{\mathrm{top}}=\Lie$ and $\mathrm{char}^{\nabla^{\mathrm{top}}}(\bigwedge^{\mathrm{top}}E) = [\tr(\theta^{\mathrm{ad}})]$.
\begin{definition}\label{def-modular class}
The characteristic class $\mathrm{char}^{\nabla^{\mathrm{top}}}(\bigwedge^{\mathrm{top}}E)$ is called the \emph{modular class of
the Loday algebroid $E$} and is denoted by $\mathrm{mod}(E)$, while a representative of $\mathrm{mod}(E)$ is called modular cocycle of $E$.
\end{definition}

\vspace{2mm}
\noindent
We say that a Loday algebroid $E$ is \emph{unimodular} if $\mathrm{mod}(E)=0$. In this case, $\nabla^{\mathrm{top}}$
is linear, i.e. the symbol of the connection operator $\theta^{\mathrm{top}} = \tr(\theta^{\mathrm{ad}})$ vanishes.
So, we have that
\begin{equation}\label{unimodularity}
\tr(\alpha(df\otimes e))-\rho(e)(f)=0.
\end{equation}

\begin{examples}

\vspace{1mm}
\noindent
\begin{enumerate}
\item
The modular class of a Loday algebra. \emph{Let $(\mathcal{E}, \lcf \cdot, \cdot \rcf)$ be a Loday algebra. The modular cocycle of $\mathcal{E}$ is the $\R$-linear map $\theta : \mathcal{E} \to \R$ given by $\theta(e) = \tr(\mathrm{ad}_e)$ and}
\[\mathrm{mod}(\mathcal{E})=[\tr(\mathrm{ad})].\]
\item
The modular class of the Lie algebroid $(TM, [\cdot,\cdot], \mathrm{Id})$. \emph{With respect the adjoint representation,
it is the class $[\mathrm{div}]$ in the Loday algebroid cohomology.}
\item
The modular class of a Courant algebroid. \emph{From \eqref{unimodularity}, we obtain that every Courant algebroid is unimodular in our cohomology, also. The unimodularity of the Courant algebroids is established in \cite{CM2021} in the context of Keller-Waldmann cohomology.}
\item
The modular class of the Loday algebroid associated to a Nambu-Poisson structure. \emph{We consider the Loday algebroid
$(\bigwedge^{n-1} T^\ast M, \lcf \cdot, \cdot \rcf, \rho, \alpha)$ associated to a Nambu-Poisson manifold $(M, \Lambda)$ of dimension $m$ (see Section \ref{section - loday alg}). According to Definition \ref{def-modular class}, $\mathrm{mod}(\bigwedge^{n-1} T^\ast M) = [\tr(\theta^{\mathrm{ad}})]$.
We go to compute the $\tr(\theta^{\mathrm{ad}})$.}
\emph{Suppose that $\Lambda$ is decomposable and non-degenerated of order $n\geq 3$, and consider a chart $(U, x^1,\ldots,x^m)$ of $M$ such that
$\Lambda=\displaystyle{\frac{\partial}{\partial x_1}\wedge \ldots \wedge \frac{\partial}{\partial x_n}}$.
Also, each $\omega\in \Gamma(\bigwedge^{n-1} T^\ast M)$ can be decomposed as
$\omega=\sum_{i=1}^n \omega_i d x^1\wedge \ldots \wedge \widehat{d x^i} \wedge \ldots \wedge d x^n + \omega'$, with $\omega_i \in C^\infty(U,\R)$,
and $\omega'\in \Gamma(\bigwedge^{n-1} T^\ast M\vert_U)$ being the rest part of $\omega$. Then,
$\rho(\omega)=\Lambda^{\sharp}(\omega)=\displaystyle{\sum_{i=1}^n(-1)^{n-i}\omega_i\frac{\partial}{\partial x^i}}$ and
$\Lambda^{\sharp}(d\omega)=\displaystyle{\sum_{i=1}^n (-1)^{i-1}\frac{\partial \omega_i}{\partial x^i}}$.
A frame of $\Gamma(\bigwedge^{n-1} T^\ast M\vert_U)$ is given by the
$\displaystyle{\binom{m}{n-1}}$ elements $d x^I=d x^{i_1}\wedge \ldots \wedge d x^{i_{n-1}}$,
where $I=(i_1,\ldots, i_{n-1})\in \{1, \ldots,m\}^{n-1}$ with $i_1< \ldots< i_{n-1}$, while its dual frame is given by the $(n-1)$-vector fields
$\displaystyle{\frac{\partial}{\partial x^I}=\frac{\partial}{\partial x^{i_1}}\wedge \ldots \wedge \frac{\partial}{\partial x^{i_{n-1}}}}$.
Notice that each $d x^i$, $i=1,\ldots,m$, takes part into $\displaystyle{\binom{m-1}{n-2}}$ elements of the frame.
Let us denote by $I_k$ the set of indexes $(i_1,\ldots,i_{n-2})$ (of length $n-2$) that don't contain $k\in \{1,\ldots,m\}$. Hence,}
\begin{eqnarray*}
\lefteqn{\tr(\theta^{\mathrm{ad}}(\omega)) = \sum_{I}\langle \lcf \omega, dx^I \rcf,\,\frac{\partial}{\partial x^I}  \rangle }\nonumber \\
& = & \sum_I \Big(\langle \mathcal{L}_{\Lambda^{\sharp}(\omega)}dx^I, \, \frac{\partial}{\partial x^I}\rangle
+ (-1)^n \Lambda(d \omega)\langle dx^I, \, \frac{\partial}{\partial x^I}\rangle \Big)\nonumber \\
& = & \sum_I \sum_{i=1}^n (-1)^{n-i} \langle d(i_{\omega_i\frac{\partial}{\partial x^i}}dx^I), \, \frac{\partial}{\partial x^I}  \rangle
+ (-1)^n \binom{m}{n-1}\Lambda(d \omega) \nonumber \\
& = &  \sum_{i=1}^n\sum_{I_i} (-1)^{n-i} \frac{\partial \omega_i}{\partial x^i} \langle dx^i \wedge dx^{I_i},\,
\frac{\partial}{\partial x^i}\wedge\frac{\partial}{\partial x^{I_i}}\rangle + (-1)^n \binom{m}{n-1}\sum_{i=1}^n(-1)^{i-1}\frac{\partial \omega_i}{\partial x^i}\nonumber \\
& = & \sum_{i=1}^n\Big( (-1)^{n-i}\binom{m-1}{n-2} + (-1)^{n+i-1}\binom{m}{n-1} \Big)\frac{\partial \omega_i}{\partial x^i} \nonumber \\
& = & \binom{m-1}{n-1}\sum_{i=1}^{n}(-1)^{n-i+1}\frac{\partial \omega_i}{\partial x^i} \nonumber \\
& = & - \binom{m-1}{n-1}\mathrm{div}_{\mu}(\rho(\omega)),
\end{eqnarray*}
\emph{where $\mathrm{div}_{\mu}(\rho(\omega))$ is the divergence of the vector field $\rho(\omega)$ with respect to the volume form $\mu = dx^1\wedge \ldots \wedge dx^m$ of $M$. So}
\[\mathrm{mod}(\bigwedge^{n-1} T^\ast M) = [- \binom{m-1}{n-1}\mathrm{div}_{\mu}\circ \rho ].\]
\item
The modular class of the Grassmann-Dorfman algebroid.
\emph{We consider the Grassmann-Dorfmann algebroid $(TM \oplus \bigwedge^p T^\ast M, \lcf \cdot, \cdot \rcf, \rho, \alpha)$ over a $n$-dimensional manifold
presented in Section \ref{section - loday alg}. We have $\mathrm{mod}(TM \oplus \bigwedge^p T^\ast M) = [\tr(\theta^{\mathrm{ad}})]$. We use the notation introduced in the
previous example and we develop similar computations. For a vector field $X = \sum_{i=1}^n\chi^i \displaystyle{\frac{\partial}{\partial x^i}}$ on $M$ and for a
multi-index $I=(i_1,\ldots,i_p)$, $1\leq i_1<\ldots < i_p \leq n$, we note by $X_I$ the $I$-component of $X$, i.e.
$X_I= \sum_{k=1}^p\chi^{i_k}\displaystyle{\frac{\partial}{\partial x^{i_k}}}$, and by $\pi_{_I} : \Gamma(TM) \to \Gamma(TM)$ the projection of $X$ on $X_I$. Then, we compute that }
\[\tr(\theta^{\mathrm{ad}}(X, \zeta)) = -\mathrm{div}_{\mu}(X) + \sum_I(\frac{\partial \chi^{i_1}}{\partial x^{i_1}}+ \ldots + \frac{\partial \chi^{i_p}}{\partial x^{i_p}}).\]
\emph{Thus,}
\[\mathrm{mod}(TM \oplus \bigwedge^p T^\ast M) = [-\mathrm{div}_{\mu}\circ \rho + \sum_I \mathrm{div}_{\mu}\circ \pi_{_I}\circ \rho ].\]
\end{enumerate}
\end{examples}

\subsubsection{Modular class of Loday algebroid comorphisms}
We study the behavior of the modular class of a Loday algebroid under Loday algebroid comorphisms over the same manifold
by defining the modular class of a Loday algebroid comorphism. Our results extend similar results for Lie algebroid morphisms
(see \cite{kos-cam-wein} and references cited therein.)

\vspace{1mm}
\noindent
Let $(E, \lcf \cdot, \cdot \rcf_{_E}, \rho_{_E}, \alpha_{_E})$ and $(F, \lcf \cdot, \cdot \rcf_{_F}, \rho_{_F}, \alpha_{_F})$
be two Loday algebroids over the same manifold $M$ and $(\hat{\Phi}, id)$ a Loday algebroid
comorphism from $E$ to $F$ (Definition \ref{def-Loday-comorphism}), which is equivalent
to a chain map $\Phi^\star : (\mathcal{D}(\mathcal{F}),\partial_F)\to (\mathcal{D}(\mathcal{E}),\partial_E)$
(Propositions \ref{Phi LA morph --> Phi* chain map} and \ref{Phi* chain map --> Phi LA morph}).

\begin{definition}
Let $\mathrm{mod}(E)$ and $\mathrm{mod}(F)$ be the modular classes of $E$ and $F$, respectively, and $(\hat{\Phi}, id)$ a Loday algebroid comorphism from $E$ to $F$.
The \emph{modular class of $(\hat{\Phi}, id)$} or simply the \emph{modular class of $\Phi : \mathcal{E} \to \mathcal{F}$}
is the $\partial_E$-cohomology class $\mathrm{mod}^{\Phi}(E, F)$ defined by
\begin{equation}\label{mod Phi}
\mathrm{mod}^{\Phi}(E, F) = \mathrm{mod}(E)-\Phi^\star \mathrm{mod}(F).
\end{equation}
\end{definition}

\vspace{1mm}
\noindent
Placed in the context of Proposition \ref{composition Loday comorphisms}, and assuming that the Loday algebroids are defined over the same base manifold and that the comorphisms cover the identity, we obtain that
\begin{equation*}
\mathrm{mod}^{\Psi\circ \Phi}(E_1, E_3) = \mathrm{mod}^{\Phi}(E_1, E_2) + \Phi^\star \mathrm{mod}^{\Psi}(E_2, E_3).
\end{equation*}

\vspace{1mm}
\noindent
The modular class $\mathrm{mod}^{\Phi}(E, F)$ can be represented as the characteristic class of a representation of $E$.
Precisely, we consider the line bundle $\mathbb{L}^{(E,F)}: = \bigwedge^{\mathrm{top}}E\otimes \bigwedge^{\mathrm{top}}F$ and the map
\begin{eqnarray*}\label{modular repres - Phi}
\nabla^{\Phi} : \Gamma(E)\times \Gamma(\mathbb{L}^{(E,F)}) & \to & \Gamma(\mathbb{L}^{(E,F)}) \nonumber \\
(e, s_{_E}\otimes s_{_F}) & \mapsto & \nabla^{\Phi}_e (s_{_E}\otimes s_{_F}) = (\Lie^E_e s_{_E})\otimes s_{_F} - s_{_E}\otimes (\Lie^F_{\Phi(e)}s_{_F}),
\end{eqnarray*}
where $s_{_E}\otimes s_{_F}$ is a smooth section of $\mathbb{L}^{(E,F)}$, $s_{_E}$ and $s_{_F}$ being smooth
sections of $\bigwedge^{\mathrm{top}}E$ and $\bigwedge^{\mathrm{top}}F$, respectively, and, $\Lie^E$ and $\Lie^F$
are the extensions of the adjoint representations of $E$ and $F$ on $\bigwedge^{\mathrm{top}}E$ and $\bigwedge^{\mathrm{top}}F$,
respectively (Examples \ref{examples - nonlinear connections}).
\begin{proposition}
Under the above assumptions,
\begin{enumerate}
\item
the line bundle $\mathbb{L}^{(E,F)}$ is a representation of $E$ relatively to the map $\nabla^{\Phi}$.
\item
the modular class $\mathrm{mod}^{\Phi}(E, F)$ is the characteristic class of the representation $\mathbb{L}^{(E,F)}$ relatively to $\nabla^{\Phi}$.
\end{enumerate}
\end{proposition}
\begin{proof}
The first assertion can be easily proved by a direct calculation, taking into account the Cartan calculus presented in Lemma
\ref{lemma Cartan} and the morphism properties \eqref{def - Loday algebra homomorphism} and \eqref{comorphism - left anchor} of $\Phi$.
The symbol of $\nabla^{\Phi}$ is the vector bundle map $\sigma: T^\ast M \otimes E\to \mathrm{End}(\mathbb{L}^{(E,F)})$
given by $\sigma(df \otimes e) = [\tr(\alpha_{_E}(df \otimes e)) - \tr(\alpha_{_F}(df \otimes \Phi(e)))]\otimes \mathrm{Id}_{\mathbb{L}^{(E,F)}}$.

\noindent
For the proof of the second assertion, let $s_{_E}\otimes s_{_F}$ be a nowhere vanishing smooth section of $\mathbb{L}^{(E,F)}$, $\theta^{\Phi}$
the connection matrix of $\nabla^{\Phi}$, and $\theta^E$ and $\theta^F$ the corresponding matrices of $\Lie^E$ and $\Lie^F$, respectively. We have
\begin{equation*}
\nabla^{\Phi}_e (s_{_E}\otimes s_{_F}) = \theta^{\Phi}(e)(s_{_E}\otimes s_{_F})
\end{equation*}
and
\begin{eqnarray*}
\nabla^{\Phi}_e (s_{_E}\otimes s_{_F}) & = & (\Lie^E_e s_{_E})\otimes s_{_F} - s_{_E}\otimes (\Lie^F_{\Phi(e)}s_{_F}) =
(\theta^E(e)- \theta^F(\Phi(e)))(s_{_E} \otimes s_{_F}) \nonumber \\
& \stackrel{\eqref{def - Phi^star}}{=}& (\theta^E- \Phi^\star_1\theta^F)(e)(s_{_E} \otimes s_{_F}).
\end{eqnarray*}
So,
\[\theta^{\Phi} = \theta^E- \Phi^\star_1\theta^F\]
and
\[[\theta^{\Phi}] = [\theta^E]- \Phi^\star[\theta^F] = \mathrm{mod}(E)-\Phi^\star \mathrm{mod}(F) = \mathrm{mod}^{\Phi}(E,F),
\footnote{We omit the index $1$ for the map $\Phi^\star$.}\]
which means that $\theta^{\Phi}$ is a representative of $\mathrm{mod}^{\Phi}(E,F)$.
\end{proof}

\subsection{The Chern-Weil homomorphism}
In this subsection, we present some basic results of Chern-Weil theory adapted to nonlinear connections of Loday algebroids that are needed later for the discussion of their secondary characteristic classes.

\noindent
Let $(E, \lcf \cdot, \cdot \rcf, \rho, \alpha)$ be a Loday algebroid over $M$, $B$ a vector bundle over $M$ of constant rank $s$, and $\nabla$ a nonlinear $E$-connection on $B$. Let, also, $\theta$ and $\Theta$ be the connection and the curvature matrix, respectively, of $\nabla$ relative to a frame of smooth sections $(b_1,\ldots,b_s)$ of $B$ on an open subset $U$ of $M$. If $(b_1', \ldots, b_s')$ is another frame of smooth sections of $B$ on $U$, then $b_i'=a_i^jb_j$, $a_i^j \in C^\infty(U,\R)$, and the matrix $a=\left[\begin{array}{c} a_i^j \\ \end{array}\right]$
can be considered as a smooth map $a : U \to \mathrm{GL}(s,\R)$. By Proposition \ref{prop - theta - theta'}, the curvature matrix $\Theta'$ of
$\nabla$ relative to $(b_1', \ldots, b_s')$ is $\mathrm{Ad}_a$-conjugated with $\Theta$: $\Theta' = a \diamond \Theta \diamond a^{-1}$.

\vspace{1mm}
\noindent
For each $x\in M$, we consider the evaluation map $\mathrm{ev}_x : \C \to \R$, $f \mapsto \mathrm{ev}_x(f): = f(x)$.
This map transforms any multidifferential operator $D\in \mathcal{D}^p(\mathcal{E})$ into a function
$D_x : \mathcal{E}\times \ldots \times \mathcal{E} \to \R$ given, for any $e_1,\ldots,e_p\in \mathcal{E}$, by
\[D_x(e_1,\ldots,e_p): = \mathrm{ev}_x(D(e_1,\ldots,e_p)) = D(e_1,\ldots,e_p)(x).\]
Thus, each $1$-differential operator $\theta_i^j$ on $\mathcal{E}$, element of $\theta$,
is transformed to the function $\theta_{i,x}^j : \mathcal{E} \to \R$, $e\mapsto \theta_{i,x}^j(e) : = \mathrm{ev}_x (\theta_i^j(e))=\theta_i^j(e)(x)$, and,
each $2$-differential operator $\Theta_i^j$ on $\mathcal{E}$, element of $\Theta$, is transformed to the
function $\Theta_{i,x}^j: \mathcal{E}\times \mathcal{E} \to \R$,
$(e_1,e_2)\mapsto \Theta_{i,x}^j (e_1,e_2) : = \mathrm{ev}_x(\Theta_i^j (e_1,e_2))= \Theta_i^j (e_1,e_2)(x)$.
Set $\theta_x = \left[\begin{array}{c}\theta_{i,x}^j \\ \end{array}
                \right]$ and $\Theta_x : = \left[\begin{array}{c} \Theta_{i,x}^j \\ \end{array}\right]$; both are elements of $\mathfrak{gl}(s,\R)$.

\vspace{2mm}
\noindent
We recall that \cite{milnor-stasheff, kob-nom}:
\begin{enumerate}
\item
An $\mathrm{Ad}(\mathrm{GL}(s,\R))$-invariant polynomial on the space $\mathfrak{gl}(s,\R)$ is a function $P: \mathfrak{gl}(s,\R)\to \R$
which can be expressed as (real) polynomial in the entries of the matrix and satisfies, for any $A,B\in \mathfrak{gl}(s,\R)$ and $T\in \mathrm{GL}(s,\R)$,
\begin{equation}\label{invariance-sym}
P(AB) = P(BA)
\end{equation}
or, equivalently,
\begin{equation*}\label{invariance-adj}
P(\mathrm{Ad}_TA) = P(TAT^{-1})=P(A).
\end{equation*}
The trace of an endomorphism of a real vector space is an $\mathrm{Ad}(\mathrm{GL}(s,\R))$-invariant polynomial.
\item
The algebra $\mathrm{Inv}(\mathfrak{gl}(s,\R)) =  \sum_{p=0}^{\infty}\mathrm{Inv}^p(\mathfrak{gl}(s,\R))$, where $\mathrm{Inv}^p(\mathfrak{gl}(s,\R))$ is
the space of $\mathrm{Ad}(\mathrm{GL}(s,\R))$-invariant symmetric $p$-multilinear mappings of $\mathfrak{gl}(s,\R)$ into $\R$, may be identified with the algebra of $\mathrm{Ad}(\mathrm{GL}(s,\R))$-invariant (homogeneous) polynomial functions on $\mathfrak{gl}(s,\R)$.
\end{enumerate}

\vspace{1mm}
\noindent
Let $P$ be a symmetric, $\mathrm{Ad}(\mathrm{GL}(s,\R))$-invariant, $p$-multilinear function on $\mathfrak{gl}(s,\R)$, $p\in \N$. Namely, a $p$-linear map
\[P : \underbrace{\mathfrak{gl}(s,\R)\times \ldots \times \mathfrak{gl}(s,\R)}_{p - times}\to \R\]
such that, for all $A_1,\ldots, A_p \in \mathfrak{gl}(s,\R)$ and $T\in \mathrm{GL}(s,\R)$,
\[P(A_1,\ldots,A_i, \ldots,A_j,\ldots, A_p) = P(A_1,\ldots,A_j, \ldots,A_i,\ldots, A_p)\]
and
\[P(\mathrm{Ad}_TA_1, \ldots, \mathrm{Ad}_TA_p) = P(A_1,\ldots, A_p).\]
Fix $x\in U$ and define $P(\Theta_x)$ to be the function
\[P(\Theta_x) : \underbrace{\mathcal{E}\times \ldots \times \mathcal{E}}_{2p - times} \to \R\]
given, for any $e_1, \ldots, e_{2p}\in \mathcal{E}$, by
\begin{equation}\label{P-Theta_x}
P(\Theta_x)(e_1,\ldots,e_{2p}) = \sum_{\tau \in Sh(\underbrace{2, \ldots, 2}_{p-times})}(-1)^{|\tau|}P(\Theta_x(e_{\tau(1)}, e_{\tau(2)}),\, \ldots, \,\Theta_x(e_{\tau(2p-1)}, e_{\tau(2p)})),
\end{equation}
where the summation is taken over all $(\underbrace{2, \ldots, 2}_{p-times})$-shuffle
permutations of $(1,\ldots,2p)$ (see footnote \ref{footnote Sh}). By the invariance of $P$ under conjugation, we have
\[P(\Theta'_x) = P(a(x) \diamond \Theta_x \diamond a^{-1}(x)) = P(\Theta_x).\]
As $x$ varies on $U$, we obtain that the evaluation of the $2p$-multidifferential operator $P(\Theta)$ on $U$
is independent of the frame of smooth sections of $B$ on $U$, i.e. $P(\Theta) = P(\Theta')\in \mathcal{D}^{2p}(\mathcal{E}, C^\infty(U,\R))$.

\noindent
Let $(U_k)_{k\in K}$ be an open cover of $M$ trivializing for $B$, $(b_{k1},\ldots,b_{ks})$ a frame of smooth sections of $B$
on $U_k$, $k\in K$, and $\Theta_k$ the curvature matrix of $\nabla$ relatively to this frame.
Then, $P(\Theta_k)$ is a $2p$-multidifferential operator on $\mathcal{E}$ with values
in $C^\infty(U_k,\R)$. On the overlap $U_k\cap U_l$, $k,l\in K$, the two $2p$-multidifferential operators $P(\Theta_k)$ and $P(\Theta_l)$ coincide.
Therefore, the collection $(P(\Theta_k))_{k\in K}$ defines globally a $2p$-differential operator on $\mathcal{E}$ with values in $\C$ which
we denote by $P(\Theta)$. Then, we have the following fundamental result:

\begin{proposition}\label{prop-closed P(Theta)}
For any $\mathrm{Ad}(\mathrm{GL}(s,\R))$-invariant polynomial function $P$ on $\mathfrak{gl}(s,\R)$, the multidifferential operator $P(\Theta)$ on $\mathcal{E}$ is $\partial_E$-closed. Thus, it represents an element of the Loday algebroid cohomology $H^{\bullet}(\mathcal{E})$ noted by $[P(\Theta)]$.
\end{proposition}The proof of the above proposition is completely analogous to that in the classical case. The exterior product of forms on $M$ is ``replaced" by the shuffle product of multidifferential operators on $\mathcal{E}$. Then, using the properties of $P$ (linearity, symmetry, $\mathrm{Ad}(\mathrm{GL}(s,\R))$-invariance) and the Bianchi identity \eqref{bianchi id - theta}, we establish the above assertion.
For details, see \cite[Fundamental Lemma, p. 296]{milnor-stasheff}.

\vspace{2mm}
\noindent
Before we proceed, we note that, in the same spirit as above, for any $p$-tuple $(D_1, \ldots, D_p)$
of multidifferential operators on $\mathcal{E}$ with values in $\mathrm{End}(\Gamma(B))$, whose the $i$-element $D_i$ is a
$k_i$-multidifferential operator on $\mathcal{E}$, $i=1,\ldots,p$, we can define the $k$-multidifferential operator $P(D_1,\ldots,D_p)$,
$k=k_1+\ldots + k_p$, by setting, for any $e_1,\ldots,e_k \in \mathcal{E}$,
\begin{eqnarray}
\lefteqn{P(D_1,\ldots,D_p)(e_1,\ldots,e_k) =} \nonumber \\
& & \hspace{-8mm} \sum_{\tau \in Sh(k_1,\ldots, k_p)}(-1)^{|\tau|}P(D_1(e_{\tau(1)},\ldots,e_{\tau(k_1)}),\ldots,D_p(e_{\tau(k_1+\ldots + k_{p-1}+1)},\ldots,e_{\tau(k_1+\ldots + k_{p})})),
\end{eqnarray}
where $Sh(k_1,\ldots, k_p)$ is the set of $(k_1,\ldots,k_p)$-shuffle
permutations of $(1,\ldots,k)$.\footnote{That is, the elements of $Sh(k_1,\ldots, k_p)$
are the permutations $\tau$ of $(1,\ldots,k)$, $k=k_1+\ldots + k_p$, that verify the conditions $\tau(1)<\ldots < \tau(k_1)$,
$\tau(k_1+1)<\ldots < \tau(k_1+k_2)$, $\ldots$, $\tau(k_1+\ldots + k_{p-1}+1) < \ldots < \tau(k)$.\label{footnote Sh}}
The $2p$-multidifferential operator $P(\Theta)$ defined, at a point $x\in U$, by \eqref{P-Theta_x}, can be considered as an
abbreviated notation for $P(\underbrace{\Theta, \ldots, \Theta}_{p-times})$. In order to express $P(D_1,\ldots,D_p)$ we remark the following.
The value of $D_i$ on a $k_i$-tuple $(e_1, \ldots, e_{k_i})$ of elements of $\mathcal{E}$, $i=1,\ldots,p$, at a point $x\in M$,
as element of $\mathfrak{gl}(s,\R)$, can be written as
\[D_i(e_1, \ldots, e_{k_i})(x) = \sum_{m=1}^{s^2}D_i^m(e_1, \ldots, e_{k_i})(x) E_m,\]
where $(E_1,\ldots,E_{s^2})$ is a basis of $\mathfrak{gl}(s,\R)$ and $D_i^m \in \mathcal{D}^{k_i}(\mathcal{E})$. Setting
\[\varpi_{m_1\ldots m_p} = P(E_{m_1}, \ldots, E_{m_p}),\quad \quad m_i=1, \ldots, s^2,\]
which are real numbers symmetric with respect to the indices, we obtain
\[P(D_1,\ldots,D_p) = \sum_{m_1,\ldots, m_p}^{s^2}\varpi_{m_1\ldots m_p} D_1^{m_1}\diamond \ldots \diamond D_p^{m_p}.\]
So, after a simple computation, we take
\begin{eqnarray}\label{d_E P}
\partial_E (P(D_1,\ldots,D_p)) & = &  P(\partial_E D_1,\ldots,D_p) + (-1)^{k_1}P(D_1, \partial_ED_2, \ldots, D_p) \nonumber \\
& & +\, \ldots + (-1)^{(k-k_p)}P(D_1, \ldots, D_{p-1}, \partial_ED_p).
\end{eqnarray}

\vspace{2mm}
\noindent
Before we continue to establish the Chern-Weil homomorphism, we prove the following lemma.

\begin{lemma}
Let $D_t\in \mathcal{D}^p(\mathcal{E})$ be a $p$-differential operator depending smoothly on a real parameter $t\in [a,b] \subset \R$. Then
\begin{equation}\label{int-D_t}
\partial_E \int_a^b D_t\,dt = \int_a^b\partial_ED_t \, dt.
\end{equation}
\end{lemma}
\begin{proof}
In a local family $(\varepsilon_1^{j_1}\otimes \ldots \otimes \varepsilon_p^{j_p})_{(j_i = 1, \ldots, m_i, \, i=1,\ldots,p)}$ of generators
of $\mathcal{D}^p(\mathcal{E})$ (see Appendix \ref{app-multidiff}), $D_t$ is written as
\begin{equation*}
D_t \stackrel{\eqref{model - D}}{=} \sum_{j_1,\ldots,j_p}h_{j_1\ldots j_p}(x,t)\varepsilon_1^{j_1}\otimes \ldots \otimes \varepsilon_p^{j_p},
\quad \quad h_{j_1\ldots j_p} \in C^\infty(U\times [a,b],\R).
\end{equation*}
Hence, it suffices to prove the equality \eqref{int-D_t} for each term $h_{j_1\ldots j_p}(x,t)\varepsilon_1^{j_1}\otimes \ldots \otimes \varepsilon_p^{j_p}$. We have
\begin{eqnarray*}
\lefteqn{\int_a^b\partial_E(h_{j_1\ldots j_p}\varepsilon_1^{j_1}\otimes \ldots \otimes \varepsilon_p^{j_p})dt} \nonumber \\
& \stackrel{\eqref{d_E cobound}}{=} &
\int_a^b [\partial_E h_{j_1\ldots j_p}\diamond (\varepsilon_1^{j_1}\otimes \ldots \otimes \varepsilon_p^{j_p})+ h_{j_1\ldots j_p}\partial_E
(\varepsilon_1^{j_1}\otimes \ldots \otimes \varepsilon_p^{j_p}) ]dt
\nonumber \\
& = & [\int_a^b \partial_E(h_{j_1\ldots j_p})dt]\diamond (\varepsilon_1^{j_1}\otimes \ldots \otimes \varepsilon_p^{j_p}) + [\int_a^b h_{j_1\ldots j_p}dt]\partial_E
(\varepsilon_1^{j_1}\otimes \ldots \otimes \varepsilon_p^{j_p}) \nonumber \\
&\stackrel{\eqref{der_E f - transpose}}{=}&
[\int_a^b \frac{\partial h_{j_1\ldots j_p} }{\partial x^i}dt](\rho^\ast dx^i)\diamond (\varepsilon_1^{j_1}\otimes \ldots \otimes
\varepsilon_p^{j_p}) + [\int_a^b h_{j_1\ldots j_p}dt]\partial_E
(\varepsilon_1^{j_1}\otimes \ldots \otimes \varepsilon_p^{j_p}) \nonumber \\
& = & \frac{\partial}{\partial x^i}[\int_a^b  h_{j_1\ldots j_p}dt](\rho^\ast dx^i)\diamond (\varepsilon_1^{j_1}\otimes \ldots \otimes
\varepsilon_p^{j_p}) + [\int_a^b h_{j_1\ldots j_p}dt]\partial_E
(\varepsilon_1^{j_1}\otimes \ldots \otimes \varepsilon_p^{j_p}) \nonumber \\
& = & \partial_E [\int_a^b  h_{j_1\ldots j_p}dt]\diamond (\varepsilon_1^{j_1}\otimes \ldots \otimes
\varepsilon_p^{j_p}) + [\int_a^b h_{j_1\ldots j_p}dt]\partial_E
(\varepsilon_1^{j_1}\otimes \ldots \otimes \varepsilon_p^{j_p}) \nonumber \\
& = & \partial_E(\int_a^b h_{j_1\ldots j_p}\varepsilon_1^{j_1}\otimes \ldots \otimes \varepsilon_p^{j_p} dt),
\end{eqnarray*}
whence we conclude that \eqref{int-D_t} is true.
\end{proof}

\begin{proposition}\label{prop-Weil}
In the above framework, let
\begin{eqnarray*}
w : \mathrm{Inv}(\mathfrak{gl}(s,\R)) & \to & H^\bullet(\mathcal{E}) \nonumber \\
P & \mapsto & w(P): = [P(\Theta)]
\end{eqnarray*}
be the map from the algebra $\mathrm{Inv}(\mathfrak{gl}(s,\R))$ of $\mathrm{Ad}(\mathrm{GL}(s,\R))$-invariant polynomials on $\mathfrak{gl}(s,\R)$ to the Loday algebroid cohomology $H^\bullet(\mathcal{E})$, that sends a polynomial $P$ of degree $p$ to the element $[P(\Theta)]$ of the Loday algebroid cohomology group $H^{2p}(\mathcal{E})$.
Then, $w(P)$ is independent of the choice of a connection and $w : \mathrm{Inv}(\mathfrak{gl}(s,\R)) \to  H^\bullet(\mathcal{E})$ is an algebra homomorphism,
called the Chern-Weil homomorphism.
\end{proposition}
\begin{proof}
Let $\nabla^0$ and $\nabla^1$ be two different nonlinear $E$-connections on $B$ with connection matrices $\theta^0$ and $\theta^1$, respectively,
and curvature matrices $\Theta^0$ and $\Theta^1$, respectively, relative to a local frame of smooth sections of $B$ on an open subset $U$ of $M$.
We note by $\vartheta$ the difference $\theta^1-\theta^0$. Let, also, $P \in \mathrm{Inv}^p(\mathfrak{gl}(s,\R))$
be a symmetric $\mathrm{Ad}(\mathrm{GL}(s,\R))$-invariant $p$-multilinear function on $\mathfrak{gl}(s,\R)$.
Denote by $P(\Theta^0)$ and $P(\Theta^1)$ the $2p$-multidifferential operators defined through \eqref{P-Theta_x} on $U$.
We construct the $1$-parameter family of nonlinear $E$-connections $\nabla^t=(1-t)\nabla^0 + t\nabla^1$, $t\in [0,1]$, on $B$.
The connection matrix of $\nabla^t$ is $\theta^t = (1-t)\theta^0 + t\theta^1 = \theta^0 + t\vartheta$. We denote by $\Theta^t$ its curvature matrix.
By the transformation rule \eqref{theta - transf - rule}, we get that, for any $e\in \mathcal{E}$,
the difference $\vartheta(e) = \theta^1(e)-\theta^0(e)\in \mathrm{End}(\Gamma(B\vert_U))$.
So, for any $x\in U$, $\vartheta_x = \mathrm{ev}_x \circ \vartheta$ can be considered as a $\mathfrak{gl}(s,\R)$-valued function on $\mathcal{E}$.
Then, for any $P \in \mathrm{Inv}^p(\mathfrak{gl}(s,\R))$, we get a well defined $(2p-1)$-differential operator $\mathfrak{P}$ by setting,
for any $e_1,\ldots,e_{2p-1}\in \mathcal{E}$,
\begin{eqnarray*}
\lefteqn{\mathfrak{P}(e_1,\ldots,e_{2p-1}) = }\nonumber \\
& &  = \sum_{\tau \in Sh (1, \!\!\!\!\underbrace{\scriptstyle 2, \ldots, 2}_{(p-1)-times} \!\!\!\!)}(-1)^{|\tau|}p\int_0^1P(\vartheta(e_{\tau(1)}),\Theta^t(e_{\tau(2)},e_{\tau(3)}), \ldots,
\Theta^t(e_{\tau(2p-2)},e_{\tau(2p-1)}))dt,
\end{eqnarray*}
such that
\begin{equation}\label{class - P}
\partial_E \mathfrak{P} = P(\Theta^1) - P(\Theta^0).
\end{equation}
In order to prove \eqref{class - P}, we remark that
\begin{enumerate}
\item
\begin{equation}\label{d/dt Theta^t}
\frac{d}{dt}\Theta^t = \frac{d}{dt}(\partial_E \theta^t - \theta^t \diamond \theta^t) = \partial_E\vartheta - (\vartheta \diamond \theta^t + \theta^t \diamond \vartheta)
\end{equation}
and, for any $e_1,e_2\in \mathcal{E}$,
\begin{equation*}
(\vartheta \diamond \theta^t + \theta^t \diamond \vartheta)(e_1,e_2)= \vartheta(e_1)\theta^t(e_2)- \vartheta(e_2)\theta^t(e_1) + \theta^t(e_1) \vartheta(e_2) - \theta^t(e_2) \vartheta(e_1).
\end{equation*}
Because of the multilinearity and the symmetry of $P$, \eqref{invariance-sym} is true in each entry of $P$, thus
\begin{equation}\label{invariance - P - 1 - 1}
P(\ldots, (\vartheta \diamond \theta^t + \theta^t \diamond \vartheta)(e_1,e_2), \ldots) = 0.
\end{equation}
\item
\[\partial_E \Theta^t \stackrel{\eqref{bianchi id - theta}}{=}\theta^t \diamond \Theta^t - \Theta^t \diamond \theta^t\]
and, for the same reasons as above, we have that, for any $e_1,e_2, e_3\in \mathcal{E}$,
\begin{equation}\label{invariance - P - 1 - 2}
P(\ldots, \partial_E \Theta^t (e_1,e_2,e_3),\ldots) = P(\ldots, (\theta^t \diamond \Theta^t - \Theta^t \diamond \theta^t)(e_1,e_2,e_3), \ldots) = 0.
\end{equation}
\end{enumerate}
Thus,
\begin{eqnarray*}
\lefteqn{p \partial_E \int_0^1P(\vartheta,\Theta^t, \ldots, \Theta^t)dt \stackrel{\eqref{int-D_t}}{=} p \int_0^1 \partial_E (P(\vartheta,\Theta^t, \ldots, \Theta^t))dt}\nonumber \\
&\stackrel{\eqref{d_E P}}{=} & p \int_0^1 P(\partial_E \vartheta,\Theta^t, \ldots, \Theta^t)dt \nonumber \\
& & - \,p \int_0^1 P(\vartheta, \partial_E \Theta^t, \ldots, \Theta^t)dt -\, \ldots \,- p \int_0^1 P(\vartheta, \Theta^t, \ldots, \partial_E \Theta^t)dt \nonumber \\
& \stackrel{\eqref{d/dt Theta^t},\eqref{invariance - P - 1 - 1},\eqref{invariance - P - 1 - 2}}{=} & p \int_0^1 P(\frac{d}{dt}\Theta^t,\Theta^t, \ldots, \Theta^t)dt \nonumber \\
& = & \int_0^1 \frac{d}{dt}P(\Theta^t, \ldots, \Theta^t)dt = P(\Theta^1, \ldots, \Theta^1) - P(\Theta^0, \ldots, \Theta^0),
\end{eqnarray*}
whence we get that \eqref{class - P} holds.\footnote{We note that,
in the first equality of matrices of differential operators, of the proof of \eqref{class - P}, we apply \eqref{int-D_t} in each entry of the matrices.} Therefore, $[P(\Theta^0)]=[P(\Theta^1)]$ which means that $w(P)$ is independent of the choice of a connection. Finally, we can easily prove that for any pair $(P,Q)$ of $\mathrm{Ad}(\mathrm{GL}(s,\R))$-invariant polynomials on $\mathfrak{gl}(s,\R)$, $(P\cdot Q)(\Theta) = P(\Theta)\diamond Q(\Theta)$, so
\[w(P\cdot Q) = [(P\cdot Q)(\Theta)] = [P(\Theta) \diamond Q(\Theta)] = [P(\Theta)]\diamond [Q(\Theta)],\]
which means that $w : (\mathrm{Inv}(\mathfrak{gl}(s,\R)), \cdot ) \to  (H^\bullet(\mathcal{E}), \diamond)$ is an algebra homomorphism.
\end{proof}

\subsection{The Chern and the Chern-Simons classes of Loday algebroids}
For the purposes of the rest of the present paper, we introduce on the graded associative unital $\R$-algebra
$\big(\mathcal{D}(\mathcal{E}, \mathrm{End}(\Gamma(B))), \diamond \big)$ of a Loday algebroid $E$ on $M$ a new product $\odot$ that \emph{``symmetrizes"}
the existent $\diamond$ given by \eqref{sh-prod-values endomorphisms}.
Using the new product $\odot$, elements of $\mathrm{Inv}(\mathfrak{gl}(s,\R))$
may be evaluated on arguments that are $\odot$-commuting multidifferential operators.
Precisely, for any $\Delta_1 \in \mathcal{D}^{p_1}(\mathcal{E}, \mathrm{End}(\Gamma(B)))$ and
$\Delta_2 \in \mathcal{D}^{p_2}(\mathcal{E}, \mathrm{End}(\Gamma(B)))$, we define
\begin{equation*}\label{odot prod}
\Delta_1 \odot \Delta_2 = \Delta_1 \diamond \Delta_2 + (-1)^{p_1p_2}\Delta_2 \diamond \Delta_1.
\end{equation*}
So,
\begin{equation}\label{sym-odot-prod}
\Delta_1 \odot \Delta_2 = (-1)^{p_1p_2}\Delta_2 \odot \Delta_1
\end{equation}
and the pair $\big(\mathcal{D}(\mathcal{E}, \mathrm{End}(\Gamma(B))), \odot \big)$ is a graded Koszul algebra.
Furthermore, if $\nabla$ is a nonlinear $E$-connection on a vector bundle $B$ over $M$,
we have that the covariant derivative $d^{\tilde{\nabla}}$ verifies the Leibniz rule for the new product also.
From \eqref{eq der diamond}, we obtain
\begin{equation}\label{eq der dot}
d^{\tilde{\nabla}}(\Delta_1 \odot \Delta_2) = (d^{\tilde{\nabla}}\Delta_1) \odot \Delta_2 + (-1)^{p_1} \Delta_1 \odot (d^{\tilde{\nabla}}\Delta_2).
\end{equation}
For any $p\in \N^\ast$, we define the \emph{Chern characters} of $\nabla$ by setting
\begin{equation*}\label{chern-characters}
\mathrm{ch}_p (\nabla) = \tr(R^{\nabla})^p,
\end{equation*}
where $(R^{\nabla})^p = \underbrace{R^{\nabla}\odot \ldots \odot R^{\nabla}}_{p - times} = p!\underbrace{R^{\nabla}\diamond \ldots \diamond
R^{\nabla}}_{p - times} \in \mathcal{D}^{2p}(\mathcal{E}, \mathrm{End}(\Gamma(B)))$.
By construction, they are $\partial_E$-closed elements of $\mathcal{D}^{2p}(\mathcal{E})$ of order $2p$ and of total order $(2p)^2$.
As a consequence of the Bianchi identity \eqref{bianchi id} and the relations \eqref{dif - Tr} and \eqref{eq der dot},
\begin{eqnarray*}
\partial_E \mathrm{ch}_p (\nabla) = 0.
\end{eqnarray*}
So, the class $[\mathrm{ch}_p (\nabla)]\in H^{2p}(\mathcal{E})$ is defined and it is called the \emph{$p$-Chern class of $(E, \nabla)$}.
Clearly, if $\nabla$ is flat, $[\mathrm{ch}_p(\nabla)]=0$. Hence, the nonvanishing of $[\mathrm{ch}_p(\nabla)]$ can be considered
as the obstruction to the existence of a representation of $(E,\nabla)$.

\vspace{2mm}
\noindent
In the theory of linear connections of an algebroid (Lie \cite{rui, cr-fer-sec-cl}, Courant \cite{CM2021}) on a vector bundle, a basic fact is that the Chern classes do not depend on the connection. This is established by using a method similar to the classical Chern-Simons construction \cite{chern-sim} and by proving that the Chern-Simons forms are transgressions of the Chern forms. In the following, in order to prove, in the context of Loday algebroid theory, the independence of the classes $[\mathrm{ch}_p(\nabla)]$ from $\nabla$, we propose a Chern-Simons-type construction and prove that the Chern-Simons multidifferential operators are transgressions of the Chern operators. A generalized Chern-Simons formula is also derived.
\vspace{1mm}
\noindent
Let $\nabla^0,\nabla^1,\ldots,\nabla^k$ be $k+1$ nonlinear $E$-connections on $B$ and
\begin{equation*}
\Delta^k = \{(t_1, \ldots, t_k) \in \R^k \, | \, t_1,\ldots, t_k \geq 0 \quad \mathrm{and} \quad t_1 + \ldots + t_k \leq 1 \}
\end{equation*}
the standard $k$-simplex in $\R^k$. For any $i\in \{1,\ldots,k\}$, we denote by $\Delta^k_i$ the $(k-1)$-simplex
which is the projection of $\Delta^k$ on the hyperplane $t_i=0$, i.e. a "face" of $\Delta^k$:
\begin{equation*}
\Delta^k_i = \{(t_1, \ldots, \hat{t}_i,\ldots, t_k) \in \R^{k-1} \, | \, t_1,\ldots,\hat{t}_i,\ldots, t_k \geq 0
\quad \mathrm{and} \quad t_1 + \ldots + \hat{t}_i + \ldots + t_k \leq 1 \}.
\end{equation*}
For each $(t_1, \ldots, t_k)\in \Delta^k$, we form the affine combination of $(\nabla^0,\nabla^1,\ldots,\nabla^k)$:
\begin{equation*}
\nabla^{0,\ldots,k}:= (1-\sum_{i=1}^kt_i)\nabla^0 + t_1\nabla^1 + \ldots + t_k \nabla^k,
\end{equation*}
which is a nonlinear $E$-connection on $B$, and we denote by $\nabla^{0,\ldots,\hat{i},\ldots,k}$ the affine combination
\begin{equation*}
\nabla^{0,\ldots,\hat{i},\ldots,k} : = (1-\sum_{{j=1}, j\neq i}^kt_j)\nabla^0 + t_1\nabla^1 + \ldots + \widehat{t_i\nabla^i} + \ldots + t_k \nabla^k
\end{equation*}
parametrized by $\Delta^k_i$.
\begin{definition}
We define the $p$-Chern-Simons transgression operators $\mathrm{cs}_p$ of $(\nabla^0,\ldots,\nabla^k)$, for all integers $p$ such that $0<k\leq 2p$, by setting
\begin{equation}\label{chern-simons dif op}
\mathrm{cs}_p(\nabla^0,\ldots,\nabla^k): = \int_{\Delta^k} \tr\Big(\frac{\partial \nabla^{0,\ldots,k}}{\partial t_1}\odot \ldots \odot
\frac{\partial \nabla^{0,\ldots,k}}{\partial t_k}\odot (R^{\nabla^{0,\ldots,k}})^{p-k}\Big)dt_1\ldots dt_k,
\end{equation}
where $R^{\nabla^{0,\ldots,k}}$ indicates the curvature of $\nabla^{0,\ldots,k}$.
\end{definition}
Clearly, $\mathrm{cs}_p(\nabla^0,\ldots,\nabla^k)$ is an element of $\mathcal{D}^{2p-k}(\mathcal{E})$.
By convention, for a point $\nabla \in \mathfrak{C}_{nl}(E,B)$, we set $\mathrm{cs}_p(\nabla) = \mathrm{ch}_p(\nabla)$.

\begin{theorem}\label{theorem-chern-simons}
Under the above assumptions, the Chern-Simons transgression operators satisfy, for all $k, p\in \N$ such that $0<k\leq 2p$, the following identity:
\begin{equation}\label{identity-chern-simons}
(p-k+1) \partial_E \mathrm{cs}_p(\nabla^0,\ldots,\nabla^k) = \sum_{i=0}^k(-1)^i  \mathrm{cs}_p(\nabla^0,\ldots, \widehat{\nabla^i},\ldots, \nabla^k).
\end{equation}
\end{theorem}

\vspace{2mm}
\noindent
Before proceeding to the proof of Theorem \ref{theorem-chern-simons}, we prove two auxiliary lemmas.

\begin{lemma}\label{lemma t=0 chern-simons}
For each $i\in \{1, \ldots,k\}$, we have
\begin{eqnarray}\label{lemma t=0 chern-simons-eq}
\lefteqn{\mathrm{cs}_p(\nabla^0,\ldots,\widehat{\nabla^i},\ldots,\nabla^k)=}  \\
& & \hspace{-5mm} \int_{\Delta^k_i} \tr\Big(\frac{\partial \nabla^{0,\ldots,k}}{\partial t_1}\odot \ldots \odot
\widehat{\frac{\partial \nabla^{0,\ldots,k}}{\partial t_i}}\odot \ldots \odot
\frac{\partial \nabla^{0,\ldots,k}}{\partial t_k}\odot (R^{\nabla^{0,\ldots,k}})^{p-(k-1)}\Big)\Big|_{t_i = 0}dt_1\ldots \widehat{dt_i}\ldots dt_k. \nonumber
\end{eqnarray}
\end{lemma}
\begin{proof}
We have $\nabla^{0,\ldots,k}\Big|_{t_i = 0} = \nabla^{0,\ldots,\hat{i},\ldots,k}$ and $R^{\nabla^{0,\ldots,k}}\Big|_{t_i=0} = R^{\nabla^{0,\ldots,\hat{i},\ldots,k}}$.
So, formula \eqref{lemma t=0 chern-simons-eq} is true.
\end{proof}

\begin{lemma}\label{lemma chern-simons 1 to k}
We have
\begin{eqnarray}\label{chern-simons eq 1 to k}
\lefteqn{\mathrm{cs}_p(\nabla^1,\ldots,\nabla^k)=}  \nonumber \\
& & \sum_{i=1}^k(-1)^{i-1}\int_{\Delta^{k-1}_i}
\tr\Big(\frac{\partial \nabla^{0,\ldots,k}}{\partial t_1}\odot \ldots \odot
\widehat{\frac{\partial \nabla^{0,\ldots,k}}{\partial t_i}}\odot \ldots \odot
\frac{\partial \nabla^{0,\ldots,k}}{\partial t_k}\odot \nonumber \\
& & \hspace{3cm} (R^{\nabla^{0,\ldots,k}})^{p-(k-1)}\Big)\Big|_{t_i = 1-\sum_{j\neq i}t_j}dt_1\ldots \widehat{dt_i}\ldots dt_k.
\end{eqnarray}
\end{lemma}
\begin{proof}
We compute the sum of the right-hand side of \eqref{chern-simons eq 1 to k}
by first making some remarks on the integrand of the first term (for $i=1$) of the sum.
Since $\displaystyle{\frac{\partial \nabla^{0,\ldots,k}}{\partial t_i}}$ is a $1$-differential operator,
$\displaystyle{\frac{\partial \nabla^{0,\ldots,k}}{\partial t_i}}\odot \displaystyle{\frac{\partial \nabla^{0,\ldots,k}}{\partial t_i}}\stackrel{\eqref{sym-odot-prod}}{=}0$,
for all $i\in \{1,\ldots,k\}$. So,
\begin{eqnarray}\label{formula odot parenthese}
\lefteqn{\frac{\partial \nabla^{0,\ldots,k}}{\partial t_2}\odot \ldots \odot\frac{\partial \nabla^{0,\ldots,k}}{\partial t_k} = }\nonumber \\
& & \big(\frac{\partial \nabla^{0,\ldots,k}}{\partial t_2} - \frac{\partial \nabla^{0,\ldots,k}}{\partial t_1}\big)\odot
\big(\frac{\partial \nabla^{0,\ldots,k}}{\partial t_3} - \frac{\partial \nabla^{0,\ldots,k}}{\partial t_1}\big)\odot \ldots \odot
\big(\frac{\partial \nabla^{0,\ldots,k}}{\partial t_k} - \frac{\partial \nabla^{0,\ldots,k}}{\partial t_1}\big)  \nonumber \\
& & + \, \sum_{i=2}^k (-1)^i \frac{\partial \nabla^{0,\ldots,k}}{\partial t_1}\odot \ldots
\odot \widehat{\frac{\partial \nabla^{0,\ldots,k}}{\partial t_i}}\odot \ldots \odot\frac{\partial \nabla^{0,\ldots,k}}{\partial t_k}.
\end{eqnarray}
Also, we remark that
\begin{equation}\label{nabla - curvature 1 to k}
\nabla^{0,\ldots,k}\Big|_{t_1 = 1 -\sum_{j\neq 1} t_j}= \nabla^{1,\ldots,k}, \quad \quad R^{\nabla^{0,\ldots,k}}\Big|_{t_1 = 1 -\sum_{j\neq 1} t_j}=
R^{\nabla^{1,\ldots,k}},
\end{equation}
and
\begin{equation}\label{nabla 1 to k}
\frac{\partial \nabla^{0,\ldots,k}}{\partial t_i} - \frac{\partial \nabla^{0,\ldots,k}}{\partial t_1}=
\nabla^i - \nabla^1 = \frac{\partial\nabla^{1,\ldots,k}}{\partial t_i}.
\end{equation}
Consequently,
\begin{eqnarray*}
\lefteqn{\sum_{i=1}^k(-1)^{i-1}\int_{\Delta^{k-1}_i}
\tr\Big(\frac{\partial \nabla^{0,\ldots,k}}{\partial t_1}\odot \ldots \odot
\widehat{\frac{\partial \nabla^{0,\ldots,k}}{\partial t_i}}\odot \ldots \odot
\frac{\partial \nabla^{0,\ldots,k}}{\partial t_k}\odot } \nonumber \\
& & \hspace{3cm} (R^{\nabla^{0,\ldots,k}})^{p-(k-1)}\Big)\Big|_{t_i = 1-\sum_{j\neq i}t_j}dt_1\ldots \widehat{dt_i}\ldots dt_k   \nonumber \\
& = & \int_{\Delta^{k-1}_1}
\tr\Big(\frac{\partial \nabla^{0,\ldots,k}}{\partial t_2}\odot \ldots \odot \frac{\partial \nabla^{0,\ldots,k}}{\partial t_k}\odot
(R^{\nabla^{0,\ldots,k}})^{p-(k-1)}\Big)\Big|_{t_1 = 1-\sum_{j\neq 1}t_j}dt_2\ldots dt_k \nonumber \\
& & +\, \sum_{i=2}^k(-1)^{i-1}\int_{\Delta^{k-1}_i}
\tr\Big(\frac{\partial \nabla^{0,\ldots,k}}{\partial t_1}\odot \ldots \odot
\widehat{\frac{\partial \nabla^{0,\ldots,k}}{\partial t_i}}\odot \ldots \odot
\frac{\partial \nabla^{0,\ldots,k}}{\partial t_k}\odot\nonumber \\
& & \hspace{3cm} (R^{\nabla^{0,\ldots,k}})^{p-(k-1)}\Big)\Big|_{t_i = 1-\sum_{j\neq i}t_j}dt_1\ldots \widehat{dt_i}\ldots dt_k   \nonumber \\
& \stackrel{\eqref{formula odot parenthese}}{=} &
\int_{\Delta^{k-1}_1}
\tr\Big(\big(\frac{\partial \nabla^{0,\ldots,k}}{\partial t_2} - \frac{\partial \nabla^{0,\ldots,k}}{\partial t_1}\big)\odot \ldots \odot
\big(\frac{\partial \nabla^{0,\ldots,k}}{\partial t_k} - \frac{\partial \nabla^{0,\ldots,k}}{\partial t_1}\big) \odot \nonumber \\
& & \hspace{3cm} (R^{\nabla^{0,\ldots,k}})^{p-(k-1)}\Big)\Big|_{t_1 = 1-\sum_{j\neq 1}t_j}dt_2\ldots dt_k \nonumber \\
& & + \, \sum_{i=2}^k (-1)^i \int_{\Delta^{k-1}_1}\tr\Big(\frac{\partial \nabla^{0,\ldots,k}}{\partial t_1}\odot \ldots
\odot \widehat{\frac{\partial \nabla^{0,\ldots,k}}{\partial t_i}}\odot \ldots \odot\frac{\partial \nabla^{0,\ldots,k}}{\partial t_k}\odot \nonumber \\
& & \hspace{3cm} (R^{\nabla^{0,\ldots,k}})^{p-(k-1)}\Big)\Big|_{t_1 = 1-\sum_{j\neq 1}t_j}dt_2\ldots \widehat{dt_i}\ldots dt_k   \nonumber \\
& & + \, \sum_{i=2}^k(-1)^{i-1}\int_{\Delta^{k-1}_i}
\tr\Big(\frac{\partial \nabla^{0,\ldots,k}}{\partial t_1}\odot \ldots \odot
\widehat{\frac{\partial \nabla^{0,\ldots,k}}{\partial t_i}}\odot \ldots \odot
\frac{\partial \nabla^{0,\ldots,k}}{\partial t_k}\odot\nonumber \\
& & \hspace{3cm} (R^{\nabla^{0,\ldots,k}})^{p-(k-1)}\Big)\Big|_{t_i = 1-\sum_{j\neq i}t_j}dt_1\ldots \widehat{dt_i}\ldots dt_k   \nonumber \\
& \stackrel{\eqref{nabla - curvature 1 to k}, \eqref{nabla 1 to k}}{=} &
\int_{\Delta^{k-1}_1} \tr\Big(\frac{\partial\nabla^{1,\ldots,k}}{\partial t_2}\odot \ldots \odot
\frac{\partial\nabla^{1,\ldots,k}}{\partial t_k}\odot (R^{\nabla^{1,\ldots,k}})^{p-(k-1)}\Big)dt_2\ldots dt_k \nonumber \\
& & + \, \sum_{i=2}^k (-1)^i \int_{\Delta^{k-1}_1}\tr\Big(\frac{\partial \nabla^{0,\ldots,k}}{\partial t_1}\odot \ldots
\odot \widehat{\frac{\partial \nabla^{0,\ldots,k}}{\partial t_i}}\odot \ldots \odot\frac{\partial \nabla^{0,\ldots,k}}{\partial t_k}\odot \nonumber \\
& & \hspace{3cm} (R^{\nabla^{0,\ldots,k}})^{p-(k-1)}\Big)\Big|_{t_1 = 1-\sum_{j\neq 1}t_j}dt_2\ldots \widehat{dt_i}\ldots dt_k   \nonumber \\
& & + \, \sum_{i=2}^k(-1)^{i-1}\int_{\Delta^{k-1}_i}
\tr\Big(\frac{\partial \nabla^{0,\ldots,k}}{\partial t_1}\odot \ldots \odot
\widehat{\frac{\partial \nabla^{0,\ldots,k}}{\partial t_i}}\odot \ldots \odot
\frac{\partial \nabla^{0,\ldots,k}}{\partial t_k}\odot\nonumber \\
& & \hspace{3cm} (R^{\nabla^{0,\ldots,k}})^{p-(k-1)}\Big)\Big|_{t_i = 1-\sum_{j\neq i}t_j}dt_1\ldots \widehat{dt_i}\ldots dt_k   \nonumber \\
&\stackrel{\eqref{chern-simons dif op}}{=}& \mathrm{cs}_p(\nabla^1,\ldots,\nabla^k).
\end{eqnarray*}
The last equality is obtained by making, for $i=2, \ldots,k$, the change of coordinates
$\Delta^{k-1}_1 \ni (t_2,\ldots, t_k) \mapsto (1- t_2 - \ldots - t_k, t_2, \ldots, \hat{t}_i, \ldots, t_k)\in \Delta^{k-1}_i$.
Then, the terms of the second sum over $i$ become the opposites of those of the first sum over $i$.
\end{proof}

\vspace{2mm}
\noindent

\begin{proof-Th} \ref{theorem-chern-simons}. By the Fundamental Theorem of Integral Calculus and the fact that
$\displaystyle{\frac{\partial^2 \nabla^{0,\ldots,k}}{\partial t_i \partial t_j}} = 0$, we obtain that
\begin{eqnarray}
\lefteqn{\partial_E \mathrm{cs}_p(\nabla^0,\ldots,\nabla^k) \stackrel{\eqref{chern-simons dif op}}{=}
\partial_E \int_{\Delta^k} \tr\Big(\frac{\partial \nabla^{0,\ldots,k}}{\partial t_1}\odot \ldots \odot
\frac{\partial \nabla^{0,\ldots,k}}{\partial t_k}\odot
 (R^{\nabla^{0,\ldots,k}})^{p-k}\Big)dt_1\ldots dt_k } \nonumber \\
& \stackrel{\eqref{int-D_t}}{=} &  \int_{\Delta^k} \partial_E \tr\Big(\frac{\partial \nabla^{0,\ldots,k}}{\partial t_1}\odot \ldots \odot
\frac{\partial \nabla^{0,\ldots,k}}{\partial t_k}\odot (R^{\nabla^{0,\ldots,k}})^{p-k}\Big)dt_1\ldots dt_k \nonumber \\
& \stackrel{\eqref{dif - Tr}}{=}& \int_{\Delta^k} \tr d^{\tilde{\nabla}^{0,\ldots,k}}\Big(\frac{\partial \nabla^{0,\ldots,k}}{\partial t_1}\odot \ldots \odot
\frac{\partial \nabla^{0,\ldots,k}}{\partial t_k}\odot (R^{\nabla^{0,\ldots,k}})^{p-k}\Big)dt_1\ldots dt_k \nonumber \\
& \stackrel{\eqref{eq der dot}}{=}&\sum_{i=1}^k(-1)^{i-1} \int_{\Delta^k} \tr\Big(\frac{\partial \nabla^{0,\ldots,k}}{\partial t_1}\odot
\ldots \odot d^{\tilde{\nabla}^{0,\ldots,k}} \frac{\partial \nabla^{0,\ldots,k}}{\partial t_i}\odot \ldots \odot
\frac{\partial \nabla^{0,\ldots,k}}{\partial t_k}\odot \nonumber \\
& & \hspace{6cm} (R^{\nabla^{0,\ldots,k}})^{p-k}\Big)dt_1\ldots dt_k \nonumber \\
& \stackrel{\eqref{der R^t},\eqref{bianchi id}}{=}&\sum_{i=1}^k(-1)^{i-1} \int_{\Delta^k} \tr\Big(\frac{\partial \nabla^{0,\ldots,k}}{\partial t_1}\odot
\ldots \odot \widehat{\frac{\partial \nabla^{0,\ldots,k}}{\partial t_i}}\odot \ldots \odot
\frac{\partial \nabla^{0,\ldots,k}}{\partial t_k}\odot \nonumber \\
& & \hspace{4cm}\frac{\partial R^{\nabla^{0,\ldots,k}}}{\partial t_i} \odot (R^{\nabla^{0,\ldots,k}})^{p-k}\Big)dt_1\ldots dt_k \nonumber \\
& = & \frac{1}{p-k+1} \sum_{i=1}^k(-1)^{i-1} \int_{\Delta^k} \tr\Big(\frac{\partial \nabla^{0,\ldots,k}}{\partial t_1}\odot
\ldots \odot \widehat{\frac{\partial \nabla^{0,\ldots,k}}{\partial t_i}}\odot \ldots \odot
\frac{\partial \nabla^{0,\ldots,k}}{\partial t_k}\odot \nonumber \\
& & \hspace{6cm}\frac{\partial (R^{\nabla^{0,\ldots,k}})^{p-k+1}}{\partial t_i} \Big)dt_1\ldots dt_k \nonumber \\
& = & \frac{1}{p-k+1} \sum_{i=1}^k(-1)^{i-1} \int_{\Delta^k} \tr\frac{\partial}{\partial t_i}\Big(\frac{\partial \nabla^{0,\ldots,k}}{\partial t_1}\odot
\ldots \odot \widehat{\frac{\partial \nabla^{0,\ldots,k}}{\partial t_i}}\odot \ldots \odot
\frac{\partial \nabla^{0,\ldots,k}}{\partial t_k}\odot \nonumber \\
& & \hspace{6cm} (R^{\nabla^{0,\ldots,k}})^{p-k+1} \Big)dt_1\ldots dt_k \nonumber \\
& = & \frac{1}{p-k+1} \sum_{i=1}^k(-1)^{i-1} \int_{\Delta^{k-1}_i} \tr\Big(\frac{\partial \nabla^{0,\ldots,k}}{\partial t_1}\odot
\ldots \odot \widehat{\frac{\partial \nabla^{0,\ldots,k}}{\partial t_i}}\odot \ldots \odot
\frac{\partial \nabla^{0,\ldots,k}}{\partial t_k}\odot \nonumber \\
& & \hspace{6cm} (R^{\nabla^{0,\ldots,k}})^{p-(k-1)} \Big)\Big|_{t_i = 1 - \sum_{j\neq i}t_j}dt_1\ldots \widehat{dt_i} \ldots dt_k \nonumber \\
& - & \frac{1}{p-k+1} \sum_{i=1}^k(-1)^{i-1} \int_{\Delta^{k-1}_i} \tr\Big(\frac{\partial \nabla^{0,\ldots,k}}{\partial t_1}\odot
\ldots \odot \widehat{\frac{\partial \nabla^{0,\ldots,k}}{\partial t_i}}\odot \ldots \odot
\frac{\partial \nabla^{0,\ldots,k}}{\partial t_k}\odot \nonumber \\
& & \hspace{6cm} (R^{\nabla^{0,\ldots,k}})^{p-(k-1)} \Big)\Big|_{t_i =0}dt_1\ldots \widehat{dt_i} \ldots dt_k \nonumber \\
&\stackrel{\eqref{chern-simons eq 1 to k}, \eqref{lemma t=0 chern-simons-eq}
}{=} & \frac{1}{p-k+1} \big[ \mathrm{cs}_p(\nabla^1,\ldots,\nabla^k) +  \sum_{i=1}^k (-1)^i \mathrm{cs}_p(\nabla^0,\ldots,\widehat{\nabla^i},\ldots,\nabla^k)\big],
\end{eqnarray}
whence we get the identity \eqref{identity-chern-simons}.
\end{proof-Th}

\begin{corol}
In the above context, the Chern cohomology classes $[\mathrm{ch}_p(\nabla)] \in H^{2p}(\mathcal{E})$, $p\in \N^\ast$, are independent of $\nabla$.
\end{corol}
\begin{proof}
It is an immediate consequence of the Theorem \ref{theorem-chern-simons}. Because of the convention $\mathrm{ch}_p(\nabla) = \mathrm{cs}_p(\nabla)$, for $k=1$, we take
 \begin{equation*}
\mathrm{ch}_p(\nabla^1) - \mathrm{ch}_p(\nabla^0) \stackrel{\eqref{identity-chern-simons}}{=} p\, \partial_E \mathrm{cs}_p(\nabla^0, \nabla^1)
\end{equation*}
and $[\mathrm{ch}_p(\nabla^0)] = [\mathrm{ch}_p(\nabla^1)]$.
\end{proof}

\begin{remark}\label{remark - cs - indepen of path}
\emph{In the case where $k=1$, $\nabla^{0,1} = (1-t)\nabla^0 + t\nabla^1$ is the line segment in the space $\mathfrak{C}_{nl}(E,B)$ joining $\nabla^0$ and $\nabla^1$,
and $\Delta^1=[0,1]$. If $\nabla^t$ is any path in $\mathfrak{C}_{nl}(E,B)$ from $\nabla^0$ to $\nabla^1$,
by reconsidering the relations \eqref{der R^t} and \eqref{dif - Tr}, we have that}
\begin{equation*}
\mathrm{ch}_p(\nabla^1) - \mathrm{ch}_p(\nabla^0) = \int_0^1\frac{d}{dt} \mathrm{ch}_p(\nabla^t)dt =
\int_0^1 \frac{d}{dt}\tr (R^{\nabla^t})^pdt =
 p\, \partial_E \int_0^1 \tr \Big(\frac{d \nabla^t}{dt}\odot (R^{\nabla^t})^{p-1}\Big)dt,
\end{equation*}
\emph{i.e.}
\begin{equation*}\label{cs(0,1)-cs(t)}
\partial_E \mathrm{cs}_p(\nabla^0, \nabla^1) = \partial_E \int_0^1 \tr\Big(\frac{d \nabla^t}{dt}\odot (R^{\nabla^t})^{p-1}\Big)dt.
\end{equation*}
\emph{The last equation means that a representative of $[\mathrm{cs}_p(\nabla^0, \nabla^1)]$ is the transgression ope\-ra\-tor $\displaystyle{\int_0^1 \tr\Big(\frac{d \nabla^t}{dt}\odot (R^{\nabla^t})^{p-1}\Big)}$. Motivated by the above, we define the} $p$-Chern-Simons transgression operators of second type $\mathrm{CS}_p(\nabla^t)$ of $\nabla^t$ \emph{by setting}
\begin{equation}\label{Chern-Simons dif op - t}
\mathrm{CS}_p(\nabla^t) = \int_0^1 \tr\Big(\frac{d \nabla^t}{dt}\odot (R^{\nabla^t})^{p-1}\Big)dt.
\end{equation}
\emph{This type of operators was already defined by Cueca and Mehta in \cite{CM2021}. Then, $[\mathrm{cs}_p(\nabla^0, \nabla^1)]= [\mathrm{CS}_p(\nabla^t)]$, i.e., the $p$-Chern-Simons
classes do not depend on the path connecting $\nabla^0$ and $\nabla^1$.}
\end{remark}

\vspace{3mm}
\noindent
The behavior of Chern classes $\mathrm{ch}_p(\nabla)$ under an automorphism $\Phi$ of $B$ covering the identity map on $M$ is studied below.
The automorphism $\Phi$ transforms
$\nabla$ into a new nonlinear $E$-connection $\nabla^\Phi$ on $B$ by setting, for any $e\in \Gamma(E)$,
\[\nabla_e^{\Phi} = \Phi \circ \nabla_e \circ \Phi^{-1}.\]
Then, for any $e_1,e_2 \in \Gamma(E)$, $R^{\nabla^\Phi}(e_1,e_2) = \Phi \circ R^\nabla (e_1,e_2)\circ \Phi^{-1}$ and
\[\underbrace{R^{\nabla^\Phi} \odot \ldots \odot R^{\nabla^\Phi}}_{p-times} =
\Phi \circ \underbrace{R^{\nabla} \odot \ldots \odot R^{\nabla}}_{p-times}\circ \, \Phi^{-1}.\]
So, because of the $\mathrm{Ad}(\mathrm{GL}(s,\R))$-invariance of $\tr$,
$[\mathrm{ch}_p(\nabla^\Phi)] = [\mathrm{ch}_p(\nabla)]$.

\begin{proposition}\label{Phi_t-nabla}
Let $\nabla$ be a nonlinear $E$-connection on $B$ and $\Phi_t$ a path of automorphisms of $B$ over the identity map on $M$,
through $\mathrm{Id}_B$ for $t=0$. Then, the $p$-Chern-Simons transgressions operators of second type $\mathrm{CS}_p(\nabla^{\Phi_t})$ are $\partial_E - exact$.
\end{proposition}
\begin{proof}
After some similar computations as above, we find
\[\mathrm{CS}_p(\nabla^{\Phi_t}) = -  \partial_E (\int_0^1 \tr\Big(\Phi_t^{-1}\circ \frac{d \Phi_t}{dt}\circ (R^{\nabla})^{p-1}\Big)dt),\]
which is exact.
\end{proof}

\subsection{Secondary characteristic classes of a Loday algebroid}
As mentioned in the literature, secondary characteristic classes of
Lie algebroids \cite{rui, cr-fer-sec-cl}, Courant algebroids \cite{CM2021},
and other structures appear as a consequence of the vanishing of the Chern characters of a representation.
The fact that the Chern characters also vanish in the context of Loday algebroids allows us to introduce secondary characteristic
classes for nonlinear $E$-connections on vector bundles using, in general, methods similar to those of the classical theory.

\vspace{1mm}
\noindent
Let $B$ equipped with a fiberwise nondegenerate symmetric bilinear form $g$ and $\nabla$ a nonlinear $E$-connection on $B$.
We define the conjugated nonlinear $E$-connection $\nabla^g$ on $B$ by setting, for any $e\in \Gamma(E)$ and $b_1,b_2 \in \Gamma(B)$,
\begin{equation}\label{nabla - g}
\rho(e)g(b_1,b_2) = g(\nabla_e^g b_1, b_2) + g (b_1, \nabla_eb_2).
\end{equation}

\begin{proposition}\label{proposition path nabla}
For any path $\nabla_s$ of nonlinear $E$-connections on $B$ and the corresponding path $\nabla^g_s$ of their conjugates on $B$, we have
\begin{equation*}
\mathrm{CS}_p(\nabla_s^g) = (-1)^p\mathrm{CS}_p(\nabla_s).
\end{equation*}
\end{proposition}
\begin{proof}
From \eqref{nabla - g}, we take
that $g^\flat \circ \displaystyle{\frac{d \nabla^g_s}{ds}} + \displaystyle{\frac{d \nabla_s}{ds}}^T\circ g^\flat = 0$,
where $\displaystyle{\frac{d \nabla_s}{ds}}^T$ is the transpose of $\displaystyle{\frac{d \nabla_s}{ds}}$ and $g^\flat : \Gamma(B) \to \Gamma(B^\ast)$
is the isomorphism defined by $g$. Also, from \eqref{dual curvature}, we obtain $g^\flat \circ R^{\nabla^g} + {R^\nabla}^T \circ g^\flat = 0$,
where ${R^\nabla}^T$ denotes the transpose of $R^\nabla$ (as endomorphism of $B$). Hence,
by replacing the above expressions in \eqref{Chern-Simons dif op - t}, we obtain the announced relation.
\end{proof}

\begin{corol}
If $\nabla^0$, $\nabla^1$ are nonlinear $E$-connections on $B$, and $\nabla^{0,g}$, $\nabla^{1,g}$ their conjugated nonlinear connections with respect of $g$, then
\begin{equation}\label{cs_p(nabla 0, nabla 1)}
\mathrm{cs}_p(\nabla^{0,g}, \nabla^{1,g}) = (-1)^p\mathrm{cs}_p(\nabla^0, \nabla^1).
\end{equation}
\end{corol}

\begin{proposition}
Let $\nabla$ be as above and $g$, $g'$ two fiberwise nondegenerate inner products on $B$. Then the $p$-Chern-Simons
transgression operator $\mathrm{cs}_p(\nabla^g, \nabla^{g'})$ is exact.
\end{proposition}
\begin{proof}
Let $g_t$ be a path of fiberwise nondegenerate inner products on $B$ connecting $g=g_0$ with $g'=g_1$.
Consider the family $\Phi_t = g^{\flat^{-1}}\circ g_t^\flat$ of automorphisms of $B$ over the identity of $M$,
which passes through $\mathrm{Id}_B$ for $t=0$. If $\nabla^{g_t}$ is the conjugated connection of $\nabla$
with respect to $g_t$, we have $\nabla_e^{g_t} = \Phi_t^{-1} \circ \nabla_e^g \circ \Phi_t$.
According to Proposition \ref{Phi_t-nabla}, the $p$-Chern-Simons transgressions operator of the second type
$\mathrm{CS}_p(\Phi_t^{-1} \circ \nabla_e^g \circ \Phi_t)$ is exact. Therefore, $\mathrm{CS}_p(\nabla_e^{g_t})$ is exact and
\begin{equation}\label{exacteness nabla^g nabla^g'}
[\mathrm{cs}_p(\nabla^g, \nabla^{g'})] = [\mathrm{CS}_p(\nabla^{g_t})] =0.
\end{equation}
\end{proof}

\begin{corol}
If $\nabla$ is a representation, then $\mathrm{cs}_p(\nabla, \nabla^g)$ is closed, and the
cohomology class $[\mathrm{cs}_p(\nabla, \nabla^g)]$ does not depend on the inner product $g$.
\end{corol}
\begin{proof}
Since $\nabla$ is a representation, $\mathrm{ch}_p(\nabla)=0$. Then,
\begin{equation*}
p \,\partial_E \mathrm{cs}_p(\nabla, \nabla^g) \stackrel{\eqref{identity-chern-simons}}{=} \mathrm{cs}_p(\nabla^g) - \mathrm{cs}_p(\nabla)
=  (-1)^p\mathrm{ch}_p(\nabla) - \mathrm{ch}_p(\nabla) =0,
\end{equation*}
that is, $\mathrm{cs}_p(\nabla, \nabla^g)$ is closed. Also, for $g$ and $g'$ as above, we have
\begin{equation*}
(p+1)\partial_E \mathrm{cs}_p(\nabla, \nabla^g, \nabla^{g'})
\stackrel{\eqref{identity-chern-simons}}{=} \mathrm{cs}_p(\nabla^g, \nabla^{g'}) -
\mathrm{cs}_p(\nabla, \nabla^{g'}) + \mathrm{cs}_p(\nabla, \nabla^g).
\end{equation*}
Because of the exactness of $\mathrm{cs}_p(\nabla^g, \nabla^{g'})$ (see \eqref{exacteness nabla^g nabla^g'}), we get
$[\mathrm{cs}_p(\nabla, \nabla^g)]=[\mathrm{cs}_p(\nabla, \nabla^{g'})]$.
\end{proof}

\begin{definition}
The \emph{secondary characteristic classes} of a representation $B$ with respect to a nonlinear $E$-connection $\nabla$ are the cohomology classes
\begin{equation*}
u_{2p-1}(E,B): = [\mathrm{cs}_p(\nabla, \nabla^g)] \in H^{2p-1}(\mathcal{E}).
\end{equation*}
\end{definition}

\begin{remark}
\emph{If there exists on $B$ a fiberwise nondegenerate inner product $g$ that is invariant with respect to $\nabla$, then $\nabla^g = \nabla$
and $\mathrm{cs}_p(\nabla, \nabla) = 0$. Hence, $u_{2p-1}(E,B)=0$. Thus, the classes $u_{2p-1}(E,B)$ can be considered as
obstructions to the existence of an invariant, with respect to $\nabla$, fiberwise nondegenerate inner product $g$ on $B$.}
\end{remark}

\begin{remark}\label{remark - cs - flat}
\emph{We are placed in the context of Remark \ref{remark - cs - indepen of path}. If furthermore, $\nabla^0$ and $\nabla^1$ are flat connections,
then the curvature of $\nabla^{0,1}$ is}
\begin{equation*}
R^{\nabla^{0,1}} = t(t-1)(\nabla^0 - \nabla^1)\diamond (\nabla^0 - \nabla^1).
\end{equation*}
\emph{Therefore,}
\begin{eqnarray*}
\mathrm{CS}_p(\nabla^{0,1})=\mathrm{cs}_p(\nabla^0, \nabla^1) & = & \int_0^1 \tr\Big(\frac{d \nabla^{0,1}}{dt}\odot (R^{\nabla^{0,1}})^{p-1}\Big)dt \nonumber \\
   &=& p(p-1)! \int_0^1 \tr\Big(\frac{d \nabla^{0,1}}{dt}\diamond \underbrace{R^{\nabla^{0,1}} \diamond \ldots \diamond R^{\nabla^{0,1}}}_{(p-1)-times} \Big)dt \nonumber \\
   & = & p(p-1)! \int_0^1 \tr\Big((\nabla^1-\nabla^0)\diamond [t(t-1)(\nabla^0 - \nabla^1)\diamond (\nabla^0 - \nabla^1)]^{p-1}\Big)dt \nonumber \\
   & = & -  p(p-1)! \big(\int_0^1t^{p-1}(t-1)^{p-1}dt\big )\tr(\nabla^0 - \nabla^1)_{\diamond}^{2p-1}.
\end{eqnarray*}
\emph{But, $\displaystyle{\int_0^1t^{p-1}(t-1)^{p-1}dt = (-1)^{p-1}\frac{(p-1)!(p-1)!}{(2p-1)!}}$, so}
\begin{equation*}
\mathrm{CS}_p(\nabla^{0,1})=\mathrm{cs}_p(\nabla^0, \nabla^1) = (-1)^p\frac{p!(p-1)!}{(2p-1)!}\tr(\nabla^0 - \nabla^1)_{\odot}^{2p-1}
\end{equation*}
\emph{and the class $u_{2p-1}(E,B)$ is represented by the above multidifferential operator (see Remark \ref{remark - cs - indepen of path})}.\footnote{The expression $(\nabla^0 - \nabla^1)_{\diamond}^{2p-1}$
means that the power of $\nabla^0 - \nabla^1$ is calculated with respect to to product $\diamond$. Similarly,
for the expression $(\nabla^0 - \nabla^1)_{\odot}^{2p-1}$ the power of $\nabla^0 - \nabla^1$ is calculated with respect to $\odot$.}
\end{remark}

\vspace{1mm}
\noindent
We obtain the following result which is analogous to Proposition 2 in \cite{cr-fer-sec-cl} and to Theorem 3.2 in \cite{balcerzak}.

\begin{theorem}\label{theorem - classes metric connect}
Let $(E,B,\nabla, g, \nabla^g)$ as above. We suppose that $\nabla$ is flat (then $\nabla^g$ is also flat) and
we consider a metric nonlinear $E$-connection $\nabla_m$ on $B$ that is compatible with $g$, in the sense that $\nabla^g_m  = \nabla_m$.
\begin{enumerate}
\item
If $p$ is even, then $u_{2p-1}(E,B)=0$.
\item
If $p$ is odd, then $\mathrm{cs}_p(\nabla, \nabla_m)$ is $\partial_E$-closed and $u_{2p-1}(E,B) = 2 [\mathrm{cs}_p(\nabla, \nabla_m)]$.
\end{enumerate}
\end{theorem}
\begin{proof}
According to the identity \eqref{identity-chern-simons}, for the triple $(\nabla, \nabla^g, \nabla_m)$ we have
\begin{eqnarray*}
(p-1)\partial_E \mathrm{cs}_p(\nabla, \nabla^g, \nabla_m) & = & \mathrm{cs}_p(\nabla^g, \nabla_m) - \mathrm{cs}_p(\nabla, \nabla_m) + \mathrm{cs}_p(\nabla, \nabla^g) \nonumber \\
& \stackrel{\eqref{cs_p(nabla 0, nabla 1)}}{=} & (-1)^p \mathrm{cs}_p(\nabla, \nabla_m) - \mathrm{cs}_p(\nabla, \nabla_m) + \mathrm{cs}_p(\nabla, \nabla^g).
\end{eqnarray*}
So, if $p$ is even, $(p-1)\partial_E \mathrm{cs}_p(\nabla, \nabla^g, \nabla_m) = \mathrm{cs}_p(\nabla, \nabla^g)$, that is, $\mathrm{cs}_p(\nabla, \nabla^g)$
is exact and the characteristic class $u_{2p-1}(E,B)=0$. If $p$ is odd, then
\[\mathrm{cs}_p(\nabla, \nabla^g) = 2 \mathrm{cs}_p(\nabla, \nabla_m) + (p-1)\partial_E \mathrm{cs}_p(\nabla, \nabla^g, \nabla_m),\]
whence we conclude that $\mathrm{cs}_p(\nabla, \nabla_m)$ is $\partial_E$-closed and
\[u_{2p-1}(E,B) =  2 [\mathrm{cs}_p(\nabla, \nabla_m)]. \]
\end{proof}

\begin{examples}
\end{examples}
\noindent
\emph{Secondary characteristic classes of a Loday algebra.} Let $(E, \lcf \cdot, \cdot \rcf)$ be a
Loday algebra equipped with the adjoint representation $\nabla^{\mathrm{ad}}$,
$g$ an inner product on $E$, and $\nabla^{\mathrm{ad}, g}$ the conjugate representation of $\nabla^{\mathrm{ad}}$ with respect to $g$.
Without loss of generality, we suppose that $g$ is represented by the identity matrix $I_r$ ($r = \dim E$) and
from \eqref{nabla - g}, for any $e\in E$, it follows that
\begin{equation*}
\nabla^{\mathrm{ad}, g}_e = - g^{\flat^{-1}}\circ \nabla_e^{{\mathrm{ad}}^T} \circ g^\flat = - \nabla_e^{{\mathrm{ad}}^T},
\end{equation*}
where $\nabla_e^{{\mathrm{ad}}^T}$ is the transpose of the map $\nabla^{\mathrm{ad}}_e : E \to E$. Then
\begin{equation*}
\mathrm{cs}_1(\nabla^{\mathrm{ad}}, \nabla^{\mathrm{ad}, g}) = \tr(\nabla^{\mathrm{ad}, g} - \nabla^{\mathrm{ad}}) = \tr(- \nabla_e^{{\mathrm{ad}}^T}
- \nabla_e^{{\mathrm{ad}}}) = -2\tr(\nabla_e^{{\mathrm{ad}}})
\end{equation*}
and
\begin{equation*}
u_1 (E,\nabla^{\mathrm{ad}}) = - 2\, \mathrm{mod}(E).
\end{equation*}
For $p>1$, since $\nabla_e^{{\mathrm{ad}}^T}$ is also flat, the secondary characteristic classes are given, up to a constant factor (see Remark \ref{remark - cs - flat}), by
\begin{equation*}
u_{2p-1}(E, \nabla^{\mathrm{ad}}) = [\mathrm{cs}_p(\nabla^{\mathrm{ad}}, \nabla^{\mathrm{ad}, g})] = c[\tr(\nabla^{\mathrm{ad}})^{2p-1}].
\end{equation*}

\vspace{2mm}
\noindent
\emph{Secondary characteristic classes of $T^\ast M$.} Consider the representation $T^\ast M$ of a
Loday algebroid $E$ with respect of $\nabla$ given in \eqref{conn on T*M}.
Assume that $M$ is endowed with a Riemannian metric $g$ and denote by $g'$ the associated inner product on $T^\ast M$.
Then, for any $e\in \Gamma(E)$ and $\eta \in \Gamma(T^\ast M)$,
\begin{equation*}
\nabla^{g'}_e\eta = \mathcal{L}_{\rho(e)}\eta + (g^\flat \circ (\mathcal{L}_{\rho(e)}g')^{\sharp})\eta.
\end{equation*}
Since $\nabla$ is flat, $\nabla^{g'}$ is also flat (see proof of Proposition \ref{proposition path nabla}), so
\begin{equation*}
u_{2p-1}(E, T^\ast M) = [\mathrm{cs}_p(\nabla, \nabla^{g'}) \stackrel{Remark \;\ref{remark - cs - flat}}{=}c[\tr(- g^\flat \circ (\mathcal{L}_{\rho(e)}g')^{\sharp})_{\odot}^{2p-1}].
\end{equation*}
Moreover, a metric connection $\nabla_m$ on $T^\ast M$ compatible with $g'$ is the dual of the Levi-Civita
connection $\triangle$ of $g$, i.e. $\triangle^\ast = g^\flat \circ \triangle\circ g^{\flat^{-1}}$. Then,
according to Theorem \ref{theorem - classes metric connect}, $u_{2p-1}(E, T^\ast M) = 0$ (if $p$ is even) and
$u_{2p-1}(E, T^\ast M) = 2[\mathrm{cs}_p(\nabla, \triangle^\ast)]$ (if $p$ is odd).

\vspace{1mm}
\noindent
The above results on characteristic classes of Loday algebroids generalize, up to a constant factor, all the previously known results concerning Lie and Courant algebroids to the most general setting of Loday algebroids, whose bracket is not skew-symmetric, and nonlinear $E$-connections on vector bundles.
For Lie algebroids, the relevant references are the works of Fernandes and Crainic \cite{rui, cr-fer-sec-cl}, while for Courant algebroids, the reference is the work of Cueca and Mehta \cite{CM2021}. In \cite{balcerzak}, Balcerzak represents the characteristic classes for $\mathbb{R}$-linear connections of Lie algebroids.

\section{Appendices}\label{section - appendices}
\subsection{Differential operators, jet bundles and jet modules}\label{app-Jet bundles}
Let $\mathbb{K}$ be a commutative ring with unit, $\mathcal{A}$ a commutative $\mathbb{K}$-algebra, $\mathcal{E}$ and $\mathcal{F}$ two
faithful $\mathcal{A}$-bimodules\footnote{An $\mathcal{A}$-module $\mathcal{E}$ is called \emph{faithful} if the set of elements $f\in \mathcal{A}$
such that, for all $e\in \mathcal{E}$, $fe=0$ reduces to $\{0\}$.}. For any $\mathbb{K}$-linear operator $D : \mathcal{E} \to \mathcal{F}$ and any $f\in \mathcal{A}$, we denote by
\begin{itemize}
\item[-]
$m_f^{\mathcal{E}} : \mathcal{E}\to \mathcal{E}$ the multiplication by $f$ in $\mathcal{E}$, $e \mapsto m_f^{\mathcal{E}}(e):= fe$;
\item[-]
$\sigma^D(f) : \mathcal{E} \to \mathcal{F}$ \emph{the symbol of $D$} given by
\[\sigma^D(f) = D \circ m_f^{\mathcal{E}} - m_f^{\mathcal{F}}\circ D.\]
\end{itemize}

\begin{definition}
We say that an element $D \in \mathrm{Hom}_{\mathbb{K}}(\mathcal{E},\mathcal{F})$ is a \emph{differential operator}
over $\mathcal{A}$ \emph{of order $\leq k$}, $k\in \N$, if $\sigma^D(f_1)\circ \ldots \circ \sigma^D(f_{k+1}) = 0$ for any set of $k+1$
elements $f_1, \ldots, f_{k+1}$ of $\mathcal{A}$.
\end{definition}
Denote by $\mathrm{Diff}_k(\mathcal{E}, \mathcal{F})$ the set of differential operators of order
$\leq k$ from $\mathcal{E}$ to $\mathcal{F}$ which is an $\mathcal{A}$-module.

\begin{definition}
A \emph{derivative endomorphism} $D$ of $\mathcal{E}$ is a first order differential operator on $\mathcal{E}$ whose
symbol is a scalar endomorphism of $\mathcal{E}$. That is, it is characterized by the following condition: for any $f\in \mathcal{A}$, there exists $f' \in \mathcal{A}$ such that
\[\sigma^D(f) = m_{f'}^{\mathcal{E}}\circ \mathrm{Id}_{\mathcal{E}}.\]
\end{definition}
As it is proved in \cite{grab-QD}, there exists a derivation $X_D$ of $\mathcal{A}$ such that $f'=X_D(f)$. The set $\mathrm{Der}(\mathcal{E})$
of derivative endomorphisms of $\mathcal{E}$ is an $\mathcal{A}$-submodule of $\mathrm{Diff}_1(\mathcal{E})$, the $\mathcal{A}$-module of differential operators of order
at most $1$ from $\mathcal{E}$ to $\mathcal{E}$. Furthermore, $(\mathrm{Der}(\mathcal{E}), \, [\cdot,\cdot]_{com})$
is a Lie algebra over $\mathbb{K}$ with respect to the commutator bracket and, for any $D_1, D_2 \in \mathrm{Der}(\mathcal{E})$,
\[X_{[D_1, D_2]_{com}} = [X_{D_1}, X_{D_2}].\]

\vspace{2mm}
\noindent
Let $\mathcal{E} = \Gamma(E)$ be the projective and finitely generated $\C$-module of smooth sections of a
vector bundle $E$ of rank $r$ over a smooth manifold $M$ of $\dim M =n$, $\mu_x = \{f\in \C \, / \, f(x)=0\}$ the
maximal ideal of smooth functions on $M$ that vanish at the point $x\in M$, and $\mu_x^{k+1}=\{f_1\cdot \ldots \cdot f_{k+1} \,/\, f_i\in \mu_x\}$
the $(k+1)^{st}$-power of $\mu_x$. The set $\mu_x^{k+1}\cdot \mathcal{E}$ is a submodule of
$\mathcal{E}$. Let $J^k_x \mathcal{E} : = \mathcal{E}/\mu_x^{k+1} \mathcal{E}$ be the quotient module,
which is also a vector space. The image of an element $e\in \mathcal{E}$ under the natural projection will be denoted by $[e]_x^k$.
The space $J^k\mathcal{E} = \bigsqcup_{x\in M}J^k_x \mathcal{E}$ is a vector bundle over $M$ \cite[14.20 Jet bundles]{jet-nestr}, called \emph{the bundle of jets
of order $k$} or \emph{$k$-jets} of the vector bundle $E$. The module of smooth sections $\Gamma(J^k\mathcal{E})$ of $J^k\mathcal{E}$ is
called the \emph{module of $k$-jets of $E$} and denoted by $\mathcal{J}^k(\mathcal{E})$.
Any element $e\in \mathcal{E}$ defines a section $j_k(e)$ of $J^k\mathcal{E}$ by setting $j_k(e)(x):=[e]_x^k$.
In a local frame $(e_1,\ldots, e_r)$ of smooth sections of $E$ on $U$, $(U,x)$ being a chart on $M$ and $r=\mathrm{rank}E$, any section $e\in \Gamma(E\vert_U)$
is written as $e = f^1e_1 + \ldots + f^re_r$, where $f^i\in C^\infty(U, \R)$. Then, in the corresponding coordinate system on $J^k\mathcal{E}$, the section $j_k(e)$
is represented by the vector
\[ j_k(e) = (f^m, \frac{\partial f^m}{\partial x^{i_1}}, \frac{\partial^2 f^m}{\partial x^{i_1} \partial x^{i_2}}, \ldots, \frac{\partial^k f^m}{\partial x^{i_1} \ldots
\partial x^{i_k}}),\]
with $m=1,\ldots,r$, $i_s = 1,\ldots, n$, $s=1,\ldots,k$.
In general, in local coordinates, $j_k(e)$ is given by $(f^m, \displaystyle{\frac{\partial^{|I|}f^m}{\partial x^{I}}})_{m=1,\ldots,r}$,
where $I$ is a multi-index of length $|I|$, $1\leq |I|\leq k$. The coordinate expression of $j_k(e)$ shows that,
as section of $J^k\mathcal{E}$, is smooth, and that the $\R$-linear map
\[j_k : \mathcal{E} \to \mathcal{J}^k(\mathcal{E}), \quad \quad e \mapsto j_k(e),\]
is a differential operator of order $\leq k$.

\vspace{1mm}
\noindent
According to Proposition 14.24 in \cite{jet-nestr}, since $\mathcal{E}$ is a projective $\C$-module,
there exists a finite set $\{e_1, \ldots, e_m\}$ of elements of $\mathcal{E}$ such that the
corresponding $k$-jets $j_k(e_1), \ldots, j_k(e_m)$ generate the $\C$-module $\mathcal{J}^k(\mathcal{E})$.

\vspace{1mm}
\noindent
Also, according to Theorem 14.25 in \cite{jet-nestr}, for any geometrical $\C$-module $\mathcal{F}$, the correspondence
\[\mathrm{Hom}_{_{\C}}(\mathcal{J}^k(\mathcal{E}), \mathcal{F})\ni h \mapsto h\circ j_k \in \mathrm{Diff}_k(\mathcal{E}, \mathcal{F})\]
defines a natural isomorphism of $\C$-modules
\begin{equation}\label{isom-diff-hom-1}
\mathrm{Hom}_{_{\C}}(\mathcal{J}^k(\mathcal{E}), \mathcal{F}) \cong \mathrm{Diff}_k(\mathcal{E}, \mathcal{F}).
\end{equation}
Thus,
\begin{equation}\label{isom-diff-hom-2}
\mathrm{Diff}_k(\mathcal{E}, \mathcal{F}) \cong (\mathcal{J}^k(\mathcal{E}))^\ast \otimes \mathcal{F},
\end{equation}
where $(\mathcal{J}^k(\mathcal{E}))^\ast$ is the dual $\C$-module of $\mathcal{J}^k(\mathcal{E})$ (see, \cite{bourbaki}).
Since the module of smooth sections of a vector bundle is a geometric module, we can apply the above result for $\mathcal{F} = \C$,
which is the module of smooth section of the vector bundle $M\times \R \to M$.
Hence, if $(\varepsilon^1, \ldots, \varepsilon^m)$ is the dual family of generators of $(\mathcal{J}^k(\mathcal{E}))^\ast$, with respect to $(j_k(e_1), \ldots, j_k(e_m))$, any element $D\in \mathrm{Diff}_k(\mathcal{E}, \C)$ is written locally as
\[ D = h_1\varepsilon^1 + \ldots h_m\varepsilon^m, \quad \quad h_j\in C^\infty(U,\R), \quad j=1, \ldots, m.\]

\subsection{Multidifferential operators}\label{app-multidiff}
\begin{definition}\cite{grab-k-pon}
Let $(\mathcal{E}_1, \ldots, \mathcal{E}_p)$ be a $p$-tuple of $\mathcal{A}$-(faithful) bimodules and $\mathcal{F}$ an $\mathcal{A}$-(faithful) bimodule.
A \emph{multidifferential operator} or \emph{$p$-differential operator} from $\mathcal{E}_1 \times \ldots \times \mathcal{E}_p$ to $\mathcal{F}$
is a $\mathbb{K}$-multilinear operator $D : \mathcal{E}_1 \times \ldots \times \mathcal{E}_p \to \mathcal{F}$ which is a differential operator with respect to each argument.

\vspace{1mm}
\noindent
We say that $D$ is \emph{a differential operator of order $\leq k$}, $k\in \N$, \emph{with respect to the $i^{th}$-argument} if,
for all $e_j \in \mathcal{E}_j$, $j\neq i$,
\[D(e_1, \ldots, e_{i-1}, \cdot, e_{i+1}, \ldots, e_p) : \mathcal{E}_i \to \mathcal{F}\]
is a differential operator of order $\leq k$. That means that the $i$-symbol $\sigma_i^D(f)$ of $D$\,\footnote{Term which is introduced in \cite{dubviol-mas}.}
given, for any $f\in \mathcal{A}$ and $(e_1, \ldots, e_i, \ldots, e_p) \in \mathcal{E}_1 \times \ldots \times \mathcal{E}_i \times \ldots \times \mathcal{E}_p$, by
\begin{eqnarray}\label{i-symbol}
\lefteqn{\sigma_i^D(f)(e_1, \ldots,e_{i-1}, e_i,e_{i+1}, \ldots, e_p): = }\nonumber \\
& & D(e_1, \ldots, e_{i-1}, fe_i, e_{i+1}, \ldots, e_p) - fD(e_1, \ldots, e_{i-1}, e_i, e_{i+1}, \ldots, e_p),
\end{eqnarray}
verifies the equation $\underbrace{\sigma_i^D(f)\circ \ldots \circ \sigma_i^D(f)}_{k+1 - times}=0$. Note that the operators $\sigma_i^D(f)$ and $\sigma_j^D(f)$ commute.

\vspace{1mm}
\noindent
We say that the operator $D$ is \emph{a multidifferential operator of order $\leq k$}, $k\in \N$, if it is of order $\leq k$ with respect to each entry, separately.

\vspace{1mm}
\noindent
We say that $D$ is \emph{a multidifferential operator of total order $\leq k$}, $k\in \N$, if, for any multi-index $(i_1,\ldots,i_{k+1})$ of
$\{1,\ldots, p\}$ of length $k+1$ and any $f_1,\ldots, f_{k+1} \in \mathcal{A}$, we have
\[\sigma_{i_1}^D(f_1) \circ \ldots \circ \sigma_{i_{k+1}}^D(f_{k+1}) = 0.\]
\end{definition}
As observed in \cite{grab-k-pon}, any multidifferential operator of total order $\leq k$ is a multidifferential operator
of order $\leq k$. Also, any $\mathbb{K}$-linear $p$-differential operator $D$ of order $\leq k$ is of total order $\leq pk$.
In fact, the assumption that $D$ is of order $\leq k$ means that, for any $i=1, \ldots, p$ and $f \in \mathcal{A}$,
$\underbrace{\sigma_i^D(f)\circ \ldots \circ \sigma_i^D(f)}_{k+1 - times}=0$ and
$\underbrace{\sigma_i^D(f)\circ \ldots \circ \sigma_i^D(f)}_{m - times}\neq 0$ for any $m\leq k$. On the other hand, any composition of type
\[\sigma^D_1(f)^{m_1}\circ \ldots \circ \sigma^D_p(f)^{m_p}\]
is identically zero if $m_i >k$, for $i=1,\ldots,p$. So, $m_1+\ldots + m_p \geq pk + 1$.

\vspace{1mm}
\noindent
By induction on $p\in \N$ and taking into account that the operators $\sigma_i^D(f)$ and $\sigma_j^D(f)$ commute,
we establish the following formula that generalises \eqref{i-symbol}:
\begin{eqnarray*}\label{multi-symbol}
\lefteqn{D(f_1e_1, \ldots, f_pe_p)  = }\nonumber \\
& & \sum_{k=0}^p\sum_{\tau \in Sh (k,\, p-k)}f_{\tau(1)}\ldots f_{\tau(k)}\sigma^D_{\tau(k+1)}(f_{\tau(k+1)})\circ \ldots \circ \sigma^D_{\tau(p)}(f_{\tau(p)})(e_1,\ldots,e_p),
\end{eqnarray*}
where $Sh(k,p-k)$ is the set of $(k,p-k)$-shuffle permutations of $(1,\ldots,p)$.

\vspace{2mm}
\noindent
For the purposes of this paper, we focus on the case where
$\mathcal{A}$ is the $\R$-algebra $\C$ of smooth real functions
on a smooth manifold $M$ and we consider the standard geometric model of (multi)differential operators acting on the $\C$-module $\mathcal{E}=\Gamma(E)$ of
smooth sections of a vector bundle $E$ over $M$, which is a projective and finitely generated $\C$-(faithful) bimodule.
As in the subsection \ref{subsection Loday cohomology}, we denote by $\mathcal{D}^p(\mathcal{E})$ the module of $p$-differential operators on $\mathcal{E}$ with values in $\C$.
In particular, we write $\mathcal{D}_{(k_1,\ldots,k_p)}^p(\mathcal{E})$ to emphasize that the differential
operators of this space is of order $k_i$ with respect to the $i$-entry, $i=1,\ldots,p$. Then, taking into account the isomorphisms
\eqref{isom-diff-hom-1} and \eqref{isom-diff-hom-2}, we can write
\begin{eqnarray}\label{tensor hom}
\lefteqn{\mathcal{D}^p_{(k_1,\ldots,k_p)}(\mathcal{E}) =
\mathrm{Diff}_{k_1}(\mathcal{E}, \mathrm{Diff}_{k_2}(\mathcal{E}, \ldots \mathrm{Diff}_{k_p}(\mathcal{E}, \C)\ldots ))} \nonumber \\
& \stackrel{\eqref{isom-diff-hom-1}}{\cong} & \mathrm{Hom}(\mathcal{J}^{k_1}(\mathcal{E}), \mathrm{Hom}(\mathcal{J}^{k_2}(\mathcal{E}), \ldots
\mathrm{Hom}(\mathcal{J}^{k_p}(\mathcal{E}), \C)\ldots )) \nonumber \\
&\cong & \mathrm{Hom}(\mathcal{J}^{k_1}(\mathcal{E})\otimes \ldots \otimes \mathcal{J}^{k_p}(\mathcal{E}), \C) \nonumber \\
&\stackrel{\eqref{isom-diff-hom-2}}{\cong} & (\mathcal{J}^{k_1}(\mathcal{E}))^\ast \otimes \ldots \otimes (\mathcal{J}^{k_p}(\mathcal{E}))^\ast \otimes \C.
\end{eqnarray}
Thus, if $(\varepsilon_i^1, \ldots, \varepsilon_i^{m_i})$ is a family of
generators of the $i$-term $(\mathcal{J}^{k_i}(\mathcal{E}))^\ast$, $i=1,\ldots,p$, of the tensor product \eqref{tensor hom}, then the family
$(\varepsilon_1^{j_1}\otimes \ldots \otimes \varepsilon_p^{j_p})_{(j_i = 1, \ldots, m_i, \, i=1,\ldots,p)}$ generates the space
$\mathcal{D}^p_{(k_1,\ldots,k_p)}(\mathcal{E})$ and any element $D$ of this space can be written, locally, as
\begin{equation}\label{model - D}
D = \sum_{j_1,\ldots,j_p}h_{j_1\ldots j_p}\varepsilon_1^{j_1}\otimes \ldots \otimes \varepsilon_p^{j_p}, \quad \quad h_{j_1\ldots j_p} \in C^\infty(U,\R).
\end{equation}

\vspace{5mm}
\noindent
\emph{Adresses}

\vspace{3mm}

\noindent
\emph{Raquel Caseiro}, CMUC, Department of Mathematics, University of Coimbra, Apartado 3008, 3001-501 Coimbra, Portugal

\vspace{1mm}
\noindent
\emph{raquel@mat.uc.pt}

\vspace{3mm}

\noindent
\emph{Fani Petalidou}, Department of Mathematics, Aristotle University of Thessaloniki, 54124 Thessaloniki, Greece

\vspace{1mm}
\noindent
\emph{petalido@math.auth.gr}


\begin{thebibliography}{60}

\bibitem{balcerzak}{Balcerzak, B.: Chern–Simons forms for $\R$-linear connections on Lie algebroids. Internat. J. Math. 29, 1850094, 13 pp (2018)}

\bibitem{bat-pet}{Batakidis, P., Petalidou, F.: Courant-Dorfman algebras of differential operators and Dorfman connections of Courant algebroids.
J. Geom. Phys. 199, Paper No. 105142 (2024)}

\bibitem{bi-sheng}{Bi, Y.H., Sheng, Y.H.: On higher analogues of Courant algebroids. Sci. China Math. 54, 437--447 (2011)}

\bibitem{bourbaki}{Bourbaki, N.: Elements of mathematics. Algebra, Part I, Chapters 1-3. Hermann, Paris; Addison-Wesley Publishing Co., Reading, MA, (1974)}


\bibitem{Cantr-Lan}{Cantrijn, F., Langerock, B.: Generalised connections over a vector bundle map. Diffe\-rential Geometry and its Applications 18, 295--317 (2003)}



\bibitem{chern-sim}{Chern, S.-S., Simons, J.: Characteristic forms and geometric invariants, Ann. of Math. 99, 48--69 (1974)}

\bibitem{cr-fer-sec-cl}{Crainic, M., Fernandes, R. L.: Secondary characteristic classes of Lie algebroids.
In \emph{Quantum field theory and noncommutative geometry}, Lecture Notes in Phys. 662, Springer-Verlag, Berlin, 157--176 (2005)}

\bibitem{dubviol-mas}{Dubois-Violette, M., Masson, Th.: On the First-Order Operators in Bimodules. Lett. Math. Phys. 37, 467--474 (1996)}

\bibitem{eil-ml}{Eilenberg, S., Mac Lane, S.: On the groups $H(\Pi, n). I$. Ann. of Math. 58, 55--106 (1953)}


\bibitem{CM2021}{Cueca, M., Mehta, R.A.: Courant Cohomology, Cartan Calculus, Connections, Curvature, Characteristic Classes. Commun. Math. Phys. 381, 1091--1113 (2021)}

\bibitem{evens-lu-weinstein}{Evens, S., Lu, J.-H., Weinstein, A.: Transverse measures, the modular class and a cohomology pairing for Lie algebroids.
Quart. J. Math. Oxford Ser. (2) 50, 417--436 (1999)}


\bibitem{rui}{Fernandes, R. L.: Lie algebroids, Holonomy and Characteristic Classes. Adv. Math. 170, 119–179 (2002)}

\bibitem{grab-QD}{Grabowski, J.: Quasi-derivations and $QD$-algebroids. Rep. Math. Phys. 52, 445--451 (2003)}


\bibitem{grab-k-pon}{Grabowski, J., Khudaverdyan, D., Poncin, N.: The supergeometry of Loday algebroids. J. Geom. Mech. 5, 185--213 (2013)}

\bibitem{hig-mck}{Higgins, P. J., Mackenzie, K.: Algebraic constructions in the category of Lie algebroids. Journal of Algebra 129, 194--230 (1990)}

\bibitem{husemoller}{Husemoller, D.: Fibre Bundles. Third edition. Grad. Texts in Math. 20, Springer-Verlag, New York, (1994)}

\bibitem{ILMP1999}{Ib\'a\~nez, R.,  Le\'on, M.,  Marrero, J. C., Padr\'on, E.: Leibniz algebroid associated with a Nambu-Poisson structure. J.  Phys A: Math. Gen.  32 (46), 8129, (1999)}



\bibitem{jurco-Vysoky}{Jur\v{c}o, B., Vysok\'y, J.: Leibniz algebroids, generalized Bismut connections and Einstein-Hilbert actions.
J. Geom. Phys. 97, 25--33 (2015)}

\bibitem{keller-waldmann}{Keller, F., Waldmann, St.: Deformation theory of Courant algebroids via the Rothstein algebra. J. Pure Appl. Algebra 219, 3391--3426 (2015)}

\bibitem{kob-nom}{Kobayashi, S., Nomizu, K.: Foundations of differential geometry. Vol. II, Wiley Classics Lib., Wiley-Intersci. Publ. John Wiley and Sons, Inc., New York, (1996)}

\bibitem{kos-cam-wein}{Kosmann-Schwarzbach, Y., Laurent-Gengoux, C., Weinstein, A.: Modular classes of Lie algebroid morphisms. Transform. Groups 13, 727--255 (2008)}



\bibitem{li-bland-meinr}{Li-Bland, D., Meinrenken, E.: Courant algebroids and Poisson geometry. Int. Math. Res. Not. IMRN, no. 11, 2106–2145 (2009)}

\bibitem{lod}{Loday, J. L.: Une version non commutative des alg\`ebres de Lie, les alg\`ebres de Leibniz. Enseign. Math. 39, 269--293 (1993)}

\bibitem{lod-pir}{Loday, J. L., Pirashvili, T.: Universal enveloping algebras of Leibniz algebras and (co)homology. Math. Annalen 296, 139--158 (1993)}

\bibitem{mang-sard}{Mangiarotti, L., Sardanashvily, G.: Connections in classical and quantum field theory. World Scientific Publishing Co., Inc., River Edge, NJ (2000)}

\bibitem{marle-coimbra}{Marle, C. M.: Differential calculus on a Lie algebroid and Poisson manifolds. The J. A. Pereira da Silva Birthday Schrift, Textos Mat. Ser. B 32, Departamento de matematica da Universidade de Coimbra, Portugal, 83--149 (2002)}

\bibitem{mck}{Mackenzie, K.C.H.: General theory of Lie groupoids and Lie algebroids. London Math. Soc. Lecture Note Series 213, Cambridge University Press, Cambridge (2005)}

\bibitem{milnor-stasheff}{Milnor, J. W., Stasheff, J. D.: Characteristic classes.
Ann. of Math. Stud. 76, Princeton University Press, Princeton, NJ; University of Tokyo Press, Tokyo (1974)}

\bibitem{jet-nestr}{Nestruev, J.: Smooth Manifolds and Observables. Second edition. Grad. Texts in Math. 220, Springer, Cham. (2020)}


\bibitem{popescu}{Popescu, P.: On generalized algebroids. New developments in differential geometry, Budapest 1996, Kluwer Academic Publishers, Dordrecht, 329--342 (1999)}






\bibitem{roy}{Roytenberg, D.: On the structure of graded symplectic supermanifolds and Courant
algebroids. In: Voronov, T. (ed.) Quantization, Poisson brackets and beyond (Manchester, 2001), pp. 169--185. Contemp. Math., 315, Amer. Math. Soc., Providence, RI, (2002)}





\bibitem{severa}{\v{S}evera, P.: Letters to Alan Weinstein about Courant algebroids. https://arxiv.org/abs/1707.00265, (2017)}

\bibitem{stienon-xu}{Sti\'enon, M., Xu, P.: Modular classes of Loday algebroids. C. R. Acad. Sci. Paris, Ser. I 346, 193--198, (2008)}


\bibitem{vaintrob}{Vaintrob, A. Yu.: Lie algebroids and homological vector fields. Russian Math. Surveys 52, 428--429 (1997)}


\bibitem{vai-nambu}{Vaisman, I.: A survey on Nambu-Poisson brackets. Acta Math. Univ. Comenian. (N.S.) 68, 213--241 (1999)}





\end{thebibliography}
\end{document}